\documentclass[10pt]{amsart}
\usepackage{amsmath,amssymb,amsthm,amsfonts,graphicx,color}
\usepackage{amssymb}
\usepackage{amsfonts,amsthm,amsmath,latexsym,bm,graphics}

\makeatother
\numberwithin{equation}{section}

\newtheorem{lemma}{Lemma}[section]
\newtheorem{coro}[lemma]{Corollary}
\newtheorem{definition}[lemma]{Definition}
\newtheorem{teo}[lemma]{Theorem}
\newtheorem{proposition}[lemma]{Proposition}
\newtheorem{corollary}[lemma]{Corollary}
\newtheorem{remark}[lemma]{Remark}

\newcommand{\R}{{\mathbb R}}

\newcommand{\be}{\begin{equation}}
\newcommand{\ben}{\begin{equation*}}
\newcommand{\ee}{\end{equation}}
\newcommand{\een}{\end{equation*}}
\newcommand{\BL}{\begin{lemma}}
\newcommand{\EL}{\end{lemma}}
\newcommand{\BT}{\begin{theorem}}
\newcommand{\ET}{\end{theorem}}
\newcommand{\BP}{\begin{proposition}}
\newcommand{\EP}{\end{proposition}}
\newcommand{\BC}{\begin{corollary}}
\newcommand{\EC}{\end{corollary}}
\def\bs{\begin{split}}
\def\es{\end{split}}

\DeclareMathOperator{\dist}{dist}

\begin{document}

\title[A gluing approach for the fractional Yamabe problem with isolated singularities]
{A gluing approach for the fractional Yamabe problem with isolated singularities}

\author[W. Ao]{Weiwei Ao}

\address{Weiwei Ao
\hfill\break\indent
Wuhan University
\hfill\break\indent
Department of Mathematics and Statistics, Wuhan, 430072, PR China}
\email{wwao@whu.edu.cn}

\author[A. DelaTorre]{Azahara DelaTorre}

\address{Azahara DelaTorre
\hfill\break\indent
Universitat Polit\`ecnica de Catalunya,
\hfill\break\indent
ETSEIB-MA1, Av. Diagonal 647, 08028 Barcelona, Spain}
\email{azahara.de.la.torre@upc.edu}

\author[M.d.M. Gonz\'alez]{Mar\'ia del Mar Gonz\'alez}

\address{Mar\'ia del Mar Gonz\'alez
\hfill\break\indent
Universidad Aut\'onoma de Madrid
\hfill\break\indent
Departamento de Matem\'aticas, Campus de Cantoblanco, 28049 Madrid, Spain}
\email{mariamar.gonzalezn@uam.es}

\author[J. Wei]{Juncheng Wei}
\address{Juncheng. Wei
\hfill\break\indent
University of British Columbia
\hfill\break\indent
 Department of Mathematics, Vancouver, BC V6T1Z2, Canada} \email{jcwei@math.ubc.ca}

\thanks{A. DelaTorre is Supported by MINECO grants MTM2014-52402-C3-1-P and MTM2017-85757-P, and  the FPI-2012 fellowship, and is part of the Catalan research group 2014SGR1083. M.d.M. Gonz\'alez is supported by MINECO grants  MTM2014-52402-C3-1-P and MTM2017-85757-P, the Fundaci\'on BBVA grant for Investigadores y Creadores Culturales 2016, and is part of the Barcelona Graduate School of Math and the Catalan research group 2014SGR1083. She also would like to acknowledge the NSF grant DMS-1440140 while she was in residence at the Mathematical Sciences Research Institute in Berkeley, CA, during Spring 2016.  J. Wei is partially supported by NSERC of Canada.}

\begin{abstract}

We construct solutions for the fractional Yamabe problem that are singular at a prescribed number of isolated points. This seems to be the first time that a gluing method is successfully applied to a non-local problem in order to construct singular solutions. There are two main steps in the proof: to construct an approximate solution by gluing half bubble towers at each singular point, and then an infinite-dimensional Lyapunov-Schmidt reduction method, that reduces the problem to an (infinite dimensional) Toda type system. The main technical part is the estimate of the interactions between different bubbles in the bubble towers.
\end{abstract}

\date{}\maketitle
%\keywords{D}

%\abbreviations{}

\centerline{AMS subject classification:  35J61, 35R11, 53A30}

\section{Introduction}\label{sec1}

In this paper, we consider the problem of finding solutions for the fractional Yamabe problem in $\mathbb R^n$, $n> 2\gamma$ for $\gamma\in(0,1)$ with isolated singularities at a prescribed finite number of points. This is, to find positive solutions for the equation
\begin{equation}\label{eq101}
\left\{\begin{split}
&(-\Delta_{\mathbb R^n})^\gamma u=c_{n,\gamma}u^{\beta} \,\,\mbox{ in }\mathbb R^n\backslash \Sigma,\\
&u\to +\infty\,\, \mbox{ as }x\to \Sigma,
\end{split}
\right.
\end{equation}
where $\Sigma=\{p_1,\cdots,p_k\}$ for $k\geq 2$, and
 $$\beta=\frac{n+2\gamma}{n-2\gamma}$$
is the critical exponent in dimension $n$.
 Remark that we are using the notation $t^\beta$ to denote the power nonlinearity $|t|^{\beta-1}t$, but this does not constitute any abuse of notation since  any solution must be positive thanks to the maximum principle. In addition, $c_{n,\gamma}>0$ is a normalization constant and can be chosen arbitrarily.

 Problem \eqref{eq101} can be formulated in geometric terms: given the Euclidean metric $|dx|^2$ on $\mathbb R^n$, we are looking for a conformal metric \begin{equation*}g_u=u^{\frac{4}{n-2\gamma}}|dx|^2,\ u>0,\end{equation*} with positive constant fractional curvature $Q^{g_u}_\gamma\equiv c_{n,\gamma}>0$. This is known as the fractional Yamabe problem (in the positive curvature case), and smooth solutions have been considered in \cite{MarChang,CaseChang,MarQing,Gonzalez-Wang,kmw1,Fang-Gonzalez,kmw} for instance. We remark that the nonlinearity in the right hand side of the equation is critical for the Sobolev embedding, a common feature of  Yamabe-type problems.

Instead, one could look at the singular version of the problem, in which the metric blows up at a prescribed set $\Sigma\subset\mathbb R^n$. Here the sign of $Q_\gamma$ is related to the size of the singular set  $\Sigma$. For instance, when $\Sigma$ is a smooth submanifold, \cite{Gonzalez-Mazzeo-Sire} shows that the positivity of fractional curvature imposes some geometric and topological restrictions, while \cite{Zhang} considers very general singular sets in the case $\gamma\in(1,2)$, with the additional assumption of positive scalar curvature. See also \cite{Jin-Queiroz-Sire-Xiong} for some capacitary arguments on the local behavior of singularities.

But all these results give necessary conditions for the existence of such metrics.
On the contrary, the question of sufficiency is expected to have only partial answers, requiring that $\Sigma$ has a very particular structure. Here we initiate the study of this program, looking at the singular Yamabe problem  with prescribed isolated singularities at the points $\{p_1,\ldots,p_k\}$, $k\geq 2$.

Thus our main theorem is:

\begin{teo}\label{main-theorem}
Fixed any configuration $\Sigma=\{p_1,\cdots,p_k\}$ of $k$ different points in $\mathbb R^n$, $k\geq 2$,
there exists a  smooth, positive solution to \eqref{eq101}.
\end{teo}

As a corollary, we also obtain existence of conformal metrics on the unit sphere $\mathbb S^n$ of constant fractional curvature with a finite number of isolated singularities. Note that our results will imply that this metric is complete.

In the case of an isolated singularity, it is shown in  \cite{CaffarelliJinSireXiong} that non-removable singularities for the problem
\begin{equation}\label{equation0}\left\{\begin{split}
&(-\Delta_{\mathbb R^n})^\gamma u=u^{\beta} \text{ in } \mathbb R^n\backslash\{0\},
\\&u\to +\infty \,\text{ as }\,x\to 0, \quad u>0,
\end{split}\right.\end{equation}
 must have the asymptotic behavior
 \begin{equation*}\label{asymptotics}
c_1r^{-\tfrac{n-2\gamma}{2}}\leq u(x)\leq c_2r^{-\tfrac{n-2\gamma}{2}},\quad\text{as}\quad r\to 0,
\end{equation*}
where $c_1$, $c_2$ are positive constants and $r=|x|$.

The geometric interpretation of \eqref{equation0} was considered in \cite{paper1}. Indeed, it corresponds to  the fractional Yamabe problem in a cylinder, which motivates the substitution \eqref{wv} below. In the paper \cite{dpgw} the authors show, using a variational approach, the existence of \emph{Delaunay} type solutions for \eqref{equation0}, i.e, solutions  of the form
\begin{equation}\label{wv}
u_L(r)=r^{-\frac{n-2\gamma}{2}}v_L(-\log r)\text{ on } \mathbb R^n\setminus\{0\},
\end{equation}
 for some smooth function $v_L$ that is periodic in the variable $t=-\log r$, for any period $L\geq L_0$. $L_0$ is known as the minimal period and has been completely characterized.
 
  This type of solutions for local problems have been known for a long time. For instance, for constant mean curvature surfaces, this construction is an old one and goes back to \cite{Delaunay}, while for the scalar curvature case $\gamma=1$, which corresponds to the classical Yamabe problem, a good reference is \cite{Schoen:notas}.

In addition, Delaunay solutions are useful in gluing problems, since they model isolated singularities: we cite, for instance, \cite{Mazzeo-Pacard:Delaunay-ends,Mazzeo-Pacard-Pollack,Schoen:isolated-singularities} for the construction of constant mean curvature surfaces with Delaunay ends, or \cite{mp, Mazzeo-Pollack-Uhlenbeck} for solutions to \eqref{eq101} in the local case $\gamma=1$. However, these classical constructions exploit the local nature of the problem and, above all, the fact that \eqref{equation0} reduces to a standard second ODE in the radial case. There the space of solutions of this ODE can be explicitly written in terms of two given parameters using phase-plane analysis,  which is not the case for a non-local equation.

Here we are able to use the gluing method for the non-local problem \eqref{eq101}.  The first difficulty is obvious: one needs to make sure that the errors created by the localization procedure are not propagated by the non-locality of the problem but, instead, they can be handled through careful estimates.

Nevertheless, the main obstacle we find is the lack of standard ODE methods, which are not valid any longer for a non-local problem. For instance, as we have mentioned above, the starting point in the classical case (\cite{mp,Mazzeo-Pacard:Delaunay-ends}) is the consideration of Delaunay metrics in order to construct an approximate solution to the original problem.  In our case, even though a Delaunay solution is our basic model, we construct a bubble tower at each singular point. These are known as {\em half-Dancer} solutions, and they  converge to a half-Delaunay.
In addition, half-Dancer solutions decay fast at infinity, thus they are appropriate in a gluing construction for a non-local problem.

The advantage of this approach is that one is able to perturb each bubble in the
tower separately.  Gluing a bubble tower at each singular point allows to construct a suitable approximate solution for \eqref{eq101} with an infinite number of parameters to be chosen. Note that the linearization at this approximate solution is not injective due to the presence of an  infinite dimensional kernel, so we use an infinite dimensional Lyapunov-Schimdt reduction procedure.  This approach, in some spirit,
is in close connection with Kapouleas' CMC construction \cite{Kapouleas}, and later adapted
by Malchiodi in \cite{m}, where he constructs new entire solutions for a semilinear equation with subcritical exponent, different from the spike solutions that were known for a long time. Malchiodi's new solutions do not tend to zero at infinity, but decay to zero away from three half lines; his method is to construct a half-Dancer solution along each half-line.

As a consequence, in order to solve the original problem from the perturbed one, we need to solve an infinite dimensional system of Toda type, which comes from studying the interactions between the different bubbles in the tower. While the strongest interactions lead to some compatibility conditions,  the remaining interactions can be made small, and are handled through a fixed point argument.

These compatibility conditions do not impose any restrictions on the location of the singularity points $p_1,\ldots,p_k$, but only on the Delaunay parameter (the neck size) at each point. We also remark that the first compatibility condition is analogous to that of the local case $\gamma=1$ of \cite{mp}, this is due to the strong influence of the underlying geometry, while the rest of the configuration depends on the Toda type system. On the other hand, in the local setting a similar procedure to remove the resonances of the linearized problem was considered in \cite{amw} and the references therein. However, in their case the Toda type system is finite dimensional.

\medskip

More precisely,  we use the gluing method and Lypunov-Schmidt reduction method in this paper. First, we find a good approximate solution $\bar{u}$, defined in \eqref{eq301}, which roughly speaking is a perturbation of summation of $k$ half-Delaunay solutions with singularity at $p_i$ for $i=1, \cdots,k$; and then, use the reduction method to find a solution $u=\bar{u}+\phi$ which satisfies the following:
\begin{equation*}
(-\Delta)^\gamma u-u^{\beta}=\sum_{i=1}^k\sum_{j=0}^\infty\sum_{l=0}^n c_{jl}^i (w_j^i)^{\frac{4\gamma}{n-2\gamma}}Z_{jl}^i
\end{equation*}
where the right hand side is some Lagrangian multiplier which contains the approximate kernels of the linearized operator.

The last step is to determine the infinitely dimensional free parameter set such that all $c_{jl}^i$ vanish.  It turns out that this is reduced to some infinite dimensional Toda system around each singular point $p_i$ (which determines the perturbation parameters for the bubbles), and some balancing conditions  (equations \eqref{eq605} and \eqref{eq606}), which determine the necksize parameters for the Delaunay solutions). A key property in the proof is to have a sufficiently good approximate solution (a half-Dancer) so that all the estimates are exponentially decreasing in terms of the index $j$ for the bubbles. The problem of adjusting the parameters to have all $c_{jl}^i$ equal to $0$ is then solved by a fixed point argument in suitable weighted spaces.

\medskip

We remark here that in all our results we do not use the well known extension problem for the fractional Laplacian \cite{Caffarelli-Silvestre}. Instead we are inspired to the previous paper \cite{dpgw} to rewrite the fractional Laplacian in radial coordinates in terms of a new integro-differential operator in the variable $t$. In any case, if we write our problem in the extension, at least for the linear theory, it provides an example of an edge boundary value problem of the type considered in \cite{Mazzeo:edge,Mazzeo:edge2}.

When the singular set $\Sigma$ is a smooth submanifold of dimension $N>0$, problem  \eqref{eq101} has been considered in \cite{ACDFGW}. In this setting, in order to have a solution one needs to impose some necessary conditions  on $N$ (see \cite{Gonzalez-Mazzeo-Sire,Zhang}). The existence of weak solutions with larger Hausdorff dimension singular set has been studied separately \cite{acw}.\\

The paper will be structured as follows: in Section \ref{sec2} we recall some results about Delaunay solutions for \eqref{equation0} from \cite{dpgw}, while in Section \ref{sec3} we use those as models to construct a suitable approximate solution for our problem. Sections \ref{sec4} and \ref{sec5} are of technical nature; here we calculate the interactions between different bubbles. Finally, the proof of Theorem \ref{main-theorem} is contained in Section \ref{sec6}.

\section{Delaunay-type solutions}\label{sec2}

In this section we recall some recent results in \cite{paper1,dpgw} on the Delaunay solutions of
\begin{equation}\label{eq201}
(-\Delta_{\mathbb R^n})^\gamma u=c_{n,\gamma}u^{\beta} \mbox{ in }\R^n \backslash \{0\}.
\end{equation}
We may reduce \eqref{eq201} by writing
\begin{equation*}
u(x)=r^{-\frac{n-2\gamma}{2}}v(-\log |x|)
\end{equation*}
and using $t=-\log |x|$.

There are two distinguished solutions to \eqref{eq201}:
\begin{itemize}
\item[\emph{i.}] The cylinder, which is $v(t)\equiv C$, that corresponds to the singular solution $u(x)=Cr^{-\frac{n-2\gamma}{2}}$.
\item[\emph{ii.}] The standard sphere (also known as ``bubble")
\begin{equation*}\label{bubble}v(t)=(\cosh(t-t_0))^{-\frac{n-2\gamma}{2}},\end{equation*} for any $t_0\in\mathbb R$, which is regular at the origin.
\end{itemize}
Moreover, it is well-known that all the smooth solutions to problem \eqref{eq101} are of the form
\begin{equation*}
w(x)=\Big(\frac{\lambda}{\lambda^2+|x-x_0|^2}\Big)^{\frac{n-2\gamma}{2}}.
\end{equation*}
For the standard bubble solution we have the following non-degeneracy result (Theorem 1 in \cite{dps}):
\begin{lemma}\label{lemma:non-degenerate}
The solution $w(x)=(\frac{1}{1+|x|^2})^{\frac{n-2\gamma}{2}}$ of \eqref{eq101} is non-degenerate in the sense that all bounded solutions of equation
\begin{equation*}
(-\Delta)^\gamma \psi-c_{n,\gamma}\beta w^{\beta-1}\psi=0 \mbox{ in }\R^n
\end{equation*}
are linear combinations of the functions
\begin{equation*}
\frac{n-2\gamma}{2}w+x\cdot \nabla w, \ \mbox{ and } \ \partial_{x_i}w, \ 1\leq i\leq n.
\end{equation*}
\end{lemma}

Note that we normalize the constant $c_{n,\gamma}$ in \eqref{eq201} such that the standard bubble is a solution. The exact value of the constants may be found in \cite{paper1} but in this paper this is not important.\\

In \cite{dpgw}, the authors consider the existence of solutions $v(t)$ which are periodic in $t$. Using the change of variable $t=-\log |x|$, the equation \eqref{eq201} can be written as
\begin{equation}\label{eq202}
\mathcal{L}_\gamma v=c_{n,\gamma}v^\beta, \quad t\in \R, \ v>0,
\end{equation}
where
$\mathcal{L}_\gamma$ is the linear operator defined by
\begin{equation*}\label{eq203}
\mathcal{L}_\gamma v=\kappa_{n,\gamma}P.V. \int_{-\infty}^{+\infty}(v(t)-v(\tau))K(t-\tau)d\tau+c_{n,\gamma}v(t)
\end{equation*}
for $K$ a singular kernel given in (2.13) of \cite{dpgw} and
\begin{equation*}
\kappa_{n,\gamma}=\pi^{-\frac{n}{2}}2^{2\gamma}\frac{\Gamma(\frac{n}{2}+\gamma)}
{\Gamma(1-\gamma)}\gamma.
\end{equation*}
One has the following asymptotic behaviour  for $K$:
\begin{lemma}[Lemma 2.5 in \cite{dpgw}]\label{lemma201}
The asymptotic expansion of the kernel $K$ is given by
\begin{equation*}
K(\xi)\sim\left\{\begin{array}{l}
|\xi|^{-1-2\gamma},\quad |\xi|\to 0,\\
e^{-\frac{n+2\gamma}{2}|\xi|  }, \quad |\xi|\to \infty.
\end{array}
\right.
\end{equation*}
\end{lemma}

Since we are looking for periodic solutions of (\ref{eq202}), we assume that $v(t+L)=v(t)$; in this case, equation (\ref{eq202}) becomes
\begin{equation*}
\mathcal{L}_\gamma^L(v)=c_{n,\gamma}v^\beta,\quad v>0,
\end{equation*}
where
\begin{equation*}
\mathcal{L}_\gamma^L(v)=\kappa_{n,\gamma}
P.V.\int_{-\frac{L}{2}}^{\frac{L}{2}}(v(t)-v(\tau))K_L(t-\tau)d\tau+c_{n,\gamma}v
\end{equation*}
for the singular kernel
\begin{equation*}
K_L(t-\tau)=\sum_{j\in Z}K(t-\tau-jL).
\end{equation*}

We are going to consider the problem
\begin{equation}\label{eq204}
\left\{\begin{split}
&\mathcal{L}_\gamma^L(v)=c_{n,\gamma}v^\beta  \ \text{ in } \ (-\tfrac{L}{2}, \tfrac{L}{2}),\\
&v'(-\tfrac{L}{2})=v'(\tfrac{L}{2})=0.
\end{split}
\right.
\end{equation}
For this we shall work with the norm given by
\begin{equation*}
\|v\|_{H_L^\gamma}=\Big( \int_{-L/2}^{L/2}\int_{-L/2}^{L/2}(v(t)-v(\tau))^2K_L(t-\tau)\,d\tau dt+\int_{-L/2}^{L/2}v^2\,dt\Big)^{1/2},
\end{equation*}
and the following functional space
\begin{equation*}
H_L^\gamma=\left\{v:(-\tfrac{L}{2}, \tfrac{L}{2})\to \R\,\,;\,\, v'(-\tfrac{L}{2})=v'(\tfrac{L}{2})=0 \mbox{ and } \|v\|_{H_L^\gamma}<\infty\right\}.
\end{equation*}

\begin{proposition}\label{pro201}
Consider problem \eqref{eq204}. Then for $L$ large there exists a unique positive  solution $v_L$ in $H_L^\gamma$ with the following properties
\begin{itemize}
\item[(a)] $v_L$ is even in $t$;
\item[(b)] $v_L=\sum_{j\in Z}v(t-jL)+\psi_L$, where $\|\psi_L\|_{H^L_\gamma}\to 0$ as $L\to \infty$,
\end{itemize}
where
\begin{equation*}
v(t):=(\cosh t)^{-\frac{n-2\gamma}{2}}
\end{equation*}
corresponds to the standard bubble solution.

More precisely, for $\gamma\in (0,1)$, and for $L$ large we have the following Holder estimates on $\psi_L$:
\begin{equation*}
\|\psi_L\|_{\mathcal C^{2\gamma+\alpha}(-L/2,L/2)}\leq Ce^{-\frac{(n-2\gamma)L}{4}(1+\xi)}
\end{equation*}
for some $\alpha\in(0,1)$, and $\xi>0$ independent of $L$ large.
\end{proposition}

As an immediate consequence of Proposition \ref{pro201} we obtain periodic solutions for the original equation \eqref{eq202}:

\begin{coro}\label{pro202}
For $L$ large there exists a unique positive  solution $v_L$ of \eqref{eq202} with the following properties
\begin{itemize}
\item[(a)] $v_L$ is periodic and even in $t$;
\item[(b)] $v_L=\sum_{j\in \mathbb Z}v(t-jL)+\psi_L$, where $\|\psi_L\|_{H_L^\gamma}\to 0$ as $L\to \infty$ in $(-L/2, L/2)$.
\end{itemize}
More precisely, for $\gamma\in (0,1)$, and for $L$ large we have the following Holder estimates on $\psi_L$:
\begin{equation*}
\|\psi_L\|_{\mathcal C^{2\gamma+\alpha}}\leq Ce^{-\frac{(n-2\gamma)L}{4}(1+\xi)}
\end{equation*}
for some $\alpha\in (0,1)$, and $\xi>0$ independent of $L$.
\end{coro}

\medskip
\noindent{\bf Proof of Proposition \ref{pro201}.}
We denote the function
\begin{equation}\label{v0L}
v_{0,L}(t)=\sum_{j=-\infty}^\infty v_j(t),
\end{equation}
where $v_j(t)=v(t-jL)=(\cosh (t-jL))^{-\frac{n-2\gamma}{2}}$.
By symmetry, this function satisfies the boundary condition at $t=\pm \frac{L}{2}$. We consider next the functional
\begin{equation*}\begin{split}
F_L(v)=\frac{\kappa_{n,\gamma}}{4}
\int_{-L/2}^{L/2}\int_{-L/2}^{L/2}(v(t)-v(\tau))^2K_L(t-\tau)\,d\tau dt+\frac{c_{n,\gamma}}{2}\int_{-L/2}^{L/2}v^2\,dt
-\frac{c_{n,\gamma}}{\beta+1}\int_{-L/2}^{L/2}v^{\beta+1}\,dt,
\end{split}\end{equation*}
in the space
\begin{equation*}
v\in H_*^\gamma, \quad H_*^\gamma=\{v\in H_L^\gamma, \ v(t)=v(-t)\}.
\end{equation*}
Solutions of equation (\ref{eq204}) are critical points of $F_L$. Moreover, we have
\begin{equation*}
F_L'(v_{0,L})[\varphi]=\langle \mathcal{L}_\gamma^L(v_{0,L}), \varphi\rangle-c_{n,\gamma}\int_{-L/2}^{L/2}v_{0,L}^\beta \varphi \,dt=\langle S(v_{0,L}), \varphi\rangle
\end{equation*}
for every test function $\varphi$, where $\langle\, , \,\rangle$ is defined by
\begin{equation*}\begin{split}
\langle \mathcal{L}_\gamma^L (v_{0,L}), \varphi\rangle&=\frac{\kappa_{n,\gamma}}{2} P.V.\int_{-L/2}^{L/2}\int_{-L/2}^{L/2}(v_{0,L}(t)-v_{0,L}(\tau))(\varphi(t)-\varphi(\tau))K_L(t-\tau)\,dt d\tau\\
&+c_{n,\gamma}\int_{-L/2}^{L/2}v_{0,L} \varphi(t) \,dt
\end{split}\end{equation*}
and
\begin{equation*}
S(v_{0,L}):=\mathcal{L}_\gamma^L(v_{0,L})-c_{n,\gamma}v_{0,L}^\beta.
\end{equation*}
Therefore, using H\"older's inequality, we easily get
\begin{equation*}
\|F'_L\|_{H^\gamma_L}\leq C\|S(v_{0,L})\|_{L^2},
\end{equation*}
where $C$ is independent of $L$ large. Hence, we need to estimate the $L^2$ norm of $S(v_{0,L})$ in $(-L/2,L/2)$. Recalling  \eqref{v0L} and  the definition of $v_j$, we have
\begin{equation*}\begin{split}
S(v_{0,L})&=
%\kappa_{n,\gamma}P.V.\int_{-\infty}^\infty &(v_{0,L}(t)-v_{0,L}(\tau))K(t-\tau)d\tau+c_{n,\gamma}v_{0,L}-c_{n,\gamma}v_{0,L}^\beta\\
c_{n,\gamma}\big[\sum_{j=-\infty}^\infty v_j^\beta-(\sum_{j=-\infty}^\infty v_j)^\beta\big] \quad\mbox{ in }(-L/2,L/2).
\end{split}\end{equation*}
%\begin{color}{blue}Here we do not have 1/2, we have to unify the value of $\kappa_{n,\gamma}$\end{color}
For $t\geq 0$, since $v_{-j}\leq v_j$, by symmetry, for $L$ large,
\begin{equation*}
|S(v_{0,L})|\leq C v_0^{\beta-1}\sum_{j\neq 0}v_j+\sum_{j\neq 0}v_j^\beta.
\end{equation*}
As a consequence, we have
\begin{equation*}
\int_{-L/2}^{L/2}S(v_{0,L})^2dt\leq C\int_{-L/2}^{L/2}v_0^{2(\beta-1)}(\sum_{j\neq 0}v_j)^2+\int_{-L/2}^{L/2}(\sum_{j\neq 0}v_j^\beta)^2.
\end{equation*}
In order to estimate the first term, we divide the domain into two subsets, $\{|t|\leq \frac{\alpha L}{2}\}$ and $\{|t|\geq \frac{\alpha L}{2}\}$ for $\alpha \in (0,1)$. In these two sets we have the estimates $\sum_{j\neq 0}v_j\leq Ce^{-\frac{(n-2\gamma)L}{4}(2-\alpha)}$ and $\sum_{j\neq 0}v_j\leq Ce^{-\frac{(n-2\gamma)L}{4}}$, respectively, by the exponential decay of $v_0$. Hence one easily finds
\begin{equation*}\begin{split}
\int_{-L/2}^{L/2}v_0^{2(\beta-1)}(\sum_{j\neq 0}v_j)^2dt&\leq Ce^{-\frac{(n-2\gamma)L}{2}(2-\alpha)}+Ce^{-\frac{(n-2\gamma)L}{2}}e^{-2(\beta-1)\frac{(n-2\gamma)}{2}\frac{\alpha L}{2}}\\
&=Ce^{-\frac{(n-2\gamma)L}{2}(2-\alpha)}+Ce^{-\frac{(n-2\gamma)L}{2}(1+\frac{4\alpha \gamma}{n-2\gamma})}
\end{split}\end{equation*}
and
\begin{equation*}
\int_{-L/2}^{L/2}(\sum_{j\neq 0}v_j^\beta)^2dt\leq Ce^{-\frac{(n-2\gamma)L\beta}{2}}.
\end{equation*}
In conclusion, we have
\begin{equation*}
\|S(v_{0,L})\|_{L^2(-L/2,L/2)}\leq Ce^{-\frac{(n-2\gamma)L}{4}(1+\xi)}
\end{equation*}
for  some $\xi>0 $ independent of $L$ large.

Next we claim that the operator $F_L''(v_{0,L})$ is invertible in the space $H_*^\gamma(-L/2,L/2)$. This follows from the non-degeneracy of the standard bubble and the fact that we are working in the subspace of even functions in $t$. This allows us to solve the problem via local inversion. In fact, we write $v_L=v_{0,L}+\psi$ and we have $F_L'(v_{0,L}+\psi)=0$ if and only if $\psi\in H_*^\gamma(-L/2,L/2)$ satisfies
\begin{equation*}
\psi=-(F_L(v_{0,L}))''[F_L'(v_{0,L})+N(\psi)],
\end{equation*}
where $N(\psi)=c_{n,\gamma}[(v_{0,L}+\psi)^\beta-v_{0,L}^\beta-\beta v_{0,L}^{\beta-1}\psi]$ is superlinear in $\psi$. We can apply the contraction mapping theorem, obtaining a solution $\psi$ which satisfies
\begin{equation*}
\|\psi\|_{H_L^\gamma}\leq C\|F_L'(v_{0,L})\|_{H_L^\gamma}\leq C\|S(v_{0,L})\|_{L^2}\leq Ce^{-\frac{(n-2\gamma)L}{4}(1+\xi)}.
\end{equation*}
For $\gamma \in (0,1)$, by the regularity estimates given in \cite{dpgw} and summarized in Remark $3.12$ in the same paper (see also \cite{hitchhiker}), it follows that $\psi$ is smooth and we have the following estimate:
%\begin{equation}
%\|\psi\|_{L^\infty}\leq Ce^{-\frac{(n-2\gamma)L}{4}(1+\xi)}.
%\end{equation}
%By elliptic regularity,
\begin{equation*}
\|\psi\|_{\mathcal C^{2\gamma+\alpha}}\leq Ce^{-\frac{(n-2\gamma)L}{4}(1+\xi)}.
\end{equation*}
The maximum principle of \cite{dpgw} concludes the proof of the proposition.

\qed

\begin{remark}
Since the equation for $v$ is translational invariant, if $v(t) $ is a solution of (\ref{eq202}), then $v(t-t_0)$ is also a solution. In the following, we will use the periodic solution $v_L$ with period $L$ which attains its minimum at the points $t=jL$, $j\in\mathbb Z$. By Corollary  \ref{pro202}, this periodic solution can be expressed as a perturbation of a bubble tower (or Dancer solution)
\begin{equation*}
v_L(t)=\sum_{j=-\infty}^\infty v\big(t-\tfrac{L}{2}-jL\big)+\psi_L(t),
\end{equation*}
where $\|\psi_L\|_{\mathcal C^{2\gamma+\alpha}}\leq Ce^{-\frac{(n-2\gamma)L}{4}(1+\xi)}$ for some $\xi>0$ independent of $L$. For the rest of the paper we write $$t_j=\tfrac{L}{2}+jL,\quad j\in \mathbb Z,$$
and
\begin{equation*}
v_j(t):=v(t-t_j)=\cosh(t-t_j)^{-\frac{n-2\gamma}{2}}, \ t\in \R.\\
\end{equation*}
\end{remark}

Now we consider only half a bubble tower; this is needed in order to have fast decay far from the singularity ($t\to -\infty$). We define
\begin{equation*}
\tilde{v}_L(t)=\sum_{j=0}^\infty v\big(t-\tfrac{L}{2}-jL\big)+\psi_L(t),
\end{equation*}
then one has the following asymptotic behaviour of $\tilde{v}_L$:
\begin{equation*}
\tilde{v}_L(0)=e^{-\frac{(n-2\gamma)L}{4}}(1+o(1)),
\end{equation*}
(this is the \emph{neck size}). And for $t\leq 0$, i.e. $|x|\geq 1$, using the fact that $v$ is exponential decaying,
\begin{equation*}\begin{split}
\tilde{v}_L(t)=v_0(t)(1+o(1))=\big(\cosh(t-\tfrac{L}{2})\big)^{-\frac{n-2\gamma}{2}}(1+o(1))
=|x|^{-\frac{n-2\gamma}{2}}e^{-\frac{(n-2\gamma)L}{4}}(1+o(1)),
\end{split}\end{equation*}
and the corresponding solution $\tilde{u}_L=|x|^{-\frac{n-2\gamma}{2}}\tilde{v}_L$ satisfies
\begin{equation}\label{estimacion-lejos}
\tilde{u}_L(x)=|x|^{-\frac{n-2\gamma}{2}}\tilde{v}_L=|x|^{-(n-2\gamma)}e^{-\frac{(n-2\gamma)L}{4}}(1+o(1)).
\end{equation}

\section{Construction of the approximate solutions}\label{sec3}

We now proceed to define a family of approximate solutions to the problem using the Delaunay solutions from the previous section. We know that the Delaunay solution with period $L$ has the form of a bubble tower, i.e,
\begin{equation}\label{exact-Delaunay}\begin{split}
u_L(x)&=|x|^{-\frac{n-2\gamma}{2}}\Big(\sum_{j=-\infty}^\infty v\big(-\log |x|-\tfrac{L}{2}-jL\big)+\psi_L(-\log |x|)\Big)\\
&=:\sum_{j=-\infty}^\infty \Big(\frac{\lambda_j}{\lambda_j^2+|x|^2}\Big)^{\frac{n-2\gamma}{2}}+\phi_L(x),
\end{split}\end{equation}
where
\begin{equation}\label{psi}\lambda_j=e^{-\frac{1+2j}{2}L}\quad\text{and} \quad\phi_L(x)=|x|^{-\frac{n-2\gamma}{2}}\psi_L(-\log |x|),\end{equation}
for $\psi_L$ the perturbation function constructed in Corollary \ref{pro202}.

As we have mentioned, one of the main ideas is that, although we would like the approximate solution to have Delaunay-type singularities around each point of $\Sigma$, it should have a fast decay once we are away from $\Sigma$ in order to glue to the background manifold $\mathbb R^n$. To this end, we will only take half a Delaunay solution (this is, only values $j=0,1,\ldots$).

In addition, we would like to introduce some perturbation parameters $R\in\mathbb R$, $a\in\mathbb R^n$, since each standard bubble has $n+1$ free parameters which correspond to scaling and translations. This is done for each bubble in the bubble tower independently, thus we will have an infinite dimensional set of perturbations.

Keeping both aspects in mind, let us give the precise construction of our approximate solution $\bar u$. First, one can always assume that all the balls $B(p_i,2)$ are disjoint, since we may dilate the problem by some factor $\kappa>0$ that will change the set $\Sigma$ into $\kappa \Sigma$ and a function $u$ defined in $\R^n \backslash \Sigma$ into $\kappa^{-\frac{n-2\gamma}{2}}u(x/\kappa)$ defined in $\R^n \backslash \kappa \Sigma$.

Let $\chi$ be a cut-off function such that
\begin{equation*}
\chi(x)=\left\{\begin{array}{l}
1, \ \text{ if } |x|\leq \frac{1}{2},\\
0, \ \text{ if }|x|\geq 1,\\
\chi\in [0,1], \ \text{ if }\frac{1}{2}\leq |x|\leq 1,
\end{array}
\right.
\end{equation*}
and set $\chi_i(x)=\chi(x-p_i)$.

Given $L>0$ large enough, we will fix
\begin{equation*}
\bar{L}=(L_1,\cdots,L_k)
\end{equation*}
to be the Delaunay parameters, which also are related to the neck sizes of each Delaunay solution. They will be chosen (large enough) in the proof. They will satisfy the following conditions:
\begin{equation*}
|L_i-L|\leq C.
\end{equation*}
More precisely, they will be related by the following:
\begin{equation}\label{q_i}q_i e^{-\frac{(n-2\gamma)L}{4}}=e^{-\frac{(n-2\gamma)L_i}{4}},\end{equation}
for some $q_i>0$, $i=1,\cdots,k$.

Also, for $i=1,\ldots,k$, $j=0,1,\ldots$, set $a_j^i\in\mathbb R^n$ and $R_j^i=R^i(1+r_j^i)\in\mathbb R$ to be the perturbation parameters. Define the approximate solution $\bar u$ as
\begin{equation}\label{eq301}\begin{split}
\bar{u}(x)&=\sum_{i=1}^k\Big[
\sum_{j=0}^\infty \big[\,|x-p_i-a_j^i|^{-\frac{n-2\gamma}{2}}v(-\log |x-p_i-a_j^i|-\tfrac{L_i}{2}-jL_i+\log R_j^i)\big]\\
&\qquad\quad+
\chi_i(x)|x-p_j|^{-\frac{n-2\gamma}{2}}\psi_{i}(-\log|x-p_i|+\log R^i)\Big]\\
&=\sum_{i=1}^k\Big[\sum_{j=0}^\infty \Big(\frac{\lambda_j^i}{|\lambda_j^i|^2+|x-p_i-a_j^i|^2}\Big)^{\frac{n-2\gamma}{2}}
+\chi_i(x)\phi_{i}(x-p_i)
\Big]\\
&=:\sum_{i=1}^k\Big[\sum_{j=0}^\infty w_j^i+\chi_i(x)\phi_{i}(x-p_i)\Big],
\end{split}
\end{equation}
where we have set $$\lambda_j^i=R_j^i e^{-\frac{(1+2j)L_i}{2}}.$$

Next we will explain in detail the perturbation parameters $q_i, a_j^i, R^i_j$.
First fix a set of positive numbers $q_1^b,\cdots,q_k^b$, and let $a_0^{i,b}, R^{i,b}$ be determined by the following balancing conditions:
%\begin{color}{blue}Should we write $R^{i,b}_0$ and $R^{i',b}_0$ instead of $R^{i,b}$ and $R^{i',b}$ to unify the notation with $a_0^{i,b}$\end{color}\\
\begin{eqnarray}\label{balance1}
q^b_i=A_2\sum_{i'\neq i}q^b_{i'}(R^{i,b}R^{i',b})^{\frac{n-2\gamma}{2}}|p_i-p_{i'}|^{-(n-2\gamma)},\quad i=1,\ldots,k,
\end{eqnarray}
and
\begin{eqnarray}\label{balance2}
\frac{{a}_0^{i,b}}{(\lambda_0^{i,b})^2}=-\frac{A_3}{A_0}\sum_{i'\neq i}\frac{p_{i'}-p_i}{|p_{i'}-p_i|^{n-2\gamma+2}}\frac{q^b_{i'}}{q^b_i}
(R^{i,b}R^{i',b})^{\frac{n-2\gamma}{2}},
\quad i=1,\ldots,k,
\end{eqnarray}
where $\lambda_0^{i,b}=R^{i,b} e^{-\frac{L^b_i}{2}}$, and the $L^b_i$ are defined from the $q_i^b$ by (\ref{q_i}), $i=1,\ldots,k$ and the constants $A_0$,$A_2>0$, $A_3<0$ are defined in the appendix.

\begin{remark}
It has been shown in Remark 3 of \cite{mp} that for $\bar{q}:=(q_1^b,\dots,q_k^b)$ in the positive octant, there exists a solution $R^{i,b}$ to equation \eqref{balance1}. Once $R^{i,b}$ is chosen, then we can use equation \eqref{balance2} to determine ${a}_0^{i,b}$.
\end{remark}

Although the meaning of these compatibility conditions will become clear in the next sections, we have just seen that they are the analogous to those of \cite{mp} for the local case. The idea is that, at the \emph{base} level, perturbations should be very close to those for a single bubble. This also shows, in particular, that although our problem is non-local, very near the singularity it presents a local behavior due to the strong influence of the underlying geometry.

However, for the rest of the parameters $a_j^i, R_j^i$, $i=1,\ldots k$, $j=0,1,\ldots$, we will have to solve an infinite dimensional system of equations. First let $R^i, q_i$ be $2k$ parameters which satisfy
\begin{equation}\label{parar}
|R^i-R^{i,b}|\leq C, \quad  |q_i-q_i^b|\leq C.
\end{equation}
and let $\lambda_0^{i,0}=R^{i} e^{-\frac{(1+2j)L_i}{2}}$.

%\begin{color}{blue}$\lambda_0^{i,0}$ depends on $j$ or why are we defining it as $R^{i} e^{-\frac{(1+2j)L_i}{2}}$ with $j$ in the exponential? Should it be $\lambda_0^{i}$ or $\lambda_0^{i,j}$? \end{color}
Set also $\hat{a}_0^i$ given by $\frac{a_0^{i,0}}{(\lambda_0^{i,0})^2}=\hat{a}_0^{i}$ be $k$ parameters satisfying
%\begin{color}{blue}Do you mean $\lambda_0^{i}$ and $a_0^{i}$ with $\lambda_0^{i,0}$ and $a_0^{i,0}$?\end{color}
\begin{equation}\label{paraa}
|\hat{a}_0^i-\hat{a}_0^{i,b}|\leq C,
\end{equation}
where $\hat{a}_0^{i,b}=\frac{{a}_0^{i,b}}{(\lambda_0^{i,b})^2}$.

Last we define the parameters $R_j^i, a_j^i$ by
\begin{equation}\label{para1}
R_j^i=R^i(1+r_j^i), \quad \frac{a_j^i}{(\lambda_j^i)^2}=\bar{a}_j^i=\hat{a}_0^i+\tilde{a}_j^i,\quad i=1,\ldots,k, \quad j=0,\ldots,\infty,
\end{equation}
where $r_j^i, \tilde{a}_j^i$ satisfy
\begin{equation}\label{para2}
|r_j^i|\leq Ce^{-\tau t^i_j}, \ |\tilde{a}_j^i|\leq Ce^{-\tau t^i_j},
\end{equation}
for some $\tau>0$, where $t_j^i=(\frac{1}{2}+j)L_i$. The exact value of the parameters will be determined in Section \ref{sec6}.

Let us give some explanation about the choice of parameters. Given the $k(n+2)$ balancing parameters $q_i^b, R^{i,b}, \hat{a}_0^{i,b}$ satisfying the balancing conditions (\ref{balance1})-(\ref{balance2}), we first choose $k(n+2)$ initial perturbation parameters $q_i, R^i, \hat{a}_0^{i}$ which are close to the balancing parameters, i.e (\ref{parar})-(\ref{paraa}). After that, we introduce infinitely many other perturbation parameters $\tilde{a}_j^i, r_j^i$ which are exponential decaying in $t_j^i$, i.e. (\ref{para1})-(\ref{para2}).

\medskip

We will prove next some quantitative estimates on the function $\bar{u}$, and in particular on its behaviour near the singular points. Before that we need to introduce the function spaces we will work with.

\begin{definition} We set the weighted norm
\begin{equation*}
\|u\|_{\mathcal C^\alpha_{\gamma_1,\gamma_2}}
=\|\dist(x,\Sigma)^{-\gamma_1}u\|_{\mathcal C^\alpha(B_1(\Sigma))}+\||x|^{-\gamma_2}u\|_{\mathcal C^\alpha(\R^n \backslash B_1(\Sigma))}.
\end{equation*}

\end{definition}

In other words, to check if $u$ is an element of some $\mathcal C^{\alpha}_{\gamma_1,\gamma_2}$, it is sufficient to check that $u$ is bounded by a constant times $|x-p_i|^{\gamma_1}$ and has its $\ell$-th order partial derivatives bounded by a constant times $|x-p_i|^{\gamma_1-\ell}$ for $\ell\leq \alpha$ near each singular point $p_i$. Away from the singular set $\Sigma$, $u$ is bounded by $|x|^{\gamma_2}$  and has its $\ell$-th order partial derivatives bounded by a constant times $|x|^{\gamma_2-\ell}$ for $\ell \leq \alpha$ (note that here we are implicitly assuming that $0\in\Sigma$, in order to simplify the notation).\\

First, we define $Z_{j,l}^i$ to be the (normalized) approximate kernels
\begin{equation*}
Z_{j,0}^i=\frac{\partial }{\partial r_{j}^i}w_{j}^i, \quad Z_{j,l}^i=\lambda_j^i
\frac{\partial }{\partial {a}_{j,l}^i}w_{j}^i=-\lambda_j^i\frac{\partial }{\partial x_l}w_j^i, \quad l=1,\cdots,n.
\end{equation*}

Without loss of generality, assume in the following that $p_i=0$. For $l=0$ we will repeatedly use the following estimates
\begin{equation}\label{Z_0}
|Z_{j,0}^i|\leq C\left\{\begin{split}
&|x|^{-\frac{n-2\gamma}{2}}|v_j^i|, \quad |x|\leq 1,\\
&|x|^{-(n-2\gamma)}(\lambda_j^i)^{\frac{n-2\gamma}{2}},\quad  |x|\geq 1.\\
\end{split}\right.
\end{equation}
In addition, for $l=1,\ldots,n$, we have
\begin{equation}\label{Z_l}
Z_{j,l}^i=(n-2\gamma)(v_j^i)^{1+\frac{2}{n-2\gamma}}
|x-a_j^i-p_i|^{-\frac{n-2\gamma}{2}-1}(x-a_j^i-p_i)_l,
\end{equation}
where we have used the obvious notation $w_j^i=|x-p_i-a^i_j|^{-\frac{n-2\gamma}{2}}v_j^i$. Then one has the following \emph{orthogonality} conditions (recentering at $p_i=0$):
\begin{equation}\begin{split}\label{orthogonality1}
\int_{\R^n}&(w_j^i)^{\beta-1}Z_{j,l}^iZ_{j',l'}^i\,dx\\
&={\frac{4(n-2\gamma)^2}{n}}\delta_{l,l'}\int_{\R^n}|x|^{-2\gamma}(v_j)^{\beta-1}
|x|^{-\frac{n-2\gamma}{2}}(v_j)^{\frac{n-2\gamma+2}{n-2\gamma}}|x|^{-\frac{n-2\gamma}{2}}
v_{j'}^{\frac{n-2\gamma+2}{n-2\gamma}}\,dx+o(1)\\
&=\frac{4(n-2\gamma)^2}{n}\delta_{l,l'}\int_{\R^n}|x|^{-n}(v_j)^{\beta+\frac{2}{n-2\gamma}}
(v_{j'})^{1+\frac{2}{n-2\gamma}}\,dx+o(1)\\
&=\frac{4(n-2\gamma)^2}{n}(\delta_{l,l'}+o(1))\,e^{-\frac{n-2\gamma+2}{2}
|t^i_j-t^i_{j'}|}.
\end{split}\end{equation}
Similar estimates also hold true for $l=0$. Indeed,
\begin{equation}\label{orthogonality2}\begin{split}
\int_{\R^n}(w_j^i)^{\beta-1}Z_{j,0}^iZ_{j',0}^{i}\,dx
&=(1+o(1))\int_{\R^n}|x-p_i|^{-n}v_j^{\beta-1}v_j'v_{j'}'\,dx\\
&=C_0(1+o(1))e^{-\frac{n-2\gamma}{2}|t^i_j-t^i_{j'}|}
\end{split}\end{equation}
for some $C_0>0$.\\

From now on, we choose $-\frac{n-2\gamma}{2}<\gamma_1<\min\{-\frac{n-2\gamma}{2}+2\gamma,0\}$. Define also
\begin{equation*}\label{norm}
\|u\|_*=\|u\|_{\mathcal C^{2\gamma+\alpha}_{\min\{\gamma_1,-\frac{n-2\gamma}{2}+\tau\},-(n-2\gamma)}},  \quad \|h\|_{**}=\|h\|_{\mathcal C^{\alpha}_{\min\{\gamma_1,-\frac{n-2\gamma}{2}+\tau\}-2\gamma, -(n+2\gamma)}},
\end{equation*}
and denote by $\mathcal C_*$ and $\mathcal C_{**}$ the corresponding weighted H\"older spaces. Here $\tau$ (small enough) is given in the definition of the perturbation parameters (\ref{para1})-(\ref{para2}). Remark that, to simplify the notation, many times we will ignore the small $\tau$ perturbation and just the weight near the singular set as $\dist(x,\Sigma)^{-\gamma_1}$, $\dist(x,\Sigma)^{-(\gamma_1-2\gamma)}$, respectively.

Our main result in this section is the following proposition:

\begin{proposition}\label{proposition301}
Suppose the parameters satisfy \eqref{parar}-\eqref{para2}, and let $\bar{u}$ be as in \eqref{eq301}. Then for $L$ large enough, there exists a function $\phi$ and a sequence $\{c_{j,l}^i\}$ which satisfies the following properties:
\begin{equation}\label{eq303}
\left\{\begin{array}{l}
(-\Delta)^\gamma (\bar{u}+\phi)-c_{n,\gamma}(\bar{u}+\phi)^\beta=\displaystyle{\sum_{i=1}^k\sum_{j=0}^\infty \sum_{l=0}^n} c_{j,l}^i(w_j^i)^{\beta-1}Z_{j,l}^i,\\
\int_{\R^n}\phi (w_j^i)^{\beta-1}Z_{j,l}^i\,dx=0 \ \ \mbox{ for }i=1,\cdots,k, \ j=0,\cdots,\infty, \ l=0,\cdots,n.
\end{array}
\right.
\end{equation}
Moreover, one has
\begin{equation}\label{decay-phi}
\|\phi\|_{*}\leq Ce^{-\frac{(n-2\gamma)L}{4}(1+\xi)},
\end{equation}
for some $-\frac{n-2\gamma}{2}<\gamma_1<\min\{-\frac{n-2\gamma}{2}+2\gamma,0\}$ and $\xi>0$ independent of $L$ large.
\end{proposition}

The proof is technically involved, so we prove some preliminary lemmas.  We first show a result involving the auxiliary linear equation
\begin{equation}\label{eq302}
\left\{\begin{array}{l}
(-\Delta)^\gamma \phi-c_{n,\gamma}\beta\bar{u}^{\beta-1}\phi=h+\displaystyle{\sum_{i=1}^k\sum_{j=0}^\infty \sum_{l=0}^n} c_{j,l}^i(w_j^i)^{\beta-1}Z_{j,l}^i,\\
\int_{\R^n}\phi (w_j^i)^{\beta-1}Z_{j,l}^i\,dx=0 \ \ \mbox{ for }i=1,\cdots,k, \ j=0,\cdots,\infty, \ l=0,\cdots,n.
\end{array}
\right.
\end{equation}

\begin{lemma}\label{lemma301}
Suppose the parameters satisfy \eqref{parar}-\eqref{para2}. Then there exists a weight $\gamma_1$ satisfying $-\frac{n-2\gamma}{2}<\gamma_1<\min\{-\frac{n-2\gamma}{2}+2\gamma,0\}$ such that, given $h$ with $\|h\|_{**}<\infty$, equation \eqref{eq302} has a unique solution $\phi$ in the space $\mathcal C_*$. Moreover, there exists a constant $C$ independent of $L$ such that
\begin{equation}\label{eq305}
\|\phi\|_{*}\leq C\|h\|_{**}.
\end{equation}
\end{lemma}

Note that Fredholm properties for the problem \eqref{eq302} in weighted spaces have been shown in \cite{Mazzeo:edge,Mazzeo:edge2}, since it is an example of an edge boundary value problem when we look at the usual extension formulation for the fractional Laplacian from \cite{Caffarelli-Silvestre}. However, in Lemma \ref{lemma301} we show, in addition, that the  estimates are independent of the choice of Delaunay parameters $(L_1,\ldots,L_k)$.

We will postpone the proof of this lemma, instead we will show first some quantitative estimates on the function $\bar{u}$ and in particular its behaviour near the singular set $\Sigma$ and at infinity.

\begin{lemma}\label{lemma302}
Let $S(\bar{u})=(-\Delta)^\gamma \bar{u}-c_{n,\gamma}\bar{u}^\beta$. Then if the parameters satisfy \eqref{parar}-\eqref{para2}, we have the following estimate on $S(\bar{u})$:
\begin{equation}\label{estimate-S}
\|S(\bar{u})\|_{**}\leq Ce^{-\frac{(n-2\gamma)L}{4}(1+\xi)}
\end{equation}
for some $\xi>0$ independent of $L$ large.
\end{lemma}

\begin{proof}
As usual, for simplicity, we prove the estimates in (\ref{estimate-S}) for the $L^\infty$ norm, namely, we prove the following estimates:
\begin{equation}\label{estimate10}
|S(\bar{u})(x)|\leq C|x-p_i|^{\min\{\gamma_1,-\frac{n-2\gamma}{2}+\tau\}-2\gamma}e^{-\frac{(n-2\gamma)L}{4}(1+\xi)},
\end{equation}
near each singular point $p_i$ and
\begin{equation}\label{estimate11}
|S(\bar{u})(x)|\leq C|x|^{-(n+2\gamma)}e^{-\frac{(n-2\gamma)L}{4}(1+\xi)}.
\end{equation}
for $\dist(x,\Sigma)\geq 1$.

First we show the estimates for the particular case that all the parameters $a_j^i$, $r_j^i$ are zero. Let ${\bar u}_0$ be the approximate solution from \eqref{eq301} in this case.
Without loss of generality, assume $p_1=0$ and we consider in the region $\dist(x, \Sigma)\geq 1$.
In this region, $\chi_i=0$ for all $i$, one has
\begin{equation*}\label{estimate-S(u0)}\begin{split}
S(\bar{u}_0)&=(-\Delta)^\gamma \bar{u}_0-c_{n,\gamma}\bar{u}_0^{\beta}
=(-\Delta)^\gamma (\sum_{i=1}^k\sum_{j=0}^\infty w_j^i+\chi_i\phi_i)-c_{n,\gamma}(\sum_{i=1}^k\sum_{j=0}^\infty w_j^i)^\beta\\
%&&=(-\Delta)^\gamma(\sum_{j=0}^\infty w_j^1)-c_{n,\gamma}(\sum_{j=0}^\infty w_j^1+\sum_{i\neq 1}\sum_{j=0}^\infty w_j^i)^\beta+(-\Delta)^\gamma (\sum_{i\neq 1}\sum_{j=0}^\infty w_j^i)\\
%&&\qquad+(-\Delta)^\gamma (\sum_{i=1}^k\chi_i\psi_i)\\
&=-c_{n,\gamma}[(\sum_{i,j}w_j^i)^\beta-\sum_{i,j}(w_j^i)^\beta]+(-\Delta)^\gamma (\sum_{i=1}^k\chi_i\phi_i)=:I_1+I_2.
\end{split}\end{equation*}
First, using the fact that
\begin{equation*}
|w_j^i|\sim(\lambda_j^i)^{\frac{n-2\gamma}{2}}|x|^{-(n-2\gamma)},
\end{equation*}
and recalling the relation between $L$ and $L_i$ from \eqref{q_i} we have
\begin{equation*}
I_1\leq C(e^{-\frac{(n-2\gamma)L}{4}}|x|^{-(n-2\gamma)})^\beta\leq Ce^{-\frac{(n+2\gamma)L}{4}}|x|^{-(n+2\gamma)}.
\end{equation*}
For $I_2$, recall that by Corollary \ref{pro202} $$\phi_i=|x-p_i|^{-\frac{n-2\gamma}{2}}\psi_i=|x-p_i|^{-\frac{n-2\gamma}{2}}
O(e^{-\frac{(n-2\gamma)L_i}{4}(1+\xi)}),$$ we have for $|x|$ large,
\begin{equation*}\begin{split}
(-\Delta)^\gamma (\chi_i\phi_i)(x)&=P.V.\int_{\R^n}\frac{\chi_i(x)\phi_i(x)-\chi_i(y)\phi_i(y)}{|x-y|^{n+2\gamma}}\,dy
=P.V.\int_{B_1}\frac{-\chi_i(y)\phi_i(y)}{|x-y|^{n+2\gamma}}\,dy\\
&=|x|^{-(n+2\gamma)}O(e^{-\frac{(n-2\gamma)L_i}{4}(1+\xi)}).
\end{split}\end{equation*}
Thus one has for $\dist(x,\Sigma)\geq 1$,
\begin{equation*}
|S(\bar{u}_0)|\leq C|x|^{-(n+2\gamma)}e^{-\frac{(n-2\gamma)L}{4}(1+\xi)}.
\end{equation*}

Next, we consider the region $\frac{1}{2}\leq |x-p_i|\leq 1$. In this case, it is easy to check that
\begin{equation*}
|S(\bar{u}_0)|\leq Ce^{-\frac{(n-2\gamma)L}{4}(1+\xi)}.
\end{equation*}

Last we consider the region $|x|\leq \frac{1}{2}$. In this region we have $\chi_1=1$ and $\chi_i=0$ for $i\neq 1$, so
\begin{equation*}
\bar{u}_0=u_{L_1}-(1-\chi_1)\phi_1+\sum_{i\neq 1}\left(\sum_{j=0}^\infty w_j^i+\chi_i\phi_i\right)-\sum_{j=-\infty}^{-1} w_j^1.
\end{equation*}
Hence
\begin{equation*}\begin{split}
S(\bar{u}_0)%&=(-\Delta)^\gamma (\sum_{j=0}^\infty w_j^1+\chi_1\psi_1)-c_{n,\gamma}(\sum_{j=0}^\infty w_j^1+\psi_1+O(e^{-\frac{(n-2\gamma)L}{4}}))^\beta\\
%&\qquad+O(e^{-\frac{(n-2\gamma)L}{4}(1+\xi)})\\
&=(-\Delta)^\gamma u_{L_1}-c_{n,\gamma}\Big(\sum_{j=0}^\infty w_j^1+\phi_1+O(e^{-\frac{(n-2\gamma)L}{4}})\Big)^\beta+O(e^{-\frac{(n-2\gamma)L}{4}(1+\xi)})\\
&=-c_{n,\gamma}\big[(u_{L_1}+O(e^{-\frac{(n-2\gamma)L}{4}}))^\beta-u_{L_1}^\beta\big]+O(e^{-\frac{(n-2\gamma)L}{4}(1+\xi)})\\
&\leq Cu_{L_1}^{\beta-1}e^{-\frac{(n-2\gamma)L}{4}}+O(e^{-\frac{(n-2\gamma)L}{4}(1+\xi)})\\
&\leq C|x|^{-2\gamma}\Big(\sum_{j=-\infty}^\infty v_j(-\log |x|)\Big)^{\beta-1}e^{-\frac{(n-2\gamma)L}{4}}+O(e^{-\frac{(n-2\gamma)L}{4}(1+\xi)})\\
&\leq C|x|^{\gamma_1-2\gamma}|x|^{-\gamma_1}\Big(\sum_{j=-\infty}^\infty v_j(-\log |x|)\Big)^{\beta-1}e^{-\frac{(n-2\gamma)L}{4}}+O(e^{-\frac{(n-2\gamma)L}{4}(1+\xi)})\\
&\leq C|x|^{\gamma_1-2\gamma}e^{-\frac{(n-2\gamma)L}{4}(1+\xi)},
\end{split}\end{equation*}
where for the last inequality we have used \eqref{estimate-near} below in the region $|x|\leq \frac{1}{2}$. We have also denoted
$$v_j(t):=v\big(t-\tfrac{L_1}{2}-jL_1\big), \quad\text{for}\quad t=-\log |x|.$$
In any case, for $t=-\log |x|<\frac{L_1}{4}$, we have
\begin{equation}\label{eq306}
\sum_{j=-\infty}^\infty v_j(-\log |x|)\leq Ce^{-\frac{(n-2\gamma)L_1}{8}},\ \ |x|\leq C,
\end{equation}
and for $t\geq \frac{L_1}{4}$, we have
\begin{equation}\label{eq307}
|x|\leq Ce^{-\frac{L_1}{4}},\ \ \sum_{j=-\infty}^\infty v_j(-\log |x|)\leq C.
\end{equation}
Combining the above two estimates, we have for $\gamma_1<0$,
\begin{equation}\label{estimate-near}
|x|^{-\gamma_1}\big(\sum_{j=-\infty}^\infty v_j(-\log |x|)\big)^{\beta-1}\leq Ce^{-\xi L_1}.
\end{equation}
So for $|x|\leq \frac{1}{2}$, one has
\begin{equation*}
|S(\bar{u}_0)|\leq C|x|^{\gamma_1-2\gamma}e^{-\frac{(n-2\gamma)L}{4}(1+\xi)}.
\end{equation*}
Thus we get  estimates \eqref{estimate10} and \eqref{estimate11} in this particular case.\\

Now we consider the case of a general configuration $r_j^i, a_j^i$. First we differentiate $S(\bar{u}_0)$ with respect to these parameters. Since the variation is linear in the displacements of the parameters, we vary the parameter of one point at one time. Varying $r_{j}^i$, we obtain
\begin{equation*}
\frac{\partial }{\partial r_{j}^i}S(\bar{u}_0)=(-\Delta)^\gamma \frac{\partial w_j^i}{\partial r_j^i}-c_{n,\gamma}\beta \bar{u}_0^{\beta-1}\frac{\partial w_j^i}{\partial r_j^i}
=\beta c_{n,\gamma}\big[(w_j^i)^{\beta-1}-\bar{u}_0^{\beta-1}\big]\frac{\partial w_j^i}{\partial r_j^i}.
\end{equation*}
From the estimate on $\phi_i$ and the condition on $r_j^i$, we have the following estimates:

For $\dist(x,\Sigma)\geq 1$,
\begin{equation*}\begin{split}
[(w_j^i)^{\beta-1}-\bar{u}_0^{\beta-1}]\frac{\partial w_j^i}{\partial r_j^i}&\leq C(e^{-\frac{(n-2\gamma)L}{4}}|x|^{-(n-2\gamma)})^{\beta-1}
e^{-\frac{(n-2\gamma)}{4}L(2j+1)}|x|^{-(n-2\gamma)}\\
&\leq C|x|^{-(n+2\gamma)}e^{-\frac{(n-2\gamma)L}{4}(1+\xi)}e^{-\sigma t_j^i},
\end{split}\end{equation*}
for a suitable choice of $\sigma>0$.
Next, when $|x-p_i|\leq 1$ for $i\neq 1$, for instance, similar to the estimates \eqref{eq306} and \eqref{eq307}, one has
\begin{equation*}\begin{split}
[(w_j^i)^{\beta-1}-\bar{u}_0^{\beta-1}]\frac{\partial w_j^i}{\partial r_j^i}&
\leq C|x-p_i|^{-2\gamma}(\sum_{j=0}^\infty v_j(-\log |x-p_i|))^{\beta-1}e^{-\frac{(n-2\gamma)L}{4}(2j+1)}\\
&\leq C|x-p_i|^{\gamma_1-2\gamma}e^{-\frac{(n-2\gamma)L}{4}(1+\xi)}e^{-\sigma t_j^i},
\end{split}\end{equation*}
while for $|x-p_1|\leq 1$, if $|t-t^i_j|\leq \frac{L_1}{2}$ it is true that
\begin{equation*}\begin{split}
[(w_j^i)^{\beta-1}-\bar{u}_0^{\beta-1}]\frac{\partial w_j^i}{\partial r_j^i}&\leq C (w_j^i)^{\beta-1}\big[\sum_{l\neq j}w_l^i+e^{-\frac{(n-2\gamma)L}{4}}\big]\\
&\leq C\big[|x|^{-\frac{n+2\gamma}{2}}\sum_{l\neq j}v_j^{\beta-1}v_l+|x|^{-2\gamma}v_j^{\beta-1}e^{-\frac{(n-2\gamma)L}{4}}\big]\\
&\leq C\big[|x|^{-\frac{n+2\gamma}{2}}e^{-\eta|t-t_j|}\sum_{l\neq j}e^{-(2\gamma-\eta)|t-t_j^i|}e^{-\frac{n-2\gamma}{2}|t-t^i_l|}\\
&\qquad+|x|^{\gamma_1-2\gamma}|x|^{\gamma_1}e^{-2\gamma|t-t_j^i|}e^{-\frac{(n-2\gamma)L}{4}}\big]\\
&\leq C\big[|x|^{-\frac{n+2\gamma}{2}}e^{-\eta|t-t_j^i|}+
|x|^{\gamma_1-2\gamma}e^{-\sigma t_j^i}\big]e^{-\frac{(n-2\gamma)L}{4}(1+\xi)},
\end{split}\end{equation*}
if we choose $0<\eta<2\gamma$. On the other hand, if $|t-t^i_l|\leq \frac{L_1}{2}$ for some $l\neq j$, one has
\begin{equation*}\begin{split}
[(w_j^i)^{\beta-1}-\bar{u}_0^{\beta-1}]\frac{\partial w_j^i}{\partial r_j^i}&\leq C (w_l^i)^{\beta-1}\frac{\partial w_j^i}{\partial r_j^i}\\
&\leq C|x|^{-\frac{n+2\gamma}{2}}v_l^{\beta-1}v_j\\
&\leq C|x|^{-\frac{n+2\gamma}{2}}e^{-\eta|t-t^i_j|}e^{\eta|t-t^i_j|}
e^{-\frac{n-2\gamma}{2}|t-t_j|}e^{-2\gamma|t-t^i_l|}\\
&\leq C|x|^{-\frac{n+2\gamma}{2}}e^{-\eta|t-t^i_j|} e^{-\frac{(n-2\gamma)L}{4}(1+\xi)},
\end{split}\end{equation*}
if $\eta<\frac{n-2\gamma}{2} $ which is chosen small enough.
Combining the above two estimates yields, for $|x|\leq 1$,
\begin{equation*}
[(w_j^i)^{\beta-1}-\bar{u}_0^{\beta-1}]\frac{\partial w_j^i}{\partial r_j^i}|r_j^i|\leq C|x|^{-\frac{n+2\gamma}{2}}e^{-\tau t}e^{-\frac{(n-2\gamma)L}{4}(1+\xi)}+|x|^{\gamma_1-2\gamma}e^{-\frac{(n-2\gamma)
L}{4}(1+\xi)}.
\end{equation*}
Moreover,  recalling \eqref{para1}, one can get that for $\dist(x,\Sigma)\geq1$,
\begin{eqnarray*}
[(w_j^i)^{\beta-1}-\bar{u}_0^{\beta-1}]\frac{\partial w_j^i}{\partial r_j^i}|r_j^i|
\leq C|x|^{-(n+2\gamma)}e^{-\frac{(n-2\gamma)L}{4}(1+\xi)}e^{-\sigma t_j^i},
\end{eqnarray*}
and for $\dist(x,\Sigma)\leq 1$,
\begin{eqnarray*}
&&[(w_j^i)^{\beta-1}-\bar{u}_0^{\beta-1}]\frac{\partial w_j^i}{\partial r_j^i}|r_j^i|\\
&&\qquad\leq C[\dist(x,\Sigma)^{-\frac{n+2\gamma}{2}}\dist(x,\Sigma)^{\tau}
+\dist(x,\Sigma)^{\gamma_1-2\gamma}]
e^{-\frac{(n-2\gamma)L}{4}(1+\xi)}\\
&&\qquad\leq C\dist(x,\Sigma)^{\min\{-\frac{n-2\gamma}{2}+\tau, \gamma_1\}-2\gamma}e^{-\frac{(n-2\gamma)L}{4}(1+\xi)}
\end{eqnarray*}
for some $-\frac{n-2\gamma}{2}<\gamma_1<\min\{0,-\frac{n-2\gamma}{2}+2\gamma\}$ and $\tau$ small enough.

Similar estimates hold for $\frac{\partial }{\partial a_{j,l}^i}S(\bar{u})$. We conclude from the above that
\begin{equation*}\begin{split}
|S(\bar{u})-S(\bar{u}_0)|&\leq \sum_{i=1}^k\sum_{j=0}^\infty\sum_{l=1}^n \Big|[(w_j^i)^{\beta-1}-\bar{u}_0^{\beta-1}]\frac{\partial w_j^i}{\partial r_j^i}|r_j^i|
+[(w_j^i)^{\beta-1}-\bar{u}_0^{\beta-1}]\frac{\partial w_j^i}{\partial a_{j,l}^i}|a_{j,l}^i|\Big|\\
&\leq \left\{\begin{array}{l}
C|x|^{-(n+2\gamma)}e^{-\frac{(n-2\gamma)L}{4}(1+\xi)}, \ \text{ if }\dist(x,\Sigma)\geq 1,\\
C\dist(x,\Sigma)^{\min\{\gamma_1,-\frac{n-2\gamma}{2}+\tau\}-2\gamma}e^{-\frac{(n-2\gamma)L}{4}(1+\xi)}, \ \text{ if }\dist(x,\Sigma)< 1.
\end{array}
\right.
\end{split}\end{equation*}
Thus we have
\begin{equation*}
\|S(\bar{u})\|_{**}\leq Ce^{-\frac{(n-2\gamma)L}{4}(1+\xi)},
\end{equation*}
as desired.

\end{proof}

\noindent{\bf Proof of Lemma \ref{lemma301}.}
The proof relies on a standard finite-dimensional Lyapunov-Schmidt reduction.

\noindent \emph{Step 1: Preliminary calculations.} Multiply equation (\ref{eq302}) by ${Z}_{j',l'}^{i'}$ and integrate over $\R^n$; we have
\begin{equation}\label{Toeplitz}
\int_{\R^n}[(-\Delta)^\gamma \phi-c_{n,\gamma}\beta\bar{u}^{\beta-1}\phi]{Z}_{j',l'}^{i'}\,dx
=\int_{\R^n}h{Z}_{j',l'}^{i'}\,dx
+\sum_{i,j,l}c_{j,l}^i\int_{\R^n}(w_j^i)^{\beta-1}Z_{j,l}^i{Z}_{j',l'}^{i'}\,dx.
\end{equation}

By the orthogonality condition satisfied by $\phi$, we have that the left hand side of \eqref{Toeplitz} is
\begin{equation}\label{ecuacion11}\begin{split}
\int_{\R^n}[(-\Delta)^\gamma& \phi-c_{n,\gamma}\beta\bar{u}^{\beta-1}\phi]{Z}_{j',l'}^{i'}\,dx
=c_{n,\gamma}\beta\int_{\R^n}[(w_{j'}^{i'})^{\beta-1}-\bar{u}^{\beta-1}]\phi {Z}_{j',l'}^{i'}\,dx\\
&=\Big[\int_{B(p_{i'},1)}+\sum_{i\neq i'}\int_{B(p_i,1)}+\int_{\R^n \backslash \ \cup_{i=1}^k B(p_i,1)}\Big]=:I_1+I_2+I_3.
\end{split}\end{equation}
Without loss of generality, assume that $i'=1$ and $p_1=0$.  First we consider the case $l'=0$. Recalling the estimates for $Z_{j',0}^{i'}$ from \eqref{Z_0},
\begin{equation}\label{ecuacion12}\begin{split}
I_1&=\int_{B_1}[(w_{j'}^{i'})^{\beta-1}-\bar{u}^{\beta-1}]\phi {Z}_{j',l'}^{i'}\,dx\leq \|\phi\|_*\int_{B_1} \big|(w_{j'}^{i'})^{\beta-1}-\bar{u}^{\beta-1}\big| \, |x|^{\gamma_1}Z_{j',l'}^{i'}\,dx\\
&\leq \|\phi\|_*\int_{B_1}|x|^{\gamma_1-\frac{n+2\gamma}{2}}v_{j'}^{\beta-1}\sum_{j\neq j'}v_j\,dx\leq C\|\phi\|_*\int_0^\infty e^{-(\gamma_1+\frac{n-2\gamma}{2})t}v_{j'}^{\beta-1}\sum_{j\neq j'}v_j\,dt\\
&\leq C\|\phi\|_* \,e^{-\frac{(n-2\gamma)L}{4}(1+\xi)}e^{-(\gamma_1+\frac{n-2\gamma}{2})t^i_{j'}},
\end{split}\end{equation}
and notice that $\gamma_1>-\frac{n-2\gamma}{2}$. Next,
\begin{equation*}\begin{split}
I_2&=\sum_{i\neq 1}\int_{B(p_i,1)}[(w_{j'}^{i'})^{\beta-1}-\bar{u}^{\beta-1}]\phi Z_{j',l'}^{i'}\,dx\\
&\leq C\|\phi\|_*\sum_{i\neq 1}\int_{B(p_i,1)}\big|(w_{j'}^{i'})^{\beta-1}-\bar{u}^{\beta-1}\big|\, Z_{j',l'}^{i'}|x-p_i|^{\gamma_1}\,dx\\
&\leq C\|\phi\|_*\sum_{i\neq 1}\int_{B(p_i,1)}|x-p_i|^{\gamma_1-2\gamma}(\lambda_{j'}^{i'})^\frac{n-2\gamma}{2}(\sum_j v_j^i)\,dx\\
&\leq C\|\phi\|_*\,(\lambda_{j'}^{i'})^\frac{n-2\gamma}{2}e^{-(n+\gamma_1-2\gamma)\frac{L}{2}}\\
&\leq C\|\phi\|_*\,e^{\gamma_1t_{j'}-\frac{n-2\gamma}{2}L-\frac{\gamma_1}{2}L}
e^{-(\gamma_1+\frac{n-2\gamma}{2})t^i_{j'}}\\
&\leq C\|\phi\|_*\,e^{-\frac{(n-2\gamma)L}{4}(1+\xi)}e^{-(\gamma_1+\frac{n-2\gamma}{2})t^i_{j'}}
\end{split}\end{equation*}
and
\begin{equation*}\begin{split}
I_3&=\int_{\R^n / \cup_i B(p_i,1)}[(w_{j'}^{i'})^{\beta-1}-\bar{u}^{\beta-1}]\phi Z_{j',l'}^{i'}\,dx\\
&\leq C\|\phi\|_*\int_{\R^n \backslash \cup_iB(p_i,1)}|x|^{-(n-2\gamma)}|x|^{-(n+2\gamma)}
(\lambda_{j'}^{i'})^{\frac{n-2\gamma}{2}}e^{-\gamma L}\,dx\\
&\leq C\|\phi\|_*\,e^{\gamma_1t_{j'}-\gamma L}e^{-(\gamma_1+\frac{n-2\gamma}{2})t^i_{j'}}\\
&\leq C\|\phi\|_*\,e^{-\frac{(n-2\gamma)L}{4}(1+\xi)}e^{-(\gamma_1+\frac{n-2\gamma}{2})t^i_{j'}},
\end{split}\end{equation*}
where we have used $-\frac{n-2\gamma}{2}<\gamma_1<-\frac{n-2\gamma}{2}+2\gamma$.

On the other hand, for $l'=1,\cdots,n$, recalling from \eqref{Z_l} that $Z_{j',l'}^{i'}=O(|x-p_i|^{-\frac{n-2\gamma}{2}}(v_{j'}^{i'})^{1+\frac{2}{n-2\gamma}})$, then one can get similar estimates as above. In conclusion, one has
\begin{equation*}
\int_{\R^n}[(-\Delta)^\gamma \phi-c_{n,\gamma}\beta\bar{u}^{\beta-1}\phi]Z_{j',l'}^{i'}\,dx\leq C\|\phi\|_*e^{-\frac{(n-2\gamma)L}{4}(1+\xi)}e^{-(\gamma_1+\frac{n-2\gamma}{2}) t^i_{j'}},
\end{equation*}
for every $l'=0,\ldots,n$, which gives a good control of the left hand side of \eqref{Toeplitz}.

Now, for the first term in the right hand side of \eqref{Toeplitz},
\begin{equation*}\begin{split}
\int_{\R^n}hZ_{j',l'}^{i'}\,dx&\leq C\int_{\R^n \backslash \cup_i B(p_i)}\|h\|_{**}\,|x|^{-(n+2\gamma)}
|x|^{-(n-2\gamma)}e^{-\frac{(n-2\gamma)}{2}t^i_{j'}}\,dx\\
&+\int_{B(p_{i'})}\|h\|_{**}\,|x-p_{i'}|^{\gamma_1-2\gamma}|x-p_{i'}|^{-\frac{n-2\gamma}{2}}
[e^{-\frac{(n-2\gamma)}{2}t^i_{j'}}+e^{-(\frac{n-2\gamma}{2}+1)t^i_{j'}}]\,dx\\
&+\sum_{i\neq i'}\int_{B(p_i)}\|h\|_{**}\,|x-p_i|^{\gamma_1-2\gamma}e^{-\frac{n-2\gamma}{2}t^i_{j'}}\,dx\\
&\leq C\|h\|_{**}\,e^{-(\gamma_1+\frac{n-2\gamma}{2}) t^i_{j'}}.
\end{split}\end{equation*}

The next step is to isolate the term $c_{j,l}^i$  in \eqref{Toeplitz}, by inverting the matrix $\int_{\mathbb R^n} (w_j^i)^{\beta-1}Z_{j,l}^i Z_{j',l'}^{i'}\,dx.$ For this, recall the orthogonality estimates from \eqref{orthogonality1}-\eqref{orthogonality2}, which yield, for all $l=0,\ldots,n$,
\begin{equation*}
\int_{\R^n}(w_j^i)^{\beta-1}Z_{j,l}^iZ_{j,l'}^{i}\,dx=C_0\delta_{l,l'}   \quad \text{and}\quad \int_{\R^n}(w_j^i)^{\beta-1}Z_{j,l}^iZ_{j',l'}^{i}\,dx=O(e^{-\frac{n-2\gamma}{2}
|t^i_j-t^i_{j'}|})\ \text{if}
\ j\neq j',
\end{equation*}
plus a tiny error.
Then using Lemma A.6 in \cite{m} for the inversion of a Toepliz-type operator, one has from \eqref{Toeplitz} that
\begin{equation*}\begin{split}
|c_{j,l}^i|&\leq C[e^{-\frac{(n-2\gamma)L}{4}(1+\xi)}\|\phi\|_*
+\|h\|_{**}]e^{-(\gamma_1+\frac{n-2\gamma}{2})t_j^i}\\
&+C\sum_{j'\neq j}[e^{-\frac{(n-2\gamma)L}{4}(1+\xi)}\|\phi\|_*
+\|h\|_{**}]e^{-\frac{n-2\gamma}{2}(1+o(1))|t_j-t_{j'}|}e^{-(\gamma_1+\frac{n-2\gamma}{2})t_{j'}^i}\\
&\leq C[e^{-\frac{(n-2\gamma)L}{4}(1+\xi)}\|\phi\|_*+\|h\|_{**}]
e^{-(\gamma_1+\frac{n-2\gamma}{2})t_j^i}.
\end{split}\end{equation*}
From the estimates for $Z_{j,l}^i$ from \eqref{Z_0}-\eqref{Z_l} and the previous bound for $c_{j,l}^i$ one can check that in $B_1(p_i)$,
\begin{equation*}\begin{split}
|c_{j,l}^i(w_j^i)^{\beta-1}Z_{j,l}^i|&\leq C|c_{j,l}^i|\cdot|x-p_i|^{-\frac{n+2\gamma}{2}}e^{-\frac{n+2\gamma}{2}|t^i-t_j^i|}\\
&\leq C |x-p_i|^{\gamma_1-2\gamma}[e^{-\frac{(n-2\gamma)L}{4}(1+\xi)}\|\phi\|_*+\|h\|_{**}]\\
&\quad\cdot\, e^{-(\gamma_1+\frac{n-2\gamma}{2})t^i_j}
e^{(\gamma_1+\frac{n-2\gamma}{2})t^i}e^{-\frac{n+2\gamma}{2}|t^i-t^i_j|}\\
&\leq C|x-p_i|^{\gamma_1-2\gamma}[e^{-\frac{(n-2\gamma)L}{4}(1+\xi)}\|\phi\|_*+\|h\|_{**}]
e^{-\sigma|t^i-t^i_j|}
\end{split}\end{equation*}
for some $\sigma>0$.

For $x\in \R^n \backslash\cup_i B_1(p_i)$, one has
\begin{equation*}\begin{split}
|c_{j,l}^i(w_j^i)^{\beta-1}Z_{j,l}^i|&\leq C(\lambda_j^i)^{\frac{n-2\gamma}{2}}|x|^{-(n+2\gamma)}|c_{j,l}^i|\\
&\leq C|x|^{-(n+2\gamma)}[e^{-\frac{(n-2\gamma)L}{4}(1+\xi)}\|\phi\|_*+\|h\|_{**}]e^{-\sigma t_j^i}.
\end{split}\end{equation*}
Combining the above two estimates yields
\begin{equation}\label{eq308}
\big\|\sum c_{j,l}^i(w_j^i)^{\beta-1}Z_{j,l}^i\big\|_{**}\leq C\big[e^{-\frac{(n-2\gamma)L}{4}(1+\xi)}\|\phi\|_*+\|h\|_{**}\big].
\end{equation}

\medskip

\noindent\emph{Step 2: A priori estimates.} We are going to prove the a priori estimate \eqref{eq305} by a contradiction argument. First let us recall the problem we are going to consider:
\begin{equation}\label{ecuacion10}
\left\{\begin{split}
&(-\Delta)^\gamma \phi-c_{n,\gamma}\beta \bar{u}^{\beta-1}\phi=\bar{h},\\
&\int_{\R^n}\phi (w_j^i)^{\beta-1}Z_{j,l}^i\,dx=0,\quad i=1,\ldots,k, \ j=0,1,\ldots, \ l=0,\ldots,n,
\end{split}
\right.
\end{equation}
where have we denoted $\bar{h}:=h+\sum c_{j,l}^i(w_j^i)^{\beta-1}Z_{j,l}^i$, and which satisfies, by (\ref{eq308}), that
\begin{equation*}
\|\bar{h}\|_{**}\leq C(\|h\|_{**}+o(1)\|\phi\|_{*}).
\end{equation*}
We are going to prove that
\begin{equation}\label{objective}
\|\phi\|_*\leq C\|\bar{h}\|_{**},
\end{equation}
for where \eqref{eq305} follows immediately.

Assume that there exist sequences $\{L_i^{(n)}\}$ with $L_i^{(n)}\to \infty$, $\{r_j^{i,(n)}, \ a_j^{i,(n)}\}$, $\{h^{(n)}\}$, $\{c_{j,l}^{i,(n)}\}$ and the corresponding solution $\{\phi^{(n)}\}$ such that
\begin{equation}\label{eq311}
\|h^{(n)}\|_{**}\to 0, \quad \|\phi^{(n)}\|_*=1.
\end{equation}
In the following we will drop the index $n$ if needed. First by the Green's representation formula for the first equation in \eqref{ecuacion10} we have
\begin{equation}\label{Green-representation}
\phi(x)=\int_{\R^n}(c_{n,\gamma}\bar{u}^{\beta-1}\phi+\bar{h})(y)G(x,y)\,dy=:I_1+I_2,
\end{equation}
where $G$ is the Green's function for the fractional Laplacian $(-\Delta)^\gamma$, given by (\cite{Caffarelli-Silvestre})
\begin{equation*}
G(x,y)=C|x-y|^{-(n-2\gamma)}.
\end{equation*}
First we consider the region $\{\dist(x,\Sigma)\geq 1\}$. Here, for $I_2$,
\begin{equation*}\begin{split}
I_2&=\int_{\R^n}\bar{h}(y)G(x,y)\,dy\\
&\leq \left[\int_{\{\dist(y,\Sigma)\leq 1\}}+\int_{\{1<\dist(y,\Sigma)<\frac{|x|}{2}\}}+\int_{\{\frac{|x|}{2}<\dist(y,\Sigma)<2|x|\}}+
\int_{\{\dist(y,\Sigma)\geq 2|x|\}}\right]\bar{h}(y)G(x,y)\,dy\\
&=:I_{21}+I_{22}+I_{23}+I_{24},
\end{split}\end{equation*}
%(\cite{Caffarelli-Silvestre,fqt,fls})
%\begin{equation*}
%G(x,y)\leq \left\{\begin{array}{l}
%C|x-y|^{-(n-2\gamma)} \mbox{ for }|x-y|\leq 1,\\
%C|x-y|^{-(n+2\gamma)}\mbox{ for }|x-y|\geq 1,
%\end{array}
%\right.
%\end{equation*}
one has
\begin{equation*}\begin{split}
I_{21}&\leq \int_{\{\dist(y,\Sigma)<1\}}\frac{1}{|x-y|^{n-2\gamma}}
\|\bar{h}\|_{**}|y|^{\gamma_1-2\gamma}\,dy
\leq C\|\bar{h}\|_{**}|x|^{-(n-2\gamma)},\\
I_{22}
&\leq \int_{\{1<\dist(y,\Sigma)<\frac{|x|}{2}\}}\frac{1}{|x-y|^{n-2\gamma}}\|\bar{h}\|_{**}
|y|^{-(n+2\gamma)}\,dy\\
&\leq C\|\bar{h}\|_{**}|x|^{-(n-2\gamma)}\int_{\{1<\dist(y,\Sigma)<\frac{|x|}{2}\}}
|y|^{-(n+2\gamma)}\,dy\leq C\|\bar{h}\|_{**}|x|^{-(n-2\gamma)},\\
I_{23}&\leq \int_{\frac{|x|}{2}<\dist(y,\Sigma)<2|x|}\frac{1}{|x-y|^{n-2\gamma}}\|\bar{h}\|_{**}
|y|^{-(n+2\gamma)}\,dy\\
&\leq \|\bar{h}\|_{**}|x|^{-(n+2\gamma)}\int_{\{x-y|<\frac{5|x|}{2}\}}
\frac{1}{|x-y|^{n-2\gamma}}\,dy\leq C\|\bar{h}\|_{**}|x|^{-(n-2\gamma)},\\
I_{24}&\leq \int_{\{\dist(y,\Sigma)\geq 2|x|\}}\frac{1}{|x-y|^{n-2\gamma}}\|\bar{h}\|_{**}|y|^{-(n+2\gamma)}\,dy\leq C\|\bar{h}\|_{**}|x|^{-(n-2\gamma)}.
\end{split}\end{equation*}
Putting all together,
\begin{equation}\label{eq309}
I_2\leq C\|\bar{h}\|_{**}|x|^{-(n-2\gamma)}.
\end{equation}
Next for $I_1$,
\begin{equation*}\begin{split}
I_1&=\left[\int_{\{\dist(y,\Sigma)\leq 1\}}+\int_{\{\dist(y,\Sigma)\geq 1\}}\right]c_{n,\gamma}\beta\bar{u}^{\beta-1}\phi\, G(x,y)\,dy=:I_{11}+I_{12}.
\end{split}\end{equation*}
Since for $\dist(y,\Sigma)\geq 1$ it holds that $\bar{u}=O(e^{-\frac{n-2\gamma}{4}L})|y|^{-(n-2\gamma)}$ (recall \eqref{estimacion-lejos}), then
\begin{eqnarray*}
I_{12}\leq Ce^{-\gamma L}\|\phi\|_{*}\int_{\{\dist(y,\Sigma)\geq 1\}} |y|^{-(n+2\gamma)}G(x,y)\,dy,
\end{eqnarray*}
and similar to the estimate above we get that
\begin{equation*}
I_{12}\leq o(1)\|\phi\|_*|x|^{-(n-2\gamma)}.
\end{equation*}
Moreover,
\begin{equation*}\begin{split}
I_{11}&\leq \sum_{i=1}^k\int_{\{|y-p_i|\leq 1\}}|y-p_i|^{-2\gamma}(\sum_{j=0}^\infty v_j^i)^{\beta-1}\|\phi\|_*|y-p_i|^{\gamma_1}|x-y|^{-(n-2\gamma)}\,dy\\
&\leq C\|\phi\|_* |x|^{-(n-2\gamma)}\int_{\{|y-p_i|<1\}}|y-p_i|^{\gamma_1-2\gamma}(
\sum_{j=0}^\infty v_j^i)^{\beta-1}\,dy\\
&\leq C\|\phi\|_* |x|^{-(n-2\gamma)}\int_0^\infty e^{-(n+\gamma_1-2\gamma)t}(\sum_{j=0}^\infty v_j^i)^{\beta-1}\,dt\\
&\leq C e^{-(n+\gamma_1-2\gamma)\frac{L}{2}}\|\phi\|_* |x|^{-(n-2\gamma)}.
\end{split}\end{equation*}
Since $\gamma_1>-\frac{n-2\gamma}{2}$, by the above estimates one has
\begin{equation}\label{eq310}
I_1\leq o(1)\|\phi\|_*|x|^{-(n-2\gamma)}.
\end{equation}
Summarizing, from (\ref{eq309}) and (\ref{eq310}) we obtain that, for $\dist(x,\Sigma)\geq 1$,
\begin{equation*}
\sup_{\dist(x,\Sigma)\geq 1}\left\{|x|^{n-2\gamma}|\phi(x)|\right\}\leq C(\|\bar{h}\|_{**}+o(1)\|\phi\|_*)\to 0 \quad \text{as}\quad L\to\infty,
\end{equation*}
by our initial hypothesis \eqref{eq311}. Moreover, because of the same reason,
 we know that there exists $p_i$ such that
\begin{equation}\label{eq312}
\sup_{|x-p_i|\leq 1}\left\{|x-p_i|^{-\gamma_1}\phi(x)|\right\}\geq \tfrac{1}{2}.
\end{equation}

The next step is to consider the region $\{|x-p_i|\leq 1\}$. In order to simplify the notation, we assume that $p_i=0, \ |x|<1$. Again, we use Green's representation formula \eqref{Green-representation}, and we estimate both integrals $I_1$, $I_2$. On the one hand,
\begin{equation*}\begin{split}
I_2&=\left[\int_{\{|y|\geq 1\}}+\int_{\{|y|<\frac{|x|}{2}\}}
+\int_{\{\frac{|x|}{2}<|y|<2|x|\}}
+\int_{\{2|x|<|y|<1\}}\right]G(x,y)\bar{h}\,dy=:I_{21}+I_{22}+I_{23}+I_{24},
\end{split}\end{equation*}
where
\begin{equation*}\begin{split}
I_{21}&\leq C\int_{\{|y|\geq 1\}}\frac{1}{|x-y|^{n-2\gamma}}\|\bar{h}\|_{**}|y|^{-(n+2\gamma)}\,dy\leq C\|\bar{h}\|_{**}|x|^{\gamma_1},\\
I_{22}&\leq \int_{\{|y|<\frac{|x|}{2}\}}\frac{1}{|x-y|^{n-2\gamma}}
\|\bar{h}\|_{**}|y|^{\gamma_1-2\gamma}\,dy
\leq C\|\bar{h}\|_{**}\int_{\{|y|<\frac{|x|}{2}\}}\frac{1}{|x|^{n-2\gamma}}
|y|^{\gamma_1-2\gamma}\,dy\leq C\|\bar{h}\|_{**}|x|^{\gamma_1},\\
I_{23}&\leq \int_{\{\frac{|x|}{2}<|y|<2|x|\}}\frac{1}{|x-y|^{n-2\gamma}}
\|\bar{h}\|_{**}|y|^{\gamma_1-2\gamma}\,dy
\\
&\leq C\|\bar{h}\|_{**}|x|^{\gamma_1-2\gamma}\int_{\{|x-y|<\frac{5|x|}{2}\}}\frac{1}{
|x-y|^{n-2\gamma}}\,dy\leq C\|\bar{h}\|_{**}|x|^{\gamma_1},\\
I_{24}&\leq \int_{\{2|x|<|y|<2\}}\frac{1}{|x-y|^{n-2\gamma}}\|\bar{h}\|_{**}|y|^{\gamma_1-2\gamma}\,dy
\leq C\|\bar{h}\|_{**}\int_{\{2|x|<|y|<2\}}|y|^{\gamma_1-2\gamma-n+2\gamma}\,dy\leq C\|\bar{h}\|_{**}|x|^{\gamma_1}.
\end{split}\end{equation*}
Thus one has
\begin{equation*}
I_2\leq C\|\bar{h}\|_{**}|x|^{\gamma_1}.
\end{equation*}
On the other hand, for $I_1$,
\begin{equation*}\begin{split}
I_1&=\left[\int_{\{|y|>1\}}+\int_{\{|y|<1\}}\right]c_{n,\gamma}\beta\bar{u}^{\beta-1}\phi\, G(x,y)\,dy=:I_{11}+I_{12}.
\end{split}\end{equation*}
Similar to the estimates for $\bar{h}$,
\begin{equation*}\begin{split}
I_{11}&\leq C\int_{\{\dist(y,\Sigma)\geq 1\}}e^{-\gamma L}|y|^{-4\gamma}\frac{1}{|x-y|^{n-2\gamma}}\|\phi\|_*|y|^{-(n-2\gamma)}\,dy\leq o(1)\|\phi\|_*|x|^{\gamma_1}.
\end{split}\end{equation*}

The final step is to estimate $I_{12}$. For this we consider $\phi$ in the region $A_j:=\sqrt{\lambda_{j+1}^i\lambda_j^i}<|x|<\sqrt{\lambda_j^i\lambda_{j-1}^i}$, and define a scaled function $\tilde{\phi}_j(\tilde{x})=(\lambda_j^i)^{-\gamma_1}\phi(\lambda_j^i \tilde x)$ defined in the region $\tilde{A}_j=\frac{A_j}{\lambda_j^i}\to (0,\infty)$ as $n\to\infty$. Then $\tilde{\phi}_j$ will satisfy the following equation
\begin{equation*}\left\{\begin{split}
&(-\Delta)^\gamma \tilde{\phi}_j-c_{n,\gamma}\beta \Big(\frac{1}{1+|\tilde{x}|^2}\Big)^{2\gamma}(1+o(1))\tilde{\phi}_j=
(\lambda_j^i)^{2\gamma-\gamma_1}\bar{h}(\lambda_j^i \tilde{x})\quad\text{in }\tilde{A}_j,\\
&\int_{\R^n} \tilde{\phi}_j (w_j^i)^{\beta-1}(\lambda_j^i \tilde{x})Z_{j,l}^i(\lambda_j^i \tilde{x})\,d\tilde{x}=0.
\end{split}
\right.
\end{equation*}
Since $|\bar{h}|\leq C\|\bar{h}\|_{**}|\lambda_j^i \tilde{x}|^{\gamma_1-2\gamma}$ as $n \to \infty$, $\tilde{\phi}_j\to \bar{\phi}$ in any compact set $\frac{1}{R}\leq |\tilde{x}|\leq R$ for $R$ large enough (to be determined later), where $\bar{\phi} $ is a solution of the following equation
\begin{equation*}
\left\{\begin{split}
&(-\Delta)^\gamma \bar{\phi}-c_{n,\gamma}\beta w^{\beta-1}\bar{\phi}=0,\\
&\int_{\R^n} \bar{\phi}w^{\beta-1}Z_l\,dx=0,
\end{split}
\right.
\end{equation*}
where $w$ is the standard bubble solution and $Z_l, \ l=0,\cdots,n$, are the corresponding kernels mentioned in Lemma \ref{lemma:non-degenerate}. By the non-degeneracy of the bubble, one has $\bar{\phi}=0$, i.e. $\tilde{\phi}_j\to 0$ in $\frac{1}{R}<|\tilde{x}|<R$. If we consider the original $\phi$, this is equivalent to that $|x|^{-\gamma_1}\phi(x)\to 0$ in $\cup_j \{\frac{\lambda_j^i}{R}<|x|<R\lambda_j^i\}$ as $n\to \infty$. Using this result, we now consider $I_{12}$:
\begin{equation*}\begin{split}
I_{12}&=\int_{\{|y|<1\}}c_{n,\gamma}\beta \bar{u}^{\beta-1}\phi \,G(x,y)\,dy\\
&\leq \sum_j\left[\int_{\{\frac{\lambda_j^i}{R}<|x|<R\lambda_j^i\}}+\int_{\{|y|<1\} \backslash\cup_j\{\frac{\lambda_j^i}{R}<|x|<\lambda_j^i R\}}\right]\bar{u}^{\beta-1}\phi\, G(x,y)\,dy=:I_{121}+I_{122}.
\end{split}\end{equation*}
Recalling \eqref{eq301}, we have that in $\{|y|<1\}$,  $\bar{u}=|y|^{-\frac{n-2\gamma}{2}}(\sum_{j=0}^\infty v_j^i)(1+o(1))$. Then in the region $\{|y|<1\} \backslash \cup_j\{\frac{\lambda_j^i}{R}<|x|<\lambda_j^i R\}$, one has $\sum_j v_j^i\leq Ce^{-\frac{(n-2\gamma)}{2}R}$ which can be small enough choosing $R$ large enough but independent of $n$. Using this estimate we can assert that
\begin{equation*}\begin{split}
I_{122}&\leq Ce^{-2R}\int_{|y|<1}|y|^{-2\gamma}\|\phi\|_*|y|^{\gamma_1}
\frac{1}{|x-y|^{n-2\gamma}}\,dy\leq Ce^{-2R}|x|^{\gamma_1}.
\end{split}\end{equation*}
In addition, by the previous argument we know that $|x|^{-\gamma_1}\phi(x)\to 0 $ in $\cup_j \{\frac{\lambda_j^i}{R}<|x|<R\lambda_j^i\}$, and one has
\begin{equation*}\begin{split}
I_{121}&\leq C\sum_j\int_{\{\frac{\lambda_j^i}{R}<|y|<R\lambda_j^i\}} |\phi||y|^{-\gamma_1}\frac{|y|^{\gamma_1-2\gamma}(\sum_j v_j^i)^{\beta-1}}{|x-y|^{n-2\gamma}}\,dy\\
&\leq o(1)\int_{\{|y|\leq 1\}}\frac{|y|^{\gamma_1-2\gamma}}{|x-y|^{n-2\gamma}}\,dy\leq o(1)|x|^{\gamma_1}.
\end{split}\end{equation*}
Combining all the above estimates yields that in the set $\{|x|<1\}$ we must have $|x|^{-\gamma_1}\phi(x)=o(1)$ as $n\to \infty$, which is a contradiction to \eqref{eq312}. This completes the proof of the a priori estimate \eqref{objective}, as desired.\\

\noindent\emph{Step 3: Existence and uniqueness.} Consider the space
\begin{equation*}
\mathcal{H}=\Big\{u\in H^\gamma(\R^n),  \ \int_{\mathbb R^n}u(w_j^i)^{\beta-1}Z_{j,l}^i\,dx=0 \quad \mbox{for all }i,j,l \Big\}.
\end{equation*}
Notice that the problem (\ref{eq302}) in $\phi$ gets rewritten as
\begin{equation}\label{eq313}
\phi+K(\phi)=\bar{h} \mbox{ in }\mathcal{H},
\end{equation}
where $\bar{h}$ is defined by duality and $K:\mathcal{H}\to \mathcal{H}$ is a linear compact operator. Using Fredholm's alternative, showing that equation \eqref{eq313} has a unique solution for each $\bar{h}$ is equivalent to showing that the equation has a unique solution for $\bar{h}=0$, which in turn follows from the previous a priori estimate. This concludes the proof of Lemma \ref{lemma301}.

\qed

\noindent{\bf Proof of Proposition \ref{proposition301}. }The proof relies on the contraction mapping in the above weighted norms. We set
\begin{equation*}
S(\bar{u})=(-\Delta)^\gamma \bar{u}-c_{n,\gamma}\bar{u}^{\beta},
\end{equation*}
and also define the linear operator
\begin{equation*}
\mathbb L(\phi)=(-\Delta)^\gamma \phi-c_{n,\gamma}\beta\bar{u}^{\beta-1}\phi.
\end{equation*}
We have that $\bar{u}+\phi$, $\phi \in \mathcal C_*$ solves equation (\ref{eq303}) if and only if $\phi$ satisfies
\begin{equation}\label{equation-contraction}
\phi=G(\phi)
\end{equation}
where
\begin{equation*}
G(\phi):=\mathbb L^{-1}(S(\bar{u}))+c_{n,\gamma}\mathbb L^{-1}(N(\phi)).
\end{equation*}
Here we have defined
$$N(\phi):=(\bar{u}+\phi)^{\beta}-\bar{u}^{\beta}-\beta\bar{u}^{\beta-1}\phi.$$ Also, by $\mathbb L^{-1}$, we are denoting the linear operator which, according to Lemma \ref{lemma301},  associates with $h\in \mathcal C_{**}$ the function $\phi \in \mathcal  C_*$ solving (\ref{eq302}).

We find a solution for \eqref{equation-contraction} by a standard contraction mapping argument. First by the definition of $G$, one has
\begin{equation*}
\|G(\phi)\|_{*}\leq C \left(\|S(\bar u)\|_{**}+\|N(\phi)\|_{**}\right).
\end{equation*}
Fixing a large $C_1>0$, we define the set
\begin{equation*}
B_{C_1}=\big\{\phi\in \mathcal C_*\,\,:\,\, \|\phi\|_*\leq C_1e^{-\frac{(n-2\gamma)L}{4}(1+\xi)}, \int_{\R^n}\phi (w_j^i)^{\beta-1}Z_{j,l}^i\,dx\,=0, \forall i,j,l\big\}.
\end{equation*}
Note that
\begin{equation*}
|(\bar{u}+\phi)^\beta-\bar{u}^\beta-\beta\bar{u}^{\beta-1}\phi|\leq C\left\{\begin{array}{l}
\bar{u}^{\beta-2}\phi^2,\quad  \text{if } |\bar{u}|\geq\frac{1}{4}\phi,\\
\phi^\beta,  \quad \text{if }|\bar{u}|\leq\frac{1}{4}\phi.
\end{array}
\right.
\end{equation*}
Now, let $\phi \in B_{C_1}$.
By our construction, we have that if $\dist(x,\Sigma) <1$,
\begin{equation*}\begin{split}
|N(\phi)|&\leq C(\bar{u}^{\beta-2}\phi^2+\phi^\beta)\\
&\leq C\sum_{i=1}^k\big[\|\phi\|_*^2\bar{u}^{\beta-2}|x-p_i|^{2\gamma_1}+\|\phi\|_*^{\beta}|x-p_i|^{\beta\gamma_1}\big]\\
&\leq C\sum_{i=1}^k\big[\|\phi\|_*^2|x-p_i|^{\gamma_1-2\gamma}|x-p_i|^{\gamma_1+\frac{n-2\gamma}{2}}+\|\phi\|_*^\beta|x-p_i|^{\gamma_1-2\gamma}|x-p_i|^{\beta \gamma_1-\gamma_1+2\gamma}\big]\\
&\leq C\sum_{i=1}^k\big[\|\phi\|_*^2+\|\phi\|_*^\beta\big]|x-p_i|^{\gamma_1-2\gamma},
\end{split}\end{equation*}
and for $\dist(x,\Sigma)\geq 1$,
\begin{equation*}\begin{split}
|N(\phi)|&\leq C[\|\phi\|_*^2\bar{u}^{\beta-2}|x|^{-2(n-2\gamma)}+\|\phi\|_*^\beta |x|^{-\beta(n-2\gamma)}]\\
&\leq C|x|^{-(n+2\gamma)}[e^{-(\beta-2)\frac{(n-2\gamma)L}{4}}\|\phi\|_*^2+\|\phi\|_*^\beta].
\end{split}\end{equation*}
Combining the above two estimates, one has
\begin{equation*}
\|N(\phi)\|_{**}\leq C[e^{-(\beta-2)\frac{(n-2\gamma)L}{4}}\|\phi\|_*^2+\|\phi\|_*^\beta]\leq o(1)\|\phi\|_*.
\end{equation*}

Now we consider two functions $\phi_1,\phi_2\in B_{C_1}$, it is easy to see that for $L$ large,
\begin{equation*}
\|N(\phi_1)-N(\phi_2)\|_{**}\leq o(1)\|\phi_1-\phi_2\|_*.
\end{equation*}
Therefore,  by the above estimates for $N(\phi)$ and \eqref{estimate-S}, $G$ is a contraction mapping in $B_{C_1}$, thus it has a fixed point in this set. This completes the proof of the Proposition.
\qed

\section{Estimates on the coefficients $c_{j,l}^i$} \label{sec4}

In this section we prove some estimates related to the coefficients $c_{j,l}^i $ obtained in the last section, first in the special case of the configuration $(r_j^i, a_j^i)=(0, 0)$ and then for a general configuration of parameters satisfying (\ref{parar})-(\ref{para2}). These are studied in subsections \ref{sec401} and \ref{sec402} respectively. Later on, in the next section, we study also the derivative with respect to a variation of the parameters.

\subsection{Estimates on the $c_{j,l}^i$ for $(a_j^i,r_j^i)=(0,0)$}\label{sec401}

In this subsection, we prove the decay of $c_{j,l}^i$ when the parameters $(a_j^i, r_j^i)=(0, 0)$. We denote $\bar u_0$ to be the approximate solution and $\phi$ the perturbation function found in Proposition \ref{proposition301} for this particular case. Define the numbers $\bar\beta_{j,l}^i$ as
\begin{equation*}
\bar\beta_{j,l}^i=\int_{\R^n}[(-\Delta )^\gamma(\bar{u}_0+\phi)-c_{n,\gamma}(\bar{u}_0+\phi)^\beta]Z_{j,l}^i\,dx.
\end{equation*}
Then we have the following estimates on $\bar\beta_{j,l}^i$:
\begin{lemma}\label{lemma401}
Given $\{R^i\}$ satisfying \eqref{parar} and let $\bar{u}_0$ be the function defined in \eqref{eq301} for the parameters $(r_j^i,a_j^i)=(0,0)$. Let $\phi$ and $(c_{j,l}^i)$ be given in Proposition \ref{proposition301}. Then the coefficients $\bar\beta_{j,l}^i$ satisfy
\begin{equation*}\begin{split}
&\bar\beta_{0,0}^i=-c_{n,\gamma}q_i\Big[A_2\sum_{i'\neq i}|p_{i'}-p_i|^{-(n-2\gamma)}(R^iR^{i'})^{\frac{n-2\gamma}{2}}q_{i'}-q_i\Big]e^{-\frac{(n-2\gamma)L}{2}}(1+o(1))
%&&\qquad\qquad
+O(e^{-\frac{(n-2\gamma)L}{2}(1+\xi)}),\\
&\bar\beta_{0,l}^i=c_{n,\gamma}\lambda_0^i\Big[A_3\sum_{i'\neq i}\frac{(p_{i'}-p_i)_l}{|p_{i'}-p_i|^{n-2\gamma+2}}(R^iR^{i'})^{\frac{n-2\gamma}{2}}
q_{i'}q_ie^{-\frac{n-2\gamma}{2}L}+O(e^{-\frac{(n-2\gamma)L}{2}(1+\xi)})\Big]\mbox{ for }l=1,\cdots,n.\\
\end{split}\end{equation*}
For $j\geq 1$, we have
\begin{eqnarray*}
&&\bar\beta_{j,0}^i=O(e^{-\frac{(n-2\gamma)L}{2}(1+\xi)}e^{-\sigma t^i_j}),\\
&&\bar\beta_{j,l}^i=O(e^{-\frac{(n-2\gamma)L}{2}(1+\xi)}e^{-(1+\sigma)t^i_j}),
\end{eqnarray*}
where $A_2>0, A_3<0$ are two constants independent of $L$ and $\sigma=\min\{\gamma_1+\frac{n-2\gamma}{2}, \frac{n-2\gamma}{4}\}$ independent of $L$ large.
\end{lemma}

\begin{proof}

With some manipulation and the orthogonality condition satisfied by $\phi$, we find that
\begin{equation*}
\bar\beta_{j,l}^i=\bar\beta_{j,l,1}^i+\bar\beta_{j,l,2}^i+\bar\beta_{j,l,3}^i,
\end{equation*}
for
\begin{equation*}\begin{split}
&\bar\beta_{j,l,1}^i=\int_{\R^n}[(-\Delta)^\gamma \bar{u}_0-c_{n,\gamma}(\bar{u}_0)^\beta]Z_{j,l}^i\,dx,\\
&\bar\beta_{j,l,2}=\int_{\mathbb R^n}\mathbb L_0(\phi)Z_{j,l}^{i} \,dx= -c_{n,\gamma}\beta\int_{\R^n}[(\bar{u}_0)^{\beta-1}-(w_{j}^i)^{\beta-1}]
Z_{j,l}^i\phi\,dx,\\
&\bar\beta_{j,l,3}^i=-c_{n,\gamma}\int_{\R^n }[(\bar{u}_0+\phi)^\beta-(\bar{u}_0)^\beta-\beta (\bar{u}_0)^{\beta-1}\phi]Z_{j,l}^i\,dx,
\end{split}
\end{equation*}
where we have defined $\mathbb L_0(\phi)=(-\Delta)^\gamma \phi-c_{n,\gamma}\beta (\bar u_0)^{\beta-1} \phi$.\\

\noindent\emph{Step 1: Estimate for $\bar\beta_{j,l,2}^i$ and $\bar\beta_{j,l,3}^i$}.
By the estimates in the proof of Lemma \ref{lemma301} and the bounds satisfied by $\phi$, one has
\begin{equation*}
|\bar\beta_{j,l,2}^i|\leq Ce^{-\frac{(n-2\gamma)L}{2}(1+\xi)}e^{-(\gamma_1+\frac{n-2\gamma}{2}) t_j}.
\end{equation*}
In addition,
\begin{equation*}\begin{split}
-\bar\beta_{j,l,3}^i&=\int_{\R^n}N(\phi)Z_{j,l}^i\,dx\\
&=\int_{B(p_i,1)}+\sum_{i'\neq i}\int_{B(p_{i'},1)}+\int_{\R^n \backslash \cup_i B(p_i,1)}=:J_1+J_2+J_3.
\end{split}\end{equation*}
We estimate this expression term by term. For $l=0$, one has
%\begin{equation*}
%Z_{j,0}^i=|x-p_i|^{-\frac{(n-2\gamma)}{2}}v'(-\log |x-p_i|+\log \lambda_j^i)
%\end{equation*}
%and for $l=1,\cdots,n$,
%\begin{equation*}
%Z_{j,l}^i=(\frac{\lambda_j^i}{(\lambda_j^i)^2+|x-p_i|^2})^{\frac{n-2\gamma}{2}}\frac{x_l}{|\lambda_j^i|^2+|x-p_i|^2}
%\end{equation*}
%So for $l=0$
\begin{equation*}\begin{split}
J_1&=\int_{B(p_i,1)}N(\phi)Z_{j,l}^i\,dx\leq C\int_{B(p_i,1)}\|\phi\|_*^2\,|x-p_i|^{2\gamma_1}\bar{u}^{\beta-2}Z_{j,l}^i\,dx\\
&\leq C\|\phi\|_*^2\int_{B(p_i,1)}|x-p_i|^{2\gamma_1-\frac{6\gamma-n}{2}}|x-p_i|^{-\frac{n-2\gamma}{2}}|v'_j|\,dx\\
&\leq Ce^{-\frac{(n-2\gamma)L}{2}(1+\xi)}e^{-\min\{2(\gamma_1+\frac{n-2\gamma}{2}), \frac{n-2\gamma}{2}\} t_j^i},
\end{split}\end{equation*}
recalling \eqref{decay-phi}. %and for $l=1,\cdots,n$
%\begin{eqnarray*}
%J_1&&=\int_{B(p_i,1)}N(\phi)Z_{j,l}^idx\\
%&&\leq C\int_{B(p_i,1)}\|\phi\|_*^2|x-p_i|^{2\gamma_1}\bar{u}^{\beta-2}Z_{j,l}^idx\\
%&&\leq C\|\phi\|_*^2\int_{B(p_i,1)}|x-p_i|^{2\gamma_1}(\frac{\lambda_j^i}{(\lambda_j^i)^2+|x-p_i|^2})^{\frac{(n-2\gamma)}{2}(\beta-1)}\frac{|x|}{|\lambda_j^i|^2+|x-p_i|^2}dx\\
%&&\leq C\|\phi\|_*^2\int_{B(\frac{1}{\lambda_j^i})}|\lambda_j^i|^{n+2\gamma_1-2\gamma-1}\frac{|x|^{2\gamma_1+1}}{(1+|x|^2)^{2\gamma+1}}dx\\
%&&\leq C\|\phi\|_*^2|\lambda_j^i|^{n+2\gamma_1-2\gamma-1}\int_0^{\frac{1}{\lambda_j^i}}\frac{r^{n+2\gamma_1}}{(1+r^2)^{2\gamma+1}}dr\\
%&&\leq C\|\phi\|_*^2|\lambda_j^i|^{2\gamma}\\
%&&\leq Ce^{-\frac{(n-2\gamma)L}{2}(1+\xi)}e^{-\sigma t_j}
%\end{eqnarray*}
Similarly,
\begin{equation*}\begin{split}
J_2&=\sum_{i'\neq i}\int_{B(p_{i'},1)}N(\phi)Z_{j,l}^i\,dx\leq C\|\phi\|_*^2\sum_{i'\neq i}\int_{B(p_{i'},1)}|x-p_{i'}|^{2\gamma_1}|x-p_{i'}|^{\frac{6\gamma-n}{2}}Z_{j,l}^i\,dx\\
&\leq C\|\phi\|_*^2\,(\lambda_j^i)^{\frac{n-2\gamma}{2}}\leq Ce^{-\frac{(n-2\gamma)L}{2}(1+\xi)}e^{-\frac{n-2\gamma}{2} t_j^i}
\end{split}\end{equation*}
and
\begin{equation*}\begin{split}
J_3&=\int_{\R^n \backslash \cup_i B(p_i,1)}N(\phi)Z_{j,l}^i\,dx\\
&\leq C\|\phi\|_*^2\int_{\R^n \backslash  \cup_i B(p_i,1)}|x|^{-2(n-2\gamma)}\bar{u}^{\beta-2}Z_{j,l}^i\,dx\\
&\leq C\|\phi\|_*^2\int_{\R^n \backslash  \cup_i B(p_i,1)}|x|^{-2(n-2\gamma)}e^{-(\beta-2)\frac{(n-2\gamma)L}{4}}
|x|^{-(n-2\gamma)(\beta-2)}(\lambda_j^i)^{\frac{n-2\gamma}{2}}|x|^{-(n-2\gamma)}\,dx\\
&\leq Ce^{-\frac{(n-2\gamma)L}{2}(1+\xi)}e^{-\frac{n-2\gamma}{2} t_j^i}.
\end{split}\end{equation*}
Combining the above estimates, we have for $l=0$,
\begin{equation*}
|\bar\beta_{j,0,2}^i|+|\bar\beta_{j,0,3}^i|\leq Ce^{-\frac{(n-2\gamma)L}{2}(1+\xi)}e^{-(\gamma_1+\frac{n-2\gamma}{2}) t_j^i}.
\end{equation*}

For $l=1,\cdots,n$, by the estimate for $\phi$ given in \eqref{decay-phi} and the bounds for the term $I_1$ from \eqref{ecuacion11} in Section \ref{sec3} one obtains a similar estimate as above. But this is not enough for our analysis; one needs to be more precise. In order to do this, first recall the definition of $\bar{u}_0$ from \eqref{eq301}. Then
\begin{equation*}
\bar{u}_0=\sum_{i=1}^k U_i(|x-p_i|),
\end{equation*}
where $U_i$ are radial functions in $|x-p_i|$.

Near each singular point $p_i$, we can decompose $$\bar{u}_0=U_i(|x-p_i|)+D_i+O(\lambda_0^i|x-p_i|),$$
where $D_i$ depends on $p_i-p_{i'}$ for $i'\neq i$. And similarly, we can decompose $S(\bar{u}_0)$ into two parts,
\begin{equation*}
S(\bar{u}_0)=\mathcal E(|x-p_i|)+\mathcal E_1(x),
\end{equation*}
where $\mathcal E $ is radial function in $|x-p_i|$ and can be controlled by $Ce^{-\frac{(n-2\gamma)L}{4}(1+\xi)}|x-p_i|^{\gamma_1-2\gamma}$; the second term can be controlled  by $Ce^{-\frac{(n-2\gamma)L}{4}(1+\xi)}|x-p_i|^{\gamma_1-2\gamma+1}$.
We now proceed as follows. Let $\varphi_i=\varphi_i(|x-p_i|)$ be the solution to
\begin{equation*}
\left\{\begin{split}
&(-\Delta)^\gamma \varphi_i-c_{n,\gamma}\beta[U_i+D_i]^{\beta-1}\varphi_i\chi_i=\mathcal E(|x-p_i|)\chi_i+\sum_{j=0}^\infty c^i_{j,0}(w_j^i)^{\beta-1}Z_{j,0}^i,\\
&\int_{\R^n}(w_j^i)^{\beta-1}Z_{j,0}^i\varphi_i\,dx\,=0 \quad\mbox{ for }j=0,\cdots, \infty.
\end{split}
\right.
\end{equation*}
Note that the existence of such a $\varphi_i$ can be proved similarly to the arguments in Section \ref{sec3}. Moreover, as in \eqref{decay-phi} one has
\begin{equation*}
\|\varphi_i\|_*\leq Ce^{-\frac{(n-2\gamma)L}{4}(1+\xi)}.
\end{equation*}
Then we decompose $\phi=\sum_{i=1}^k\varphi_i\chi_i+\tilde{\varphi}$. In this case, since we have cancelled the radial part in the error near each singular point $p_i$ by $\varphi_i$, then the extra error will have an extra factor $|x-p_i|$ and $\tilde{\varphi}$ will satisfy
\begin{equation*}
|\tilde{\varphi}|\leq C\left\{\begin{array}{l}
e^{-\frac{(n-2\gamma)L}{4}(1+\xi)}|x-p_i|^{\gamma_1+1} \mbox{ in }B(p_i,1),\\
e^{-\frac{(n-2\gamma)L}{4}(1+\xi)}|x|^{-(n-2\gamma)} \mbox{ in }\R^n \backslash \cup_i B(p_i,1).
\end{array}
\right.
\end{equation*}
Therefore, by the above decomposition of $\phi$ into radial and nonradial parts one deduces
\begin{equation*}\begin{split}
-\bar\beta_{j,l,2}^i&=-\int_{\R^n}\mathbb L_0(\phi)Z_{j,l}^i\,dx\\
&=\int_{\R^n }[(\bar{u}_0)^{\beta-1}-(w_j^i)^{\beta-1}]Z_{j,l}^i\phi \,dx\\
&=\int_{\R^n}[(U_i+D_i)^{\beta-1}-(w_j^i)^{\beta-1}]Z_{j,l}^i\tilde{\varphi}\,dx
+\int_{\R^n}[(\bar{u}_0)^{\beta-1}-(U_i+D_i)^{\beta-1}]Z_{j,l}^i\phi \,dx.
\end{split}\end{equation*}
Similar to the estimate of $I_1$ in the proof of Lemma \ref{lemma301}, recalling the asymptotic behaviour of $\tilde{\varphi}$ near $p_i$, we can get that the first term can be controlled by \begin{equation*}
\lambda_j^ie^{-\frac{(n-2\gamma)L}{2}(1+\xi)}e^{-(\gamma_1+\frac{n-2\gamma}{2}) t_j^i}
\end{equation*}
in a ball $B(p_i,1)$ (see \eqref{ecuacion12} and notice the extra factor $\lambda_j^i$).
For the second term,
\begin{equation*}\begin{split}
\int_{B(p_i,1)}&[(\bar{u}_0)^{\beta-1}-(U_i+D_i)^{\beta-1}]Z_{j,l}^i\phi \,dx\\
&\leq Ce^{-\frac{(n-2\gamma)L}{4}}\|\phi\|_*\int_{B(p_i,1)}|x|^{\gamma_1-2\gamma+1}(\sum_{j'} v_{j'})^{\beta-2}v_j^{1+\frac{2}{n-2\gamma}}\,dx\\
&\leq Ce^{-\frac{(n-2\gamma)L}{2}(1+\xi )}\int_{B(p_i,1)}|x|^{\gamma_1-2\gamma+1}v_j^{\beta-1+\frac{2}{n-2\gamma}}\,dx\\
&\leq C\lambda_j^ie^{-\frac{(n-2\gamma)L}{2}(1+\xi)}e^{-(\gamma_1+\frac{n-2\gamma}{2}) t_j^i}.
\end{split}\end{equation*}
 Combining the above two estimates,
\begin{equation*}
\int_{B(p_i,1)}\mathbb L_0(\phi)Z_{j,l}^i\,dx
=O(\lambda_j^ie^{-\frac{(n-2\gamma)L}{2}(1+\xi)}e^{-(\gamma_1+\frac{n-2\gamma}{2}) t_j^i}).
\end{equation*}
Next, the asymptotic behaviour of $Z_{j,l}^i$ at infinity, given by
\begin{equation*}
|Z_{j,l}^i|=(\lambda_j^i)^{\frac{n-2\gamma}{2}+1}|x|^{-(n-2\gamma+1)} \quad\mbox{ if }|x-p_i|\geq 1,
\end{equation*}
yields  that
\begin{equation*}
\sum_{i'\neq i}\int_{B(p_{i'},1)}\mathbb L_0(\phi)Z_{j,l}^i\,dx+\int_{\R^n \backslash \cup_{i'} B(p_{i'},1)}\mathbb L_0(\phi)Z_{j,l}^i\,dx\leq C\lambda_j^ie^{-\frac{(n-2\gamma)L}{2}(1+\xi)}e^{-(\gamma_1+\frac{n-2\gamma}{2}) t_j^i}.
\end{equation*}
Using similar argument, we obtain an analogous estimate for $\bar\beta_{j,l,3}^i$. Thus for $l=1,\cdots,n$,
\begin{equation*}
|\bar\beta_{j,l,2}^i|+|\bar\beta_{j,l,3}^i|\leq C\lambda_j^ie^{-\frac{(n-2\gamma)L}{2}(1+\xi)}e^{-(\gamma_1+\frac{n-2\gamma}{2}) t_j^i},
\end{equation*}
which completes the proof of Step 1.\\

\noindent\emph{Step 2: Estimate for $\bar\beta_{j,l,1}^i$}.
Denote by $E:=S(\bar u_0)=(-\Delta)^\gamma \bar{u}_0-c_{n,\gamma}\bar{u}_0^\beta.$ We compute
\begin{equation*}\begin{split}
\bar\beta_{j,l,1}^i&=\int_{\R^n}EZ_{j,l}^i\,dx=\Big[\int_{B(p_i,1)}+\sum_{i'\neq i}\int_{B(p_{i'},1)}+\int_{\R^n \backslash\cup_{i'} B(p_{i'},1)}\,\big]\,dx\\
&=I_{1,j,l}+I_{2,j,l}+I_{3,j,l}.
\end{split}\end{equation*}
Recalling the estimate for $E$ from \eqref{estimate-S} one has
\begin{equation*}\begin{split}
I_{2,j,0}&\leq Ce^{-\frac{(n-2\gamma)L}{4}(1+\xi)}\sum_{i'\neq i}\int_{B(p_{i'},1)}|x-p_{i'}|^{\gamma_1-2\gamma}(\lambda_j^i)^{\frac{n-2\gamma}{2}}\,dx
\leq Ce^{\frac{(n-2\gamma)L}{2}(1+\xi)}e^{-\frac{n-2\gamma}{4} t_j^i}
\end{split}\end{equation*}
and
\begin{equation*}\begin{split}
I_{3,j,0}&\leq Ce^{-\frac{(n-2\gamma)L}{4}(1+\xi)}\int_{\R^n / \cup_{i'} B(p_{i'},1)}|x|^{-(n+2\gamma)}(\lambda_j^i)^{\frac{n-2\gamma}{2}}|x|^{-(n-2\gamma)}\,dx
\leq Ce^{-\frac{(n-2\gamma)L}{2}(1+\xi)}e^{-\frac{n-2\gamma}{4} t_j^i}.
\end{split}\end{equation*}
For $l=1,\cdots,n$, we know that $|Z_{j,l}^i|=O(\lambda_j^i)Z_{j,0}^i$ , which yields easily that
\begin{equation*}
|I_{2,j,l}|+|I_{3,j,l}|\leq Ce^{-\frac{(n-2\gamma)L}{2}(1+\xi)}e^{-(1+\frac{n-2\gamma}{4}) t^i_j}.
\end{equation*}

Next, for $l=0,\ldots,n$, we consider $I_{1,j,l}$: fixed $i,j$, substitute the expression for $\bar u_0$ from \eqref{eq301}, so
\begin{equation*}\begin{split}
E&=(-\Delta)^\gamma \bar{u}_0-c_{n,\gamma}\bar{u}_0^\beta\\
&=(-\Delta)^\gamma (\sum_{j'=0}^\infty w_{j'}^i)+(-\Delta)^\gamma \phi_i-\beta c_{n,\gamma}(w_j^i)^{\beta-1}\phi_i+(-\Delta)^\gamma(1-\chi_i)\phi_i\\
&\qquad-c_{n,\gamma}[(\sum_{j'=0}^\infty w_{j'}^i+\phi_i+\sum_{i'\neq i}w_j^{i'})^\beta-\sum_{i'\neq i}(w_j^{i'})^\beta-\beta (w_j^i)^{\beta-1}\phi_i]\\
&=-c_{n,\gamma}[(\sum_{j'=0}^\infty w_{j'}^i+\phi_i+\sum_{i'\neq i}w_j^{i'})^\beta-\sum_{j=0}^\infty (w_j^i)^\beta-\sum_{i'\neq i}(w_j^{i'})^\beta-\beta (w_j^i)^{\beta-1}\phi_i]\\
&\qquad+\mathbb L_j^i\phi_i+(-\Delta)^\gamma(1-\chi_i)\phi_i\\
&=-c_{n,\gamma}\big[\bar{u}^\beta-\sum_{j'=0}^\infty (w_{j'}^i)^\beta-\sum_{i'\neq i}(w_j^{i'})^\beta-\beta(w_j^i)^{\beta-1}\big(\phi_i+\sum_{i'\neq i}w_j^{i'}+\sum_{j'\neq j}w_{j'}^i\big)\\
&\qquad+\beta(w_j^i)^{\beta-1}\big(\sum_{i'\neq i}w_j^{i'}+\sum_{j'\neq j}w_{j'}^i\big)\big]+\mathbb L_j^i\phi_i+(-\Delta)^\gamma(1-\chi_i)\phi_i.
\end{split}\end{equation*}
Here $\mathbb L_j^i$ denotes the linearized operator around $w_j^i$. Looking at the equation that $\phi_i$ satisfies and its bounds (see formula \eqref{psi} and Corollary \ref{pro202}), we have in general the following estimates:

On the one hand, for $l=0$, since $Z_{j,0}^i$ is odd in the variable $t^i-t_j^i$, where we have defined $t^i=-\log |x-p_i|$, by the above expansion for $E$,
\begin{equation*}\begin{split}
I_{1,j,0}&\leq C\int_{B(p_i,1)}(w_j^i)^{\beta-1}Z_{j,0}^i{\big(\sum_{i'\neq i}w_0^{i'}+\sum_{j'\neq j}w_{j'}^i\big)}\,dx
+O(e^{-\frac{(n-2\gamma)L}{4}(1+\xi)}(\lambda_j^i)^{\frac{n-2\gamma}{2}})\\
&\leq C\sum_{i'\neq i}(\lambda_0^{i'})^{\frac{n-2\gamma}{2}}
\int_{B_1}|x|^{-2\gamma}(v_j^i)^{\beta-1}|x|^{-\frac{n-2\gamma}{2}}|(v_j^i)'|\,dx
+\int_{B_1}\sum_{j'\neq j}|x|^{-n}(v_j^i)^{\beta-1}(v_j^i)' v_{j'}^i\,dx\\
&\leq Ce^{-\frac{n-2\gamma}{4}L}\int_{0}^\infty e^{-\frac{n-2\gamma}{2}t}(v_j^i)^\beta \,dt+\int_0^\infty (v_j^i)^{\beta-1}(v_j^i)'\sum_{j'\neq j}v_{j'}^i\,dt\\
&\quad+O(e^{-\frac{(n-2\gamma)L}{4}(1+\xi)}(\lambda_j^i)^{\frac{n-2\gamma}{2}})\\
&\leq Ce^{-\frac{n-2\gamma}{4}L}e^{-\frac{n-2\gamma}{2}t^i_j}+\sum_{j'\leq 2j}\int_0^\infty (v_j^i)^{\beta-1}(v_j^i)'v_{j'}^i\,dt+\sum_{j'\geq 2j+1}\int_0^\infty (v_j^i)^{\beta-1}(v_j^i)'v_{j'}^i\,dt.
\end{split}\end{equation*}
Let us bound the two terms in this expression:
\begin{equation*}\begin{split}
&\int_0^\infty (v_j^i)^{\beta-1}(v_j^i)'\sum_{j'\leq 2j}v_{j'}^i\,dt\\
&=\int_0^{t_{2j}^i+\frac{L_i}{2}}(v_j^i)^{\beta-1}(v_j^i)'\sum_{j'\leq 2j}v_{j'}^i\,dt+\int_{ t_{2j}^i+\frac{L_i}{2}}^\infty (v_j^i)^{\beta-1}(v_j^i)'\sum_{j'\leq 2j}v_{j'}^i\,dt\\
&(t'=t-t_j^i)\\
&=\int_{-t_j^i}^{t_j^i}v^{\beta-1}v'(t')\sum_{0\leq j'\leq 2j}v_{j'}^i\,dt'+\int_{t\geq t_{2j}^i+\frac{L_i}{2}}(v_j^i)^{\beta-1}(v_j^i)'\sum_{j'\leq 2j}v_{j'}^i\,dt.
\end{split}\end{equation*}
For the first integral, since $v'(t')$ is odd in $t'$ and $\sum_{j'\leq 2j}v_{j'}^i$ is an even function of $t'$, this integral is $0$. In the meantime, thanks to the exponential decaying of $v$, the second integral is bounded by $e^{-\frac{n+2\gamma}{2}t_j^i}$, and we may conclude that
\begin{equation*}
\int_0^\infty (v_j^i)^{\beta-1}(v_j^i)'\sum_{j'\leq 2j}v_{j'}^i\,dt\leq Ce^{-\frac{n+2\gamma}{2}t_j^i}.
\end{equation*}
In addition,
\begin{equation*}\begin{split}
\sum_{j'\geq 2j+1}&\int_0^\infty (v_j^i)^{\beta-1}(v_j^i)'v_{j'}^i\,dt\\
&\leq Ce^{-\frac{(n-2\gamma)}{2}|t^i_{2j+1}-t^i_j|}\leq Ce^{-\frac{n-2\gamma}{4}L}e^{-\frac{n-2\gamma}{2}t_j^i.}
\end{split}\end{equation*}
In conclusion, one has
\begin{equation*}
I_{1,j,0}\leq Ce^{-\frac{n-2\gamma}{4}L}e^{-\frac{n-2\gamma}{2}t_j^i}+e^{-\frac{n+2\gamma}{2}t_j^i}.
\end{equation*}

On the other hand,  for $l=1,\cdots,n$, since $Z_{j,l}^i$ is odd in $x-p_i$, and both $w_j^i$, $\phi_i$ are even in $x-p_i$, one has
\begin{equation*}\begin{split}
I_{1,j,l}&\leq C\int_{B(p_i,1)}\sum_{i'\neq i}(w_j^i)^{\beta-1}Z_{j,l}^iw_0^{i'}\,dx\\
&\leq C\int_{B(p_i,1)}\sum_{i'\neq i}(w_j^i)^{\beta-1}|x-p_i|^{-\frac{n-2\gamma}{2}+1}v_j^{1+\frac{2}{n-2\gamma}}
(\lambda_0^{i'})^{\frac{n-2\gamma}{2}}\,dx\\
&\leq Ce^{-\frac{n-2\gamma}{4}L}\int_{B_1}|x|^{-2\gamma}v_j^{\beta-1}|x|^{-\frac{n-2\gamma}{2}+1}
v_j^{1+\frac{2}{n-2\gamma}}\,dx\\
&\leq Ce^{-\frac{n-2\gamma}{4}L}e^{-(1+\frac{n-2\gamma}{2})t^i_j}.
\end{split}\end{equation*}
From the above two estimates, when $j\geq 1$,
\begin{equation*}
I_{1,j,l}\leq \left\{\begin{array}{l}
e^{-\frac{(n-2\gamma)L}{2}(1+\xi)}e^{-\sigma t_j^i}\ \mbox{ if }l=0,\\
e^{-\frac{(n-2\gamma)L}{2}(1+\xi)}e^{-(1+\sigma) t^i_j}\ \mbox{ if }l=1,\cdots,n,
\end{array}
\right.
\end{equation*}
for some $\xi,\sigma>0$.\\

On the contrary, for $j=0$ one has
\begin{equation*}
I_{1,0,0}=O(e^{-\frac{n-2\gamma}{2}L}), \quad I_{1,0,l}=O(\lambda_0^i)e^{-\frac{n-2\gamma}{2}L},
\end{equation*}
but can obtain more accurate estimates in this case. This is going to be the crucial step in the proof of the Lemma since it gives the formula for the compatibility condition of the balancing condition.

First, if $l=0$,
\begin{equation*}\begin{split}
I_{1,0,0}&=-c_{n,\gamma}\beta\int_{\R^n}(w_0^i)^{\beta-1}(w_1^i+\sum_{i'\neq i}w_0^{i'})Z_{0,0}^i\,dx+O(e^{-\frac{(n-2\gamma)L}{2}(1+\xi)})\\
&=:T_{1,0}+T_{2,0}+O(e^{-\frac{(n-2\gamma)L}{2}(1+\xi)}).
\end{split}\end{equation*}
In this case, by expression \eqref{eqa01},
\begin{equation*}
T_{1,0}=-c_{n,\gamma}F(|\log \tfrac{\lambda_1^i}{\lambda_0^i}|)\frac{\log\frac{\lambda_1^i}{\lambda_0^i}}{|\log\frac{\lambda_1^i}{\lambda_0^i}|},
\end{equation*}
and by formula \eqref{eqa02} and the relation that $Z_{0,0}^i=\frac{\partial w_0^i}{\partial \lambda_0^i}\lambda_0^i R^i$,
\begin{equation*}
T_{2,0}=-c_{n,\gamma}\sum_{i'\neq i}A_2|p_{i'}-p_i|^{-(n-2\gamma)}(\lambda^i_0\lambda_0^{i'})^{\frac{n-2\gamma}{2}}[1+O(\lambda^i_0)^2].
\end{equation*}
Combining the above two estimates yields
\begin{equation*}\begin{split}
I_{1,0,0}&={c_{n,\gamma}}\Big[F(L_1)-\sum_{i'\neq i}A_2|p_{i'}-p_i|^{-(n-2\gamma)}(\lambda^i_0\lambda_0^{i'})^{\frac{n-2\gamma}{2}}\Big]+O(e^{-\frac{(n-2\gamma)L}{2}(1+\xi)})\nonumber\\
&=c_{n,\gamma}q_i\Big[q_i-A_2\sum_{i'\neq i}q_{i'}(R^iR^{i'})^{\frac{n-2\gamma}{2}}|p_i-p_{i'}|^{-(n-2\gamma)}\Big]e^{-\frac{(n-2\gamma)L}{2}}(1+o(1))\nonumber\\
&+O(e^{-\frac{(n-2\gamma)L}{2}(1+\xi)}).
\end{split}\end{equation*}

On the other hand, for  $l=1,\cdots,n$
\begin{equation*}
I_{1,0,l}=-c_{n,\gamma}\beta\int_{\R^n}(w_0^i)^{\beta-1}(\sum_{i'\neq i}w_0^{i'})Z_{0,0}^i\,dx+O(e^{-\frac{(n-2\gamma)L}{2}(1+\xi)}\lambda_0^i),
\end{equation*}
and recall that $Z_{j,l}^i=-\frac{\partial w_j^i}{\partial x_l}\lambda_j^i$, by the estimate \eqref{eqa03}, one has
\begin{eqnarray*}
I_{1,0,l}=c_{n,\gamma}A_3\lambda_0^i[\sum_{i'\neq i}\frac{(p_{i'}-p_i)_l}{|p_{i'}-p_i|^{n-2\gamma+2}}(\lambda_0^i\lambda_0^{i'})^{\frac{n-2\gamma}{2}}]+O(e^{-\frac{(n-2\gamma)L}{2}(1+\xi)}\lambda_0^i).
\end{eqnarray*}
Then combining the estimates for $\bar\beta_{j,l,1}^i,\bar\beta_{j,l,2}^i,\bar\beta_{j,l,3}^i$, we achieve the proof of the Lemma.

%For $j=1$
%\begin{eqnarray*}
%I_{1,1,l}&&=-c_{n,\gamma}\beta\int_{\R^n}(w_1^i)^{\beta-1}(w_{2}^i+w_0^i)Z_{1,l}^idx+O(e^{-\frac{(n-2\gamma)L}{2}(1+\xi)})
%\end{eqnarray*}
%So one has
%\begin{equation}
%I_{1,1,l}=O(e^{-\frac{(n-2\gamma)L}{2}(1+\xi)}),
%\end{equation}
%and for $j\geq 2$, one has
%\begin{eqnarray*}
%E&&=c_{n,\gamma}[u_{L_i}^\beta-\sum_{j=-\infty}^{-1}(w_j^i)^\beta-\sum_{i'\neq i}(w_j^{i'})^\beta]\\
%&&-c_{n,\gamma}(u_{L_i}-\sum_{j=-\infty}^{-1}w_j^i-\sum_{i'\neq i}w_j^{i'})^\beta+O(e^{-\frac{(n-2\gamma)L}{4}(1+\xi))}
%\end{eqnarray*}
%Thus
%\begin{eqnarray*}
%I_{1,j,l}
%&&\leq Cu_{L_i}^{\beta-1}Z_{j,l}^i[w_{-1}^i+\sum_{i'\neq i}w_0^{i'}]dx+O(e^{-\frac{(n-2\gamma)L}{2}(1+\xi))}e^{-\sigma t_j}\\
%&&\leq C(e^{-\frac{(n-2\gamma)L}{2}(1+\xi))}e^{-\sigma t_j}.
%\end{eqnarray*}
\end{proof}
\begin{remark}\label{remark401}
Fixed $i$, if we consider the approximate solution $\bar{u}_i^0$ with $a_{j}^i=0, r^i_j=0$ for $j=0,\cdots, \infty$, then the same estimates for $\bar\beta_{j,l}^i$ in the above lemma hold.
\end{remark}

\subsection{Estimates for general parameters}\label{sec402}
Next we study the coefficients $c_{j,l}^i$ for a general configuration space satisfying (\ref{para1}) and (\ref{para2}). Most of the estimates of the previous subsection will continue to hold, but we need to be especially careful when considering $\beta_{j,l,1}^i$. First, from  Remark \ref{remark401}, one can see that only the perturbations of $a_j^i, r_j^i$ will affect the numbers $\beta_{j,l}^i$, i.e, we can get the same estimates for $\beta_{j,l}^i$ for general parameters $(a_j^{i'}, r_j^{i'})$ satisfying (\ref{para1}) and (\ref{para2}) if $(a_j^i,r_j^i)=(0,0)$. So fixing $i=I$, we would like to study the estimates for $\beta_{j,l}^i$ when we vary the parameters $a_j^i, r_j^i$.  First we have the following estimates:

\begin{lemma}\label{lemma402}
Suppose that the parameters $R^i$ satisfy \eqref{parar}. Let $e_J^I$ be a vector in $\R^n$ and $r_J^I$ be a real number in $\R$. We let $X(t)$ be the configuration for which all the parameters $(a_j^i, r_j^i)$ are fixed to be $(0, 0)$ if $j\neq J$ and where $(a_J^I,R_J^I)=(te_J^I,R^I(1+tr_J^I))$. Assume that $|e_J^I|\leq C(\lambda_J^I)^2$ and $|r_J^I|\leq Ce^{-\tau t^I_J}$. We also let $w_{J,t}^I=w_{R^I(1+t r_J^I)}(x-te_J^I)$. Then we have the following:
\begin{itemize}
\item If $i=I$, $J\neq j$,
\begin{equation*}\begin{split}
\frac{\partial }{\partial t}&\Big|_{t=0}\int_{\R^n}S(\bar{u}_t)Z_{j,l}^i\,dx\\
&=\left\{\begin{array}{l}
-\dfrac{\partial }{\partial t}\big[c_{n,\gamma}F(|\log \frac{\lambda_J^i}{\lambda_j^i}|)\frac{\log \frac{\lambda_J^i}{\lambda_j^i}}{\big|\log \frac{\lambda_J^i}{\lambda_j^i}\big|}\big]+O(e^{-\frac{(n-2\gamma)L}{2}(1+\xi)}e^{-\tau t_J^I}e^{-\sigma |t^I_j-t^I_J|}) \mbox{ if }l=0,\\
-\dfrac{\partial }{\partial t}\big[c_{n,\gamma}A_0\lambda_j^i\min\{\frac{\lambda_J^i}{\lambda_j^i},\frac{\lambda_j^i}{\lambda_J^i}\}^{\frac{n-2\gamma}{2}}
\frac{te_l}{|\max\{\lambda_j^i,\lambda_J^i\}|^2}\big]\\
\qquad\qquad\qquad\qquad\quad+O(\lambda_j^ie^{-\frac{(n-2\gamma)L}{2}(1+\xi)}e^{-\tau t^I_J}e^{-\sigma |t^I_j-t^I_J|}) \ \mbox{ if }l=1,\cdots,n,
\end{array}
\right.
\end{split}\end{equation*}
for some $\sigma>0$ independent of $\tau$ and $L$.
\item If $i=I$, $J=j=0$,
\begin{equation*}\begin{split}
\frac{\partial }{\partial t}&\Big|_{t=0}\int_{\R^n}S(\bar{u}_t)Z_{j,l}^i\,dx\\
&=\left\{\begin{array}{l}
-c_{n,\gamma}\dfrac{\partial }{\partial t}\Big[A_2\displaystyle{\sum_{i'\neq i}}|p_{i'}-p_i|^{-(n-2\gamma)}(\lambda_0^i\lambda_0^{i'})^{\frac{n-2\gamma}{2}}
+F(|\log \frac{\lambda_0^i}{\lambda_{1}^i}|)\frac{\log \frac{\lambda_{1}^i}{\lambda_0^i}}{\big|\log \frac{\lambda_{1}^i}{\lambda_0^i}\big|}\Big]\\
\qquad\qquad\qquad\qquad\qquad\qquad\qquad\qquad\qquad\qquad+O(e^{-\frac{(n-2\gamma)L}{2}(1+\xi)}) \, \mbox{ if }l=0,\\
c_{n,\gamma}\lambda_0^i\dfrac{\partial }{\partial t}\Big[A_3\displaystyle{\sum_{i'\neq i}}\frac{p_{i'}-p_i}{|p_{i'}-p_i|^{n-2\gamma+2}}(\lambda_0^i\lambda^{i'}_0)^{\frac{n-2\gamma}{2}}
        +A_0\min\{\frac{\lambda_{0}^i}{\lambda_1^i},\frac{\lambda_1^i}{\lambda_{0}^i}\}^{\frac{n-2\gamma}{2}}
        \frac{te_l}{|\max\{\lambda_{j'}^i,\lambda_J^i\}|^2}\Big]
\\ \qquad\qquad\qquad\qquad\qquad\qquad\qquad\qquad\qquad\qquad+O(\lambda_0^ie^{-\frac{(n-2\gamma)L}{2}(1+\xi)}) \ \mbox{ if }l=1,\cdots,n.
\end{array}
\right.
\end{split}\end{equation*}

\item If $i=I$, $J=j\geq 1$,
\begin{equation*}\begin{split}
\frac{\partial }{\partial t}&\Big|_{t=0}\int_{\R^n}S(\bar{u}_t)Z_{j,l}^i\,dx\\
&=\left\{\begin{array}{l}
-c_{n,\gamma}\displaystyle{\sum_{j'\neq j}}\dfrac{\partial }{\partial t}\big[F(|\log \tfrac{\lambda_j^i}{\lambda_{j'}^i})\frac{\log \frac{\lambda_{j'}^i}{\lambda_j^i}}{|\log \frac{\lambda_{j'}^i}{\lambda_j^i}|}\big]+O(e^{-\frac{(n-2\gamma) L}{2}(1+\xi)}e^{-\tau t_j}) \ \mbox{ if }l=0,\\
c_{n,\gamma}\lambda_j^i \dfrac{\partial }{\partial t}
\big[A_0\displaystyle{\sum_{j'\neq J}}\min\{\frac{\lambda_{j'}^i}{\lambda_J^i},\frac{\lambda_J^i}{\lambda_{j'}^i}\}^{\frac{n-2\gamma}{2}}
\frac{te_l}{|\max\{\lambda_{j'}^i,\lambda_J^i\}|^2}\big]
+O(\lambda_j^ie^{-\frac{(n-2\gamma)L}{2}(1+\xi)}e^{-\tau t_J})\\
\qquad\qquad\qquad\qquad\qquad\qquad\qquad\qquad\qquad\qquad\qquad\qquad\qquad \mbox{ if }l=1,\cdots,n.
\end{array}
\right.
\end{split}\end{equation*}
\end{itemize}
\end{lemma}

\begin{proof}
Fix $i=I$.
We first consider the case in which $J\neq j$. We have
\begin{equation*}\begin{split}
\frac{d}{dt}\Big|_{t=0}\int_{\R^n}S(\bar{u}_t)Z_{j,l}^i\,dx
&=c_{n,\gamma}\beta\int_{\R^n}[(w_{J,t}^i)^{\beta-1}-\bar{u}_t^{\beta-1}]\frac{\partial w_{J,t}^i}{\partial t}Z_{j,l}^i\,dx\\
&=-c_{n,\gamma}\beta\int_{\R^n}(w_j^i)^{\beta-1}\frac{\partial w_{J,t}^i}{\partial t}Z_{j,l}^i\,dx
\\&\quad+c_{n,\gamma}\beta\int_{\R^n}[(w_j^i)^{\beta-1}+(w_{J,t}^{i})^{\beta-1}-\bar{u}_t^{\beta-1}]\frac{\partial w_{J,t}^i}{\partial t}Z_{j,l}^i\,dx\\
&=:M_1+M_2.
\end{split}\end{equation*}
From the proof of the appendix, more precisely, (\ref{eqa01}) for $l=0$ and (\ref{eqa04}) for $l=1,\ldots,n$, one can find that
\begin{equation*}\begin{split}
M_1&=-c_{n,\gamma}\beta\int_{\R^n}(w_j^i)^{\beta-1}\frac{\partial w_{J,t}^i}{\partial t}Z_{j,l}^i\,dx=-c_{n,\gamma}\beta\frac{\partial }{\partial t}\int_{\R^n}(w_j^i)^{\beta-1}Z_{j,l}^iw_{J,t}^i\,dx\\
%&&=-c_{n,\gamma}\frac{\partial }{\partial t}\int_{\R^n}\frac{\partial }{\partial l}(w_j^i)^\beta w_{J,t}^idx\\
&=\left\{\begin{array}{l}
-\dfrac{\partial }{\partial t}\Big[c_{n,\gamma}F(|\log \frac{\lambda_J^i}{\lambda_j^i}|)\frac{\log \frac{\lambda_J^i}{\lambda_j^i}}{|\log \frac{\lambda_J^i}{\lambda_j^i}|}\Big]
%\qquad\qquad\qquad+O(e^{-\frac{(n-2\gamma)L}{2}(1+\xi)}e^{-\tau t_J}e^{-\sigma |t_j-t_J|}) \
\ \mbox{ if }l=0,\\
-\dfrac{\partial }{\partial t}\Big[c_{n,\gamma}A_0\lambda_j^i\min\{\frac{\lambda_J^i}{\lambda_j^i},\frac{\lambda_j^i}{\lambda_J^i}\}^{\frac{n-2\gamma}{2}}
\frac{te_l}{|\max\{\lambda_j^i,\lambda_J^i\}|^2}\Big]
%\\
%\qquad\qquad\qquad+O(\lambda_j^ie^{-\frac{(n-2\gamma)L}{2}(1+\xi)}e^{-\tau t_J}e^{-\sigma |t_j-t_J|})\ \mbox{ if }l=1,\cdots,n,
\end{array}
\right.
\end{split}\end{equation*}
for a constant $A_0>0$.
Moreover,
\begin{equation*}\begin{split}
M_2&=c_{n,\gamma}\beta\int_{\R^n}[(w_j^i)^{\beta-1}+(w_{J,t}^{i})^{\beta-1}-\bar{u}_t^{\beta-1}]\frac{\partial w_{J,t}^i}{\partial t}Z_{j,l}^i\,dx\\
&=\left\{\begin{array}{l}
O(e^{-\frac{(n-2\gamma)L}{2}(1+\xi)})e^{-\tau t^I_J}e^{-\sigma |t^I_j-t^I_J|} \ \mbox{ for }l=0,\\
O(\lambda_j^ie^{-\frac{(n-2\gamma)L}{2}(1+\xi)})e^{-\tau t^I_J}e^{-\sigma |t^I_j-t^I_J|} \  \mbox{ for }l=1,\cdots,n,
\end{array}
\right.
\end{split}\end{equation*}
which proves the assertion when $J\neq j$.

Now we consider the case $J=j$, for which we have
\begin{equation*}\begin{split}
\frac{d}{dt}\Big|_{t=0}\int_{\R^n}S(\bar{u}_t)Z_{j,t,l}^i\,dx
&=\int_{\R^n}\frac{\partial }{\partial t}S(\bar{u}_t)Z_{j,t,l}^i\,dx+\int_{\R^n} S(\bar{u}_t)\frac{\partial }{\partial t}Z_{j,t,l}^i\,dx\\
&=\int_{\R^n}\big[(-\Delta)^\gamma-c_{n,\gamma}\beta(w_{J,t}^i)^{\beta-1}\big]\big[\sum_{j'\neq J }w_{j'}^i+\sum_{i'\neq i}w_j^{i'}+\phi_{L_i}\big]\frac{\partial }{\partial t}Z_{j,t,l}^i\,dx\\
&\quad-\beta(\beta-1)c_{n,\gamma}\int_{\R^n}(w_{J,t}^i)^{\beta-2}\big[\sum_{j'\neq J }w_{j'}^i+\sum_{i'\neq i}w_j^{i'}+\phi_{L_i}\big]\frac{\partial w_{J,t}^i}{\partial t}Z_{j,t,l}^i\,dx\\
&\quad+\left\{\begin{array}{l}
O(e^{-\frac{(n-2\gamma)L}{2}(1+\xi)}e^{-\tau t^I_j})  \ \mbox{ if }l=0,\\
O(\lambda_j^i e^{-\frac{(n-2\gamma)L}{2}(1+\xi)}e^{-\tau t^I_j})\ \mbox{ if }l\geq 1.
\end{array}
\right.
\end{split}\end{equation*}
From the equation satisfied by $Z_{t,l,i}$, and taking derivative with respect to $t$, one can cancel the terms containing $\phi_{L_i}$, which yields
\begin{equation*}\begin{split}
\frac{d}{dt}\Big|_{t=0}\int_{\R^n}S(\bar{u}_t)Z_{j,t,l}^i\,dx
&=\int_{\R^n}-c_{n,\gamma}\beta(w_{J,t}^i)^{\beta-1}\big[\sum_{j'\neq J }w_{j'}^i+\sum_{i'\neq i}w_j^{i'}\big]\frac{\partial }{\partial t}Z_{j,t,l}^i\,dx\\
&\quad-\beta(\beta-1)c_{n,\gamma}\int_{\R^n}(w_{J,t}^i)^{\beta-2}\big[\sum_{j'\neq J }w_{j'}^i+\sum_{i'\neq i}w_j^{i'}\big]\frac{\partial w_{J,t}^i}{\partial t}Z_{j,t,l}^i\,dx\\
&\quad+\left\{\begin{array}{l}
O(e^{-\frac{(n-2\gamma)L}{2}(1+\xi)}e^{-\tau t^I_j}) \ \mbox{ if }l=0,\\
O(\lambda_j^i e^{-\frac{(n-2\gamma)L}{2}(1+\xi)}e^{-\tau t^I_j}) \ \mbox{ if }l\geq 1,
\end{array}
\right.\\
&=-c_{n,\gamma}\frac{\partial }{\partial t}\int_{\R^n}\beta(w_{J,t}^i)^{\beta-1}Z_{j,t,l}^i\big[\sum_{j'\neq J }w_{j'}^i+\sum_{i'\neq i}w_j^{i'}\big]\,dx\\
&\quad+\left\{\begin{array}{l}
O(e^{-\frac{(n-2\gamma)L}{2}(1+\xi)}e^{-\tau t^I_j}) \  \mbox{ if }l=0,\\
O(\lambda_j^i e^{-\frac{(n-2\gamma)L}{2}(1+\xi)}e^{-\tau t^I_j}) \ \mbox{ if }l\geq 1,
\end{array}
\right.\\
&=:N_1+N_2+O(\cdots).\end{split}
\end{equation*}
Similar to the estimates before, one can get that for $l=1,\cdots,n$, by estimate (\ref{eqa04}) in the Appendix,
\begin{equation*}\begin{split}
N_1&=\int_{\R^n}\beta (w_{J,t}^i)^{\beta-1}Z_{j,t,l}^i\sum_{j'\neq J}w_{j'}^i\,dx\\
&=-A_0\lambda_j^i\sum_{j'\neq J}\min\Big\{\frac{\lambda_{j'}^i}{\lambda_J^i},\frac{\lambda_J^i}{\lambda_{j'}^i}\Big\}^{\frac{n-2\gamma}{2}}
\frac{te_l}{|\max\{\lambda_{j'}^i,\lambda_J^i\}|^2}+O(\lambda_j^i
e^{-\frac{(n-2\gamma)L}{2}}(1+\xi)e^{-\tau t^I_J}),
\end{split}\end{equation*}
and  from (\ref{eqa03}),
\begin{equation*}\begin{split}
N_2&=\int_{\R^n}\beta (w_{J,t}^i)^{\beta-1}Z_{j,t,l}^i\sum_{i'\neq i}w_{j}^{i'}\,dx\\
&=\left\{\begin{array}{l}
-\lambda_0^i\Big[A_3\displaystyle {\sum_{i'\neq i}}\frac{(p_{i'}-p_i)_l}{|p_{i'}-p_i|^{n-2\gamma+2}}
(\lambda_0^i\lambda^{i'}_0)^{\frac{n-2\gamma}{2}}
+O(e^{-\frac{(n-2\gamma)L}{2}(1+\xi)})\Big] \ \mbox{ if }J=0,\\
O(\lambda_j^ie^{-\frac{(n-2\gamma)L}{2}(1+\xi)}e^{-\tau t^I_J}) \ \mbox{ if }J\geq 1.
\end{array}
\right.
\end{split}\end{equation*}
On the other hand, for $l=0$, by (\ref{eqa01}),
\begin{equation*}\begin{split}
N_1&=\int_{\R^n}\beta (w_{J,t}^i)^{\beta-1}Z_{j,t,l}^i\sum_{j'\neq J}w_{j'}^i\,dx=\sum_{j'\neq j}F\big(\big|\log \tfrac{\lambda_j^i}{\lambda_{j'}^i}\big|\big)\frac{\log \frac{\lambda_{j'}^i}{\lambda_j^i}}{|\log \frac{\lambda_{j'}^i}{\lambda_j^i}|},
%+O(e^{-\frac{(n-2\gamma) L}{2}(1+\xi)}e^{-\tau t_j}),
\end{split}\end{equation*}
and using (\ref{eqa02}),
\begin{equation*}\begin{split}
N_2&=\int_{\R^n}\beta (w_{J,t}^i)^{\beta-1}Z_{j,t,l}^i\sum_{i'\neq i}w_{j}^{i'}\,dx\\
&=\left\{\begin{array}{l}
A_2\displaystyle{\sum_{i'\neq i}}|p_{i'}-p_i|^{-(n-2\gamma)}(\lambda_0^i\lambda_0^{i'})^{\frac{n-2\gamma}{2}}+O(e^{-\frac{(n-2\gamma)L}{2}(1+\xi)}) \ \mbox{ if }J=0,\\
e^{-\frac{(n-2\gamma)L}{2}(1+\xi)}e^{-\tau t^I_J} \ \mbox{ if }J\geq 1.
\end{array}
\right.
\end{split}\end{equation*}
Combining all the above estimates, the proof of the Lemma is completed.

\end{proof}

From Lemma \ref{lemma401} and Lemma \ref{lemma402}, we find the decay estimate for the $\beta_{j,l}^i$ for general parameters $a_j^i, r_j^i$ satisfying conditions (\ref{para1}) and \eqref{para2}:

\begin{lemma}\label{lemma403}
For the parameters $(a_j^i, R_j^i)$ satisfying \eqref{para1} and \eqref{para2}, we have the following estimates:

\begin{eqnarray*}
&&\beta_{0,0}^i=-c_{n,\gamma}q_i\Big[A_2\sum_{i'\neq i}|p_{i'}-p_i|^{-(n-2\gamma)}(R_0^iR_0^{i'})^{\frac{n-2\gamma}{2}}q_{i'}
-\Big(\frac{R_1^i}{R_0^i}\Big)^{\frac{n-2\gamma}{2}}q_i\Big]e^{-\frac{(n-2\gamma)L}{2}}(1+o(1))\nonumber\\
&&\qquad+O(e^{-\frac{(n-2\gamma)L}{2}(1+\xi)}),\\
&&\beta_{0,l}^i=c_{n,\gamma}\lambda_0^i\Big[A_3\sum_{i'\neq i}\frac{(p_{i'}-p_i)_l}{|p_{i'}-p_i|^{n-2\gamma+2}}(R_0^iR_0^{i'})^{\frac{n-2\gamma}{2}}q_{i'}
+A_0\Big(\frac{R_1^i}{R_0^i}\Big)^{\frac{n-2\gamma}{2}}\frac{a_0^i-a_1^i}{(\lambda_0^i)^2} q_i \Big]q_ie^{-\frac{n-2\gamma}{2}L}\nonumber\\
&&\qquad+O(\lambda_0^ie^{-\frac{(n-2\gamma)L}{2}(1+\xi)}), \text{ for }l\geq 1,\\
&&\beta_{j,l}^i=\left\{\begin{array}{l}
O(e^{-\frac{(n-2\gamma)L}{2}(1+\xi)}e^{-\sigma t^i_j}+e^{-\frac{(n-2\gamma)L}{2}}e^{-\tau t_{j-1}^i}),\ l=0,\\
O(\lambda_j^i e^{-\frac{(n-2\gamma)L}{2}(1+\xi)}e^{-\sigma t_j^i}+\lambda_j^ie^{-\frac{(n-2\gamma)L}{2}}e^{-\tau t_{j-1}^i}),\ l\geq 1,
\end{array}
\right. \mbox{ for }j\geq 1.
\end{eqnarray*}
where $\sigma$ is obtained in Lemma \ref{lemma401}.
\end{lemma}

\begin{proof}
Using the notation in the previous subsection, we first estimate $\beta_{j,l,1}^i$. Using Lemma \ref{lemma401} and Lemma \ref{lemma402}, and integrating in $t$ from $0$ to $1$, varying the parameters $(a,R)=(0, R^i)$ to  $(a_j^i, R_j^i)$, and using the estimates satisfied by the parameters. The integration yields
\begin{equation*}\begin{split}
\beta_{0,0,1}^i&=-c_{n,\gamma}q_i\Big[A_2\sum_{i'\neq i}|p_{i'}-p_i|^{-(n-2\gamma)}(R_0^iR_0^{i'})^{\frac{n-2\gamma}{2}}q_{i'}-\Big(\frac{R_1^i}{R_0^i}\Big)^{\frac{n-2\gamma}{2}}q_i\Big]\\
&\,\cdot e^{-\frac{(n-2\gamma)L}{2}}(1+o(1))+O(e^{-\frac{(n-2\gamma)L}{2}(1+\xi)}),\\
\beta_{0,l,1}^i&=c_{n,\gamma}\lambda_0^i\Big[A_3\sum_{i'\neq i}\frac{(p_{i'}-p_i)_l}{|p_{i'}-p_i|^{n-2\gamma+2}}(R_0^iR_0^{i'})^{\frac{n-2\gamma}{2}}q_{i'}
+A_0\Big(\frac{R_1^i}{R_0^i}\Big)^{\frac{n-2\gamma}{2}}\frac{a_0^i-a_1^i}{(\lambda_0^i)^2} q_i \Big]q_ie^{-\frac{n-2\gamma}{2}L}\\
&+O(\lambda_0^ie^{-\frac{(n-2\gamma)L}{2}(1+\xi)}), \text{ for }l\geq 1,\\
\beta_{j,l,1}^i&=\left\{\begin{array}{l}
O(e^{-\frac{(n-2\gamma)L}{2}(1+\xi)}e^{-\sigma t^i_j}+e^{-\frac{(n-2\gamma)L}{2}}e^{-\tau t_{j-1}^i}) \ \mbox{ if } l=0,\\
O(\lambda_j^i e^{-\frac{(n-2\gamma)L}{2}(1+\xi)}e^{-\sigma t^i_j}+\lambda_j^ie^{-\frac{(n-2\gamma)L}{2}}e^{-\tau t_{j-1}^i})\ \mbox{ if } l\geq 1,
\end{array}
\right. \mbox{ for }j\geq 1.
\end{split}\end{equation*}
Similarly to the estimates in subsection \ref{sec401}, $\beta_{j,l,2}^i$ and $\beta_{j,l,3}^i$ can be bounded by
\begin{equation*}
|\beta_{j,l,2}^i|+|\beta_{j,l,3}^i|\leq \left\{\begin{array}{l}
Ce^{-\frac{(n-2\gamma)L}{2}(1+\xi)}e^{-\sigma t^i_j},\\
C\lambda_j^ie^{-\frac{(n-2\gamma)L}{2}(1+\xi)}e^{-\sigma t_j^i}.
\end{array}
\right.
\end{equation*}
Hence we get the desired bounds on $\beta_{j,l}^i$.
\end{proof}

\section{Derivatives of the coefficients $\beta_{j,l}^i$ with respect to the variation of $\{a_{j,l}^i\}$ and $\{r_{j,l}^i\}$}\label{sec5}

Here study the derivatives of the coefficients $\beta_{j,l}^i$ with respect to the parameters $\{a_{j,l}^i\}$ and $\{r_{j,l}^i\}$. As in the previous remark, we only need to care about the perturbation of $a_j^i, r_j^i$. Thus we first consider the derivatives of $\beta_{j,l}^i$ with respect to $a_j^i,r_j^i $ for the special configuration space that $a_j^i, r_j^i=0$ for fixed $i$. For this, we need to  consider the variation of $\phi$ with respect to these parameters.
\subsection{Derivatives of $\beta_{j,l}^i$ for $a_j^i,r_j^i$ all equal to zero}

In this section, we fix $i$ and let $\bar{u}^0_i$ to be the approximate solution with $a_j^i,r_j^i=0$. Given $\phi $ as in Proposition \ref{proposition301} for the approximate solution $\bar{u}=\bar{u}_0^i$, we introduce the operator
\begin{equation}\label{eq501}
\tilde{\mathbb L}=(-\Delta)^\gamma -c_{n,\gamma}\beta(\bar{u}_i^0+\phi)^{\beta-1}.
\end{equation}
Denote by $\xi_{j,0}^i=r_j^i$, and $ \xi_{j,l}^i=a_{j,l}^i$ for $l=1,\cdots,n$.

\begin{lemma}\label{lemma501}
For $L$ large, let $\bar{u}_i^0$ and $\phi$ be as above. Then we have the following estimates on $\frac{\partial \phi}{\partial \xi_{j,l}^i} $ near $p_i$:
\begin{equation}\label{equation501}
\Big|\frac{\partial \phi}{\partial \xi_{j,l}^i}\Big|\leq
\left\{\begin{array}{l}
Ce^{-\frac{(n-2\gamma)L}{4}(1+\xi)}|x-p_i|^{-\frac{n-2\gamma}{2}}e^{-\sigma |t^i-t^i_j|} \ \mbox{ for }l=0,\\
\frac{C}{\lambda_j^i}e^{-\frac{(n-2\gamma)L}{4}(1+\xi)}|x-p_i|^{-\frac{n-2\gamma}{2}}e^{-\sigma |t^i-t^i_j|} \ \mbox{ for }l\geq 1,
\end{array}
\right.
\ \mbox{ in }B(p_i,1).
%\Big\|\frac{\partial \phi}{\partial \xi_{j,l}^i}\Big\|_*\leq Ce^{-\frac{(n-2\gamma)L}{4}(1+\xi)}e^{-\sigma t_j}.
\end{equation}

\end{lemma}
\begin{proof}
We first consider the case $l=0$. If we differentiate the first equation in  Proposition \ref{proposition301} with respect to $r_j^i$, after some manipulation, we obtain that
\begin{equation*}
\tilde{\mathbb L}\Big(\frac{\partial \phi}{\partial r_j^i}\Big)=\tilde{h}+\sum_{j',l,i'}\frac{\partial c_{j',l}^{i'}}{\partial r_j^i}(w_{j'}^{i'})^{\beta-1}Z_{j',l}^{i'},
\end{equation*}
where
\begin{equation}\label{tilde_h}
\tilde{h}=-\frac{\partial S(\bar{u}_i^0)}{\partial r_j^i}+c_{n,\gamma}\beta[(\bar{u}_i^0+\phi)^{\beta-1}-(\bar{u}_i^0)^{\beta-1}]\frac{\partial \bar{u}_i^0}{\partial r_j^i}+\sum_{l}c_{j,l}^i\frac{\partial }{\partial r_j^i} [(w_j^i)^{\beta-1}Z_{j,l}^i].
\end{equation}

We now introduce two new norms:
\begin{equation*}\begin{split}
\|\phi\|_{*_\sigma}&=\big\| \, |x-p_i|^{\frac{n-2\gamma}{2}}e^{\sigma |t^i-t_j^i|}\phi\big\|_{\mathcal C^{2\gamma+\alpha}(B(p_i,1))}+\sum_{i'\neq i}\big\| \, |x-p_{i'}|^{\frac{n-2\gamma}{2}}\phi\|_{\mathcal C^{2\gamma+\alpha}(B(p_{i'},1))}\\
&+\big\| \,|x|^{n-2\gamma}\phi\,\big\|_{\mathcal C^{2\gamma+\alpha}(\R^n \backslash\cup_{i'} B(p_{i'},1))}
\end{split}\end{equation*}
and
\begin{equation*}\begin{split}
\|\phi\|_{{**}_\sigma}&=\big\| \,|x-p_i|^{\frac{n+2\gamma}{2}}e^{\sigma |t^i-t_j^i|}\phi\big\|_{\mathcal C^{2\gamma+\alpha}(B(p_i,1))}+\sum_{i'\neq i}\big\|\,|x-p_{i'}|^{\frac{n+2\gamma}{2}}\phi\big\|_{\mathcal C^{2\gamma+\alpha}(B(p_{i'},1))}\\
&+\big\|\,|x|^{n+2\gamma}\phi\big\|_{\mathcal C^{2\gamma+\alpha}(\R^n \backslash \cup_{i'} B(p_{i'},1))},
\end{split}\end{equation*}
where $t^i=-\log|x-p_i|$ and $\sigma>0$ is a small positive constant to be determined later.

Similarly to the proof of Lemma \ref{lemma301}, if we work in the above weighted norm spaces, one can check that given $\|\tilde h\|_{{**}_\sigma}<+\infty$, the following problem is solvable:
\begin{equation*}
\left\{\begin{array}{l}
\tilde{\mathbb L}v=\tilde h+\displaystyle{\sum_{j,l,i}}c_{j,l}^i (w_j^i)^{\beta-1}Z_{j,l}^i,\\
\int_{\R^n}v (w_j^i)^{\beta-1}Z_{j,l}^i\,dx=0, \quad j=0,1,\ldots,\ l=0,\ldots,n,
\end{array}
\right.
\end{equation*}
and the solution $v$ satisfies $\|v\|_{*_\sigma}\leq C\|\tilde h\|_{{**}_\sigma}$ where $C$ only depends on $\sigma$. We would like to apply this estimate to (\ref{eq501}), but we do not have the orthogonality condition on $\frac{\partial \phi}{\partial r_j^i}$. This can be recovered by adding some corrections.

For this, the $L^2$-product of $\frac{\partial \phi}{\partial r_j^i}$ and $(w_j^i)^{\beta-1}Z_{j,l}^i$ can be estimated as follows: differentiating the second equation in Proposition \ref{proposition301} with respect to $r_j^i$, we obtain by the estimate satisfied by $\phi$ given in \eqref{decay-phi} that
\begin{equation}\label{eq503}
\Big|\int_{\R^n}\frac{\partial \phi}{\partial r_j^i}(w_{j'}^i)^{\beta-1}Z_{j',l}^i\,dx\Big|=\Big|-\int_{\R^n }\phi\frac{\partial }{\partial r_j^i}[(w_{j'}^i)^{\beta-1}Z_{j',l}^{i}]\,dx\Big|\leq Ce^{-\frac{(n-2\gamma)L}{4}(1+\xi)}\delta_{jj'}e^{-\sigma t^i_j},
\end{equation}
for some $\sigma>0$ independent of $L$ large.

Since for $i'\neq i$, the orthogonality is satisfied, we set
\begin{equation}\label{hat-phi}
\hat{\phi}=\frac{\partial \phi}{\partial r_j^i}+\sum_{j,l}\alpha_{j,l}^iZ_{j,l}^i,
\end{equation}
for some $\alpha_{j,l}^i\in\mathbb R$. We would like to choose the numbers $\alpha_{j,l}^i$ so that the new function $\hat{\phi}$ will satisfy the orthogonality condition. In order to have this, we need
\begin{equation*}
\int_{\R^n }\frac{\partial \phi}{\partial r_j^i}(w_{j'}^{i'})^{\beta-1}Z_{j',l'}^{i'}\,dx+\sum_{j,l}\int_{\R^n}\alpha_{j,l}^iZ_{j,l}^i(w_{j'}^{i'})^{\beta-1}Z_{j',l'}^{i'}\,dx=0.
\end{equation*}
From estimate (\ref{eq503}), one has
\begin{equation}\label{eq504}
|\alpha_{j,l}^i|\leq Ce^{-\frac{(n-2\gamma)L}{4}(1+\xi)}e^{-\sigma t^i_j}.
\end{equation}
Then $\hat{\phi}$ will satisfy the following equation
\begin{equation*}
\left\{\begin{array}{l}
\tilde{\mathbb L}(\hat{\phi})=\hat{h}+\displaystyle{\sum_{j',l',i'}}\frac{\partial c_{j',l'}^{i'}}{\partial r_j^i}(w_{j'}^{i'})^{\beta-1}Z_{j',l'}^{i'},\\
\int_{\R^n}\hat{\phi}(w_{j'}^{i'})^{\beta-1}Z_{j',l'}^{i'}\,dx=0,
\end{array}
\right.
\end{equation*}
where
\begin{equation*}
\hat{h}=\tilde{h}+c_{n,\gamma}\beta\sum_{j,l,i}\alpha_{j,l}^i Z_{j,l}^i[(\bar{u}_i^0)^{\beta-1}-(\bar{u}_i^0+\phi)^{\beta-1}].
\end{equation*}
In conclusion, to estimate $\hat{\phi}$ and hence $\frac{\partial \phi}{\partial r_j^i}$, it suffices to estimate $\hat{h}$. So we now bound $\hat{h}$ term by term.

Concerning $\frac{\partial S(\bar{u}_i^0)}{\partial r_j^i}$, we have from the arguments in Section \ref{sec3} that
\begin{equation*}%\label{equation101}
\Big\|\frac{\partial S(\bar{u}_i^0)}{\partial r_j^i}\Big\|_{{**}_\sigma }\leq Ce^{-\frac{(n-2\gamma)L}{4}(1+\xi)}.
\end{equation*}

From the estimates satisfied by $\phi$, the same estimate holds for the second term in \eqref{tilde_h} if $\sigma<\gamma_1+\frac{n-2\gamma}{2}$.  For the third term, it contains the symbol $\delta_{jj'}$, so the estimate follows by the bounds for $c_{j,l}^i$ in the proof of Lemma \ref{lemma301}.  Moreover, from (\ref{eq504}), one can get the same estimate for the fourth term. In conclusion, one has
\begin{equation*}
\|\hat{h}\|_{{**}_\sigma}\leq Ce^{-\frac{(n-2\gamma)L}{4}(1+\xi)}.
\end{equation*}
Hence by the above reasoning, $\|\hat{\phi}\|_{*_\sigma}\leq Ce^{-\frac{(n-2\gamma)L}{4}(1+\xi)}
$. Formulas \eqref{hat-phi} and \eqref{eq504} yield
\begin{equation*}
\Big\|\frac{\partial \phi}{\partial r_j^i}\Big\|_{*_\sigma}\leq Ce^{-\frac{(n-2\gamma)L}{4}(1+\xi)}.
\end{equation*}
Finally, by the definition of $\|\cdot\|_{*_\sigma}$ norm, we obtain the first assertion in (\ref{equation501}).\\

Similarly, differentiating the first equation in Proposition \ref{proposition301} with respect to $a_j^i$ and arguing as above, always keeping in mind that $\frac{\partial w_j^i}{\partial a_j^i}\sim\frac{1}{\lambda_j^i}|x-p_i|^{-\frac{n-2\gamma}{2}}(v^i_j)^{\frac{2}{n-2\gamma}+1}$, one can get that
\begin{equation*}
\Big\|\frac{\partial \phi}{\partial a_j^i}\Big\|_{*_\sigma}\leq \frac{C}{\lambda_j^i}e^{-\frac{(n-2\gamma)L}{4}(1+\xi)},
\end{equation*}
which yields the second estimate in \eqref{equation501}.
\end{proof}

We now  describe the asymptotic profile of  the function $\frac{\partial \phi}{\partial \xi_{j,l}^i}$. First of all, we consider the ideal case when there is only one point singularity at $p=0$ and $u=u_L$, i.e., the exact Delaunay solution from \eqref{exact-Delaunay}. By definition, $u_L$ is a solution of \eqref{eq303} with $\phi=0$ and vanishing right hand side. For $j\geq 1$, we assume that we are varying $w_j$ by $a_j, r_j$, and denote the corresponding approximate solution by $\bar{u}_L$. We are still able to perform the reduction in Proposition \ref{proposition301} to find a solution of the form $\bar{u}_L+\bar{\phi}$ of the following equation
\begin{equation}\label{equation100}
\left\{\begin{array}{l}
(-\Delta)^\gamma (\bar{u}_L+\bar{\phi})-c_{n,\gamma}(\bar{u}_L+\bar{\phi})^\beta
=\displaystyle{\sum_{{j,l}}}c_{j,l}w_j^{\beta-1}Z_{j,l},\\
\int_{\R^n}\bar{\phi}w_j^{\beta-1}Z_{j,l}\,dx=0.
\end{array}
\right.
\end{equation}
Note that an estimate similar to that of Lemma \ref{lemma501} will hold true for the corresponding $\bar\phi$. But we also need to control the derivative of $c_{j,l}$ with respect to the perturbations. In order to do so, we first introduce some notation.
Define
\begin{equation*}
\beta_{j,l}:=\int_{\R^n}[(-\Delta)^\gamma (\bar{u}_L+\bar{\phi})-c_{n,\gamma}(\bar{u}_L+\bar{\phi})^\beta]{Z}_{j,l}\,dx.
\end{equation*}
We are interested in the derivatives of $\beta_{j',l'}$ with respect to $\{\xi_{j,l}\}=\{r_j, a_{j,1}, \cdots,a_{j,n}\}$ for $l=0,\cdots,n$.

\begin{lemma}\label{lemma502}
For $L$ large, the following estimates hold:
\begin{equation*}\begin{split}
&\frac{\partial \beta_{j,0}}{\partial r_j}=-2c_{n,\gamma}F'(L)+O(e^{-\frac{(n-2\gamma)L}{2}(1+\xi)}),\ \ \\
& \frac{\partial \beta_{j-1,0}}{\partial r_j}=c_{n,\gamma}F'(L)+O(e^{-\frac{(n-2\gamma)L}{2}(1+\xi)}),  \\
&\frac{\partial \beta_{j+1,0}}{\partial r_j}=c_{n,\gamma}F'(L)+O(e^{-\frac{(n-2\gamma)L}{2}(1+\xi)}), \\
&\frac{\partial \beta_{J,0}}{\partial r_j}=O(e^{-\frac{(n-2\gamma)L}{2}(1+\xi)}e^{-\sigma |t_J-t_j|}), \mbox{ for }|J-j|\geq 2.
\end{split}\end{equation*}
And for $l=1,\ldots,n$,
\begin{equation*}\begin{split}
&\frac{\partial \beta_{j,l}}{\partial a_{j,l}}
=c_{n,\gamma}\lambda_j\sum_{j'\neq j}\min\big\{\frac{\lambda_{j'}}{\lambda_j},\frac{\lambda_j}{\lambda_{j'}}\big\}^{\frac{n-2\gamma}{2}}
\frac{1}{\max\{\lambda_{j'}^2,\lambda_j^2\}}+O(e^{-\frac{(n-2\gamma)L}{2}(1+\xi)}),\ \ \\
&\frac{\partial \beta_{J,l}}{\partial a_{j,l}}=-c_{n,\gamma}\lambda_J\Big[\min\big\{\frac{\lambda_J}{\lambda_j},\frac{\lambda_j}{\lambda_J}\big\}
^{\frac{n-2\gamma}{2}}\frac{1}{\max\{\lambda_J^2,\lambda_j^2\}}
+O\big(\frac{1}{\lambda_i}e^{-\frac{(n-2\gamma)L}{2}(1+\xi)}e^{-\sigma |t_J-t_j|}\big)
\Big], \ \mbox{if } j\neq J,
%&&\frac{\partial \beta_{j-1,l}}{\partial a_{j,l}}=-c_{n,\gamma}e^{-\frac{(n-2\gamma)L}{2}}\frac{1}{\lambda_{j-1}^2}+O(e^{-\frac{(n-2\gamma)L}{2}(1+\xi)}e^{-\sigma t_j})],\ \ \\
%&&\frac{\partial \beta_{j+1,l}}{\partial a_{j,l}}=-c_{n,\gamma}e^{-\frac{(n-2\gamma)L}{2}}\frac{1}{\lambda_j^2}+O(e^{-\frac{(n-2\gamma)L}{2}(1+\xi)}e^{-\sigma t_j}) ,\\
%&&\frac{\partial \beta_{J,0}}{\partial r_j}=O(e^{-\frac{(n-2\gamma)L}{2}(1+\xi)}e^{-\sigma|t_J-t_j|}), \\
%&&\frac{\partial \beta_{J,l}}{\partial a_{j,l}}=O\big(e^{-\frac{(n-2\gamma)L}{2}(1+\xi)}e^{-\sigma|t_J-t_j|}\frac{1}{\max\{\lambda_j^2,\lambda_J^2\}}\big),\mbox{ if }|j-J|\geq 2\\
\end{split}\end{equation*}
In addition,
\begin{equation*}
\frac{\partial \beta_{J,l}}{\partial \xi_{j,l'}}=0, \  \mbox{if }l\neq l'.
\end{equation*}
Here the derivatives are evaluated at $a_j,r_j=0$.
\end{lemma}

\begin{proof}

Differentiating the expression for $\beta_{J,l}$ with respect to $r_j$, and recalling equation \eqref{equation100} one has
\begin{equation}\label{ecuacion14}
\frac{\partial \beta_{J,l}}{\partial r_j}=\int_{\R^n}\Big[\bar{\mathbb L}\big(\frac{\partial \bar{\phi}}{\partial r_j}\big)+\frac{\partial S(\bar{u}_L)}{\partial r_j}\Big]Z_{J,l}\,dx+\sum_{j'',l'}c_{j'',l'}
\int_{\R^n}w_{j''}^{\beta-1}Z_{j'',l'}\frac{\partial {Z}_{J,l}}{\partial r_j}\,dx,
\end{equation}
where we have defined
\begin{equation*}
\bar{\mathbb L}=(-\Delta)^\gamma-c_{n,\gamma}\beta u_L^{\beta-1},
\end{equation*}
since $\bar{\phi}=0$ when $\bar{u}_L=u_L$.
We write
\begin{equation*}\begin{split}
\int_{\R^n}\bar{\mathbb L}\big(\frac{\partial \bar{\phi}}{\partial r_j}\big){Z}_{J,l}\,dx&=\int_{\R^n}\mathbb L_{J}({Z}_{J,l})\frac{\partial \bar{\phi}}{\partial r_j}\,dx+\int_{\R^n}(\bar{\mathbb L}-\mathbb L_{J})({Z}_{J,l})\frac{\partial \bar{\phi}}{\partial r_j}\,dx+O(e^{-\frac{(n-2\gamma)L}{2}(1+\xi)}e^{-\sigma |t_{J}-t_j|})\\
&=c_{n,\gamma}\beta\int_{\R^n}[w_{J}^{\beta-1}-u_L^{\beta-1}]\frac{\partial \bar{\phi}}{\partial r_j}{Z}_{J,l}\,dx+O(e^{-\frac{(n-2\gamma)L}{2}(1+\xi)}e^{-\sigma |t_{J}-t_j|})\\
&=O(e^{-\frac{(n-2\gamma)L}{2}(1+\xi)}e^{-\sigma| t_{J}-t_j|}),
\end{split}\end{equation*}
where $\mathbb L_{J}=(-\Delta)^\gamma-c_{n,\gamma}\beta w_{J}^{\beta-1}$.

Since $u_L$ is the exact solution, the corresponding $c_{j'',l}=0$. Thus for the last term in \eqref{ecuacion14} we have
\begin{equation*}
\sum_{j'',l}c_{j'',l}\int_{\R^n}w_{j''}^{\beta-1}Z_{j'',l'}\frac{\partial {Z}_{J,l}}{\partial r_j}\,dx=0.
\end{equation*}
In conclusion, one has
\begin{equation*}
\frac{\partial \beta_{J,l}}{\partial r_j}=\frac{\partial }{\partial r_j}\int_{\R^n }S(\bar{u}_L){Z}_{J,l}\,dx+O(e^{-\frac{(n-2\gamma)L}{2}(1+\xi)}e^{-\sigma |t_{J}-t_j|}),
\end{equation*}
where we have used the fact that $S(\bar{u}_L)=0$ for $\bar{u}_L=u_L$.

Similarly, recalling the definition of $Z_{J,l}=\lambda_j\frac{\partial w_{J}}{\partial a_{J,l}}$, and the estimates for $\frac{\partial \bar{\phi}}{\partial a_j}$ from the previous paragraphs,
\begin{equation*}
\frac{\partial \beta_{J,l}}{\partial a_j}=\frac{\partial }{\partial a_j}\int_{\R^n }S(\bar{u}_L){Z}_{J,l}\,dx
+O\big(\frac{\lambda_{J}}{\lambda_j}e^{-\frac{(n-2\gamma)L}{2}(1+\xi)}e^{-\sigma |t_{J}-t_j|}\big).
\end{equation*}
Both variations above can be calculated from Lemma \ref{lemma402}, with the obvious modifications as we just have one singular point so there is no summation in $i$.  Thus one has
\begin{equation*}
\frac{\partial \beta_{J,0} }{\partial r_j}=c_{n,\gamma}\sum_{j'\neq j'}\frac{\partial }{\partial r_j}\Big[F(|\log \tfrac{\lambda_{j'}}{\lambda_{J}}|)\frac{\log\frac{\lambda_{j'}}{\lambda_{J}}}{|\log \frac{\lambda_{j'}}{\lambda_{J}}|}\Big]+O(e^{-\frac{(n-2\gamma)L}{2}(1+\xi)}e^{-\sigma |t_{J}-t_j|}).
\end{equation*}
The first four conclusions in Lemma \ref{lemma502} follow by taking different values of $J$ and from the definition of $\lambda_j$.

Very similarly, for $l=1,\ldots,n$, applying Lemma \ref{lemma402} we obtain
\begin{equation*}\begin{split}
\frac{\partial \beta_{J,l}}{\partial a_{j,l}}
%&=c_{n,\gamma}\sum_{J\neq j'}\frac{\partial }{\partial a_{j,l}}\int_{\R^n}\beta w_{j'}^{\beta-1}w_J{Z}_{j',l}\,dx
%+O\big(\frac{\lambda_{j'}}{\lambda_j}e^{-\frac{(n-2\gamma)L}{2}(1+\xi)}e^{-\sigma |t_{j'}-t_j|}\big)\\
&=O\big(\dfrac{\lambda_{J}}{\lambda_j}e^{-\frac{(n-2\gamma)L}{2}(1+\xi)}e^{-\sigma |t_{J}-t_j|}\big)\\
&\quad+\left\{\begin{array}{l}
c_{n,\gamma}\lambda_j\dfrac{\partial }{\partial a_{j,l}}\displaystyle{\sum_{j'\neq j}}\Big[\min\big\{\frac{\lambda_{j'}}{\lambda_j},\frac{\lambda_j}{\lambda_{j'}}\big\}
^{\frac{n-2\gamma}{2}}
\frac{a_{j,l}}{\max\{\lambda_{j'}^2,\lambda_j^2\}}\Big] \ \mbox{ if }J=j,\\
-c_{n,\gamma}\lambda_{J}\displaystyle{\frac{\partial }{\partial a_{j,l}}}\Big[\min\big\{\frac{\lambda_{J}}{\lambda_j},\frac{\lambda_j}{\lambda_{J}}\big\}
^{\frac{n-2\gamma}{2}}
\frac{a_{j,l}}{\max\{\lambda_{J}^2,\lambda_j^2\}}\Big]
\ \mbox{ if }J\neq j.
\end{array}
\right.
\end{split}\end{equation*}
In addition, by the symmetry of the problem, we have $\frac{\partial \beta_{J,l'}}{\partial \xi_{j,l}}=0$ if $l\neq l'$. This completes the proof of the Lemma.

\end{proof}

The reason we have studied the special configuration $u_L$ is that we will identify the quantities $\frac{\partial \beta_{j,l}^i}{\partial \xi_{j,l}^i}$ as the limits of the derivatives of $\beta_{j,l}$ with respect to $\xi_{j,l}$ as $j\to \infty$. We fix a point $p=p_i$ and a Delaunay parameter $L=L_i$. We denote $\bar{u}_{L_i}$, $\bar{\phi}_i$ the pair that gives the solution to \eqref{equation100}. Before we state the result, we first need to compare the functions                                                   $\frac{\partial \phi}{\partial \xi_{j,l}^i}$ and $\frac{\partial \bar{\phi}_i}{\partial \xi_{j,l}}$ for $i$ fixed, as in the lemma below:

\begin{lemma}\label{lemma503}
Take $\bar{u}_i^0$ and $\phi$ as in Lemma \ref{lemma501}. For $i$ fixed and $j\geq 1$ we have the following estimate:
\begin{equation*}
\Big\|\frac{\partial \phi}{\partial \xi_{j,l}^i}-\frac{\partial \bar{\phi}_i}{\partial \xi_{j,l}^i}\Big\|_{*_\sigma}\leq \left\{\begin{array}{c}
Ce^{-\frac{(n-2\gamma)L}{4}(1+\xi)}e^{-\sigma t_j^i} \  \mbox{ for }l=0,\\
\frac{C}{\lambda_j^i}e^{-\frac{(n-2\gamma)L}{4}(1+\xi)}e^{-\sigma t_j^i} \ \mbox{ for }l\geq 1.
\end{array}
\right.
%\Big\|\frac{\partial \phi}{\partial \xi_{j,l}^i}-\frac{\partial \bar{\phi}}{\partial \xi_{j,l}}\Big\|_*\leq Ce^{-\frac{(n-2\gamma)L}{4}(1+\xi)}e^{-\sigma t_j}.
\end{equation*}
\end{lemma}

\begin{proof}
As before, we write down the equations satisfied by $\phi$ and $\bar{\phi}_i$,
\begin{equation*}
(-\Delta)^\gamma (\bar{u}_i^0+\phi)-c_{n,\gamma}(\bar{u}_i^0+\phi)^\beta=\sum_{i',j,l} c_{j,l}^{i'}(w_j^{i'})^{\beta-1}Z_{j,l}^{i'}
\end{equation*}
and
\begin{equation*}
(-\Delta)^\gamma (\bar{u}_{L_i}+\bar{\phi}_i)-c_{n,\gamma}(\bar{u}_{L_i}+\bar{\phi}_i)^\beta=\sum_{j,l} c_{j,l}(w_j^i)^{\beta-1}Z_{j,l}^i.
\end{equation*}
We will differentiate both expressions with respect to $\xi_{j,l}^i$; one has from the first equation that
\begin{equation*}\begin{split}
&(-\Delta)^\gamma \big(\frac{\partial \bar{u}_i^0}{\partial \xi_{j,l}^i}+\frac{\partial \phi}{\partial \xi_{j,l}^i}\big)-c_{n,\gamma}\beta (\bar{u}_i^0+\phi)^{\beta-1}\big(\frac{\partial \bar{u}_i^0}{\partial \xi_{j,l}^i}+\frac{\partial \phi}{\partial \xi_{j,l}^i}\big)\\
&\quad\quad= \sum_{i'}\sum_{j=0}^\infty c_{j,l}^{i'}\frac{\partial }{\partial \xi_{j,l}^i}[(w_j^{i'})^{\beta-1}Z_{j,l}^{i'}]+\sum_{i'}\sum_{j'=0}^\infty \frac{\partial c_{j',l}^{i'}}{\partial \xi_{j,l}^i}(w_{j'}^{i'})^{\beta-1}Z_{j',l}^{i'}
\end{split}\end{equation*}
Here, by the definition of the approximate solution $\bar{u}_i^0$, one knows that $\frac{\partial \bar{u}_i^0}{\partial \xi_{j,l}^i}=\frac{\partial w_j^i}{
\partial \xi_{j,l}^i}$.

Next, differentiating the second equation,
\begin{equation*}\begin{split}
&(-\Delta)^\gamma (\frac{\partial \bar{u}_{L_i}}{\partial \xi_{j,l}}+\frac{\partial \bar{\phi}_i}{\partial \xi_{j,l}})-c_{n,\gamma}\beta(\bar{u}_{L_i}+\bar{\phi}_i)^{\beta-1}(\frac{\partial \bar{u}_{L_i}}{\partial \xi_{j,l}}+\frac{\partial \bar{\phi}_i}{\partial \xi_{j,l}})\\
&\quad\quad=\sum_{j'=-\infty}^\infty\frac{\partial c_{j',l'}}{\partial \xi_{j,l}}(w_{j'}^i)^{\beta-1}Z_{j,l}+\sum_{j=-\infty}^\infty c_{j,l}\frac{\partial }{\partial \xi_{j,l}}[w_j^{\beta-1}Z_{j,l}].
\end{split}\end{equation*}
To simplify this expression, recall that when $\bar{u}_{L_i}=u_{L_i}$ is  a exact solution, one has $c_{j,l}=0, \ \bar{\phi}_i=0$. So when evaluating at  $\xi_{j,l}=0$, this equation becomes
\begin{eqnarray*}
(-\Delta)^\gamma \frac{\partial \bar{\phi}_i}{\partial \xi_{j,l}}-c_{n,\gamma}u_{L_i}^{\beta-1}\frac{\partial \bar{\phi}_i}{\partial \xi_{j,l}}+(-\Delta )^\gamma \frac{\partial w_j}{\partial \xi_{j,l}}-c_{n,\gamma}\beta u_{L_i}^{\beta-1}\frac{\partial w_j }{\partial \xi_{j,l}} =\sum_{j'=-\infty}^\infty\frac{\partial c_{j',l'}}{\partial \xi_{j,l}}(w_{j'}^i)^{\beta-1}Z_{j,l}
\end{eqnarray*}
Denote $\mathbb L_{\bar{u}_i^0}=(-\Delta)^\gamma-c_{n,\gamma}\beta (\bar{u}_i^0)^{\beta-1}$. Taking the difference of the above two expressions we obtain an equation for $\frac{\partial \phi}{\partial \xi_{j,l}^i}-\frac{\partial \bar{\phi}}{\partial \xi_{j,l}}$:

\begin{equation*}\begin{split}
\mathbb L&_{\bar{u}_i^0}\Big(\frac{\partial \phi}{\partial \xi_{j,l}^i}-\frac{\partial \bar{\phi}_i}{\partial \xi_{j,l}}\Big)\\
&=c_{n,\gamma}\beta\Big[\big((\bar{u}_i^0+\phi)^{\beta-1}-u_{L_i}^{\beta-1}\big)\frac{\partial w_{j}^i}{\partial \xi_{j,l}^i}+\big((\bar{u}_i^0)^{\beta-1}-u_{L_i}^{\beta-1}\big)\frac{\partial \bar{\phi}_i}{\partial \xi_{j,l}}\Big]+c_{n,\gamma}\beta[(\bar{u}_i^0+\phi)^{\beta-1}-(\bar{u}_i^0)^{\beta-1}]\frac{\partial \phi}{\partial \xi_{j,l}^i}\\
&\quad+\sum_{i'}\sum_{j=0}^\infty c_{j,l}^{i'}\frac{\partial }{\partial \xi_{j,l}^i}[(w_j^{i'})^{\beta-1}Z_{j,l}^{i'}]+\sum_{i'}\sum_{j'=0}^\infty \frac{\partial c_{j',l}^{i'}}{\partial \xi_{j,l}^i}(w_{j'}^{i'})^{\beta-1}Z_{j',l}^{i'}-\sum_{j'=-\infty}^\infty\frac{\partial c_{j',l'}}{\partial \xi_{j,l}}(w_{j'}^i)^{\beta-1}Z_{j,l}\\
&=c_{n,\gamma}\beta\Big[((\bar{u}_i^0+\phi)^{\beta-1}-u_{L_i}^{\beta-1})\frac{\partial w_{j}^i}{\partial \xi_{j,l}^i}+((\bar{u}_i^0)^{\beta-1}-u_{L_i}^{\beta-1})\frac{\partial \bar{\phi}_i}{\partial \xi_{j,l}}\Big]
+c_{n,\gamma}\beta[(\bar{u}_i^0+\phi)^{\beta-1}-(\bar{u}_i^0)^{\beta-1}]\frac{\partial \phi}{\partial \xi_{j,l}^i}\\
&\quad+\sum_{j=0}^\infty c_{j,l}^i\frac{\partial }{\partial \xi_{j,l}^i}[(w_j^i)^{\beta-1}Z_{j,l}^i]+\sum_{j'<0}\frac{\partial c_{j',l'}}{\partial \xi_{j,l}}(w_{j'}^i)^{\beta-1}Z_{j',l'}\\
&\quad+\sum_{j'=0}^\infty\big(\frac{\partial  c_{j',l'}^i}{\partial \xi_{j,l}^i}-\frac{\partial c_{j',l'}}{\partial \xi_{j,l}}\big)(w_{j'}^i)^{\beta-1}Z_{j',l'}^i+\sum_{i'\neq i}\frac{\partial c_{j',l'}^{i'}}{\partial \xi_{j,l}^i}(w_{j'}^{i'})^{\beta-1}Z_{j',l'}^{i'}.
\end{split}\end{equation*}
Neglecting the terms in the last line, taking into account the estimates in Section \ref{sec3}, the estimates for $\frac{\partial \phi}{\partial r_j^i}$ in Lemma \ref{lemma501} and the estimates in Lemma \ref{lemma502}, one can find that the right hand side of the above equation can be bounded by $e^{-\frac{(n-2\gamma)L}{2}(1+\xi)}e^{-\sigma t_j^i}$ in $\|\cdot\|_{{**}_\sigma}$ norm.

We first consider the case $l=0$, i.e. $\xi_{j,l}^i=r_j^i$. In this case, to have control on $\frac{\partial \phi}{\partial r_j^i}-\frac{\partial \bar{\phi}_i}{\partial r_j^i}$, we can reason similarly as in Lemma \ref{lemma501}. More precisely, we first set
\begin{equation*}
\hat{\phi}=\frac{\partial \phi}{\partial r_j^i}-\frac{\partial \bar{\phi}_i}{\partial r_j^i}+\sum_{i',j',l'} \alpha_{j',l'}^{i'}Z_{j',l'}^{i'}.
\end{equation*}
In order to get the orthogonality condition $\int_{\R^n}\hat{\phi}\,(w_{J}^{I})^{\beta-1}Z_{J,L}^{I}\,dx=0$ for every $I,J,L$. we need
\begin{equation}\label{ecuacion13}
\int_{\R^n}\Big[\frac{\partial \phi}{\partial r_j^i}-\frac{\partial \bar{\phi}_i}{\partial r_j^i}\Big](w_{J}^{I})^{\beta-1}Z_{J,L}^{I}\,dx=\sum_{i',j',l'} \alpha_{j',l'}^{i'}\int_{\R^n}(w_{J}^{I})^{\beta-1}Z_{j',l'}^{i'}Z_{J,L}^I\,dx.
\end{equation}
Differentiating the orthogonality condition of $\phi$ and $\bar{\phi}_i$ w.r.t $r_j^i$, in analogy with (\ref{eq503}), we get
\begin{equation*}\begin{split}
\Big|\int_{\R^n}\Big[\frac{\partial \phi}{\partial r_j^i}-\frac{\partial \bar{\phi}_i}{\partial r_j^i}\Big](w_{J}^{I})^{\beta-1}Z_{J,L}^{I}\,dx\Big|
&=\sum_{I,J,L}\Big|\int_{\R^n}(\phi-\bar{\phi}_i)\frac{\partial }{\partial r_j^i}[(w_{J}^{I})^{\beta-1}Z_{J,L}^{I}]\,dx\Big|\\
&\leq \left\{\begin{array}{l}
Ce^{-\frac{(n-2\gamma)L}{4}(1+\xi)}e^{-\sigma t^i_j} \ \mbox{ if }j=J,\\
0 \ \mbox{ if }j\neq J,
\end{array}
\right.
\end{split}\end{equation*}
where we have used the fact that $\bar{\phi}_i=0$ for $u_{L_i}$.
Therefore, from \eqref{ecuacion13} we have the following estimates:
\begin{equation}\label{equation502}
|\alpha_{j',l'}^{i'}|\leq \left\{\begin{array}{l}
Ce^{-\frac{(n-2\gamma )L}{4}(1+\xi)}e^{-\sigma t^i_j} \ \mbox{ if }i'=i,\ j'=j,\\
Ce^{-\frac{(n-2\gamma )L}{4}(1+\xi)}e^{-\sigma |t_j^i-t_{j'}^i|} \ \mbox{ if }i=i', \ j'\neq j,\\
0\mbox{ if }i\neq i'.
\end{array}
\right.
\end{equation}
Moreover, $\hat{\phi}$ solves
\begin{equation*}\begin{split}
\mathbb L_{\bar{u}_i^0}(\hat{\phi})&=\mathbb L_{\bar{u}_i^0}\big(\frac{\partial \phi}{\partial r_{j}^i}-\frac{\partial \bar{\phi}_i}{\partial r_j^i}\big)
+\mathbb L_{\bar{u}_i^0}(\sum_{j',l'} \alpha_{j',l'}^{i} Z_{j',l'}^{i})\\
&=\mathbb L_{\bar{u}_i^0}\big(\frac{\partial \phi}{\partial r_j^i}-\frac{\partial \bar{\phi}_i}{\partial r_j^i}\big)-c_{n,\gamma}\beta\sum_{j',l'}\alpha_{j',l'}^i((\bar{u}_i^0)^{\beta-1}-(w_{j'}^i)^{\beta-1})Z_{j',l'}^i.
\end{split}\end{equation*}
By estimate (\ref{equation502}), we know that the second term is bounded by $e^{-\frac{(n-2\gamma)L}{2}(1+\xi)}e^{-\sigma t_j^i}$ in $\|\cdot\|_{{**}_\sigma}$ norm. So the right hand side of the above equation can be controlled by $e^{-\frac{(n-2\gamma)L}{4}(1+\xi)}e^{-\sigma t^i_j } $ in $\|\cdot\|_{{**}_\sigma}$ norm, and thus, applying Proposition \ref{proposition301} to $\hat \phi$,
\begin{equation*}
\|\hat{\phi}\|_{*_\sigma}\leq Ce^{-\frac{(n-2\gamma)L}{4}(1+\xi)}e^{-\sigma t^i_j }.
\end{equation*}
Looking back at $\eqref{hat-phi}$, from the estimates for $\alpha_{j,l}^i$ we have
\begin{equation*}
\Big\|\frac{\partial \phi}{\partial r_j^i}-\frac{\partial \bar{\phi}_i}{\partial r_j^i}\Big\|_{*_\sigma}\leq Ce^{-\frac{(n-2\gamma)L}{4}(1+\xi)}e^{-\sigma t^i_j }.
\end{equation*}

A similar argument yields
\begin{equation*}
\Big\|\frac{\partial \phi}{\partial a_j^i}-\frac{\partial \bar{\phi}_i}{\partial a_j^i}\Big\|_{*_\sigma}\leq \frac{C}{\lambda_j^i}e^{-\frac{(n-2\gamma)L}{4}(1+\xi)}e^{-\sigma t^i_j }.
\end{equation*}
The proof of the Lemma is completed.
\end{proof}

From the previous lemma we can obtain estimates on derivatives of $\beta_{j,l}^i$ with respect to $\xi_{j,l}^i$.
\begin{lemma}\label{lemma504}
In the previous setting, we have the following estimates
\begin{equation*}
\Big|\frac{\partial \beta_{j',l'}^i}{\partial \xi_{j,l}^i}-\frac{\partial \beta_{j',l'}}{\partial \xi_{j,l}}\Big|
\leq \left\{\begin{array}{l}
Ce^{-\frac{(n-2\gamma)L}{2}(1+\xi)}e^{-\sigma t^i_j}e^{-\sigma |t^i_j-t^i_{j'}|} \ \mbox{ if }l=0,\\
\frac{C}{\lambda_j^i}e^{-\frac{(n-2\gamma)L}{2}(1+\xi)}e^{-\sigma t^i_j}e^{-\sigma |t^i_j-t^i_{j'}|} \ \mbox{ if }l\geq 1,
\end{array}
\right. \mbox{ for }j\geq 1.
\end{equation*}
\end{lemma}

\begin{proof}
Recall that by the definition of $\beta_{j,l}^i$ and $\beta_{j,l}$,
\begin{eqnarray*}
&&\beta_{j,l}^i=\int_{\R^n}[(-\Delta )^\gamma (\bar{u}_i^0+\phi)-c_{n,\gamma}(\bar{u}_i^0+\phi)^\beta]{Z}_{j,l}^i\,dx,\\
&&\beta_{j,l}=\int_{\R^n}[(-\Delta )^\gamma (\bar{u}_{L_i}+\bar{\phi}_i)-c_{n,\gamma}(\bar{u}_{L_i}+\bar{\phi}_i)^{\beta}]{Z}_{j,l}^i\,dx.
\end{eqnarray*}
Differentiating the above equations w.r.t $\xi_{j,l}^i$ and taking the difference, one has
\begin{equation*}\begin{split}
\frac{\partial }{\partial \xi_{j,l}^i}&(\beta_{j',l'}^i-\beta_{j',l'})\\
&=\int_{\R^n}\mathbb L_{\bar{u}_i^0}(\frac{\partial \phi}{\partial \xi_{j,l}^i}-\frac{\partial \bar{\phi}_i}{\partial \xi_{j,l}^i}){Z}_{j',l'}^i\,dx-c_{n,\gamma}\beta\int_{\R^n}((\bar{u}_i^0)^{\beta-1}-u_{L_i}^{\beta-1})\frac{\partial \bar{\phi}_i}{\partial \xi_{j,l}^i}{Z}_{j',l'}^i\,dx\\
&-c_{n,\gamma}\beta\int_{\R^n }(u_{L_i}^{\beta-1}-(\bar{u}_i^0+\phi)^{\beta-1})\frac{\partial w_j^i}{\partial \xi_{j,l}^i}{Z}_{j',l'}^i\,dx-c_{n,\gamma}\beta\int_{\R^n}[(\bar{u}_i^0+\phi)^{\beta-1}-(\bar{u}_i^0)^{\beta-1}]\frac{\partial \phi}{\partial \xi_{j,l}^i}{Z}_{j',l'}^i\,dx\\
&+c_{j,l}^i\int_{\R^n}(w_j^i)^{\beta-1}Z_{j,l}^i\frac{\partial}{\partial \xi_{j,l}^i}{Z}_{j',l'}^i\,dx.
\end{split}\end{equation*}
By oddness, one can first get that the term in the last line vanishes. Moreover, by the estimates in Lemma \ref{lemma501} and \ref{lemma503}, one can get that the first two lines can be controlled by $e^{-\frac{(n-2\gamma)L}{2}(1+\xi)}e^{-\sigma t^i_j}e^{-\sigma |t^i_j-t^i_{j'}|}$ when $l=0$ and $\frac{1}{\lambda_j^i}e^{-\frac{(n-2\gamma)L}{2}(1+\xi)}e^{-\sigma t^i_j}e^{-\sigma |t^i_j-t^i_{j'}|}$ when $l\geq 1$. Thus one has
\begin{equation*}
\Big|\frac{\partial }{\partial \xi_{j,l}^i}(\beta_{j',l'}^i-\beta_{j',l'})\Big|\leq \left\{\begin{array}{l}
Ce^{-\frac{(n-2\gamma)L}{2}(1+\xi)}e^{-\sigma t^i_j}e^{-\sigma |t^i_j-t^i_{j'}|} \ \mbox{ if }l=0,\\
\frac{C}{\lambda_j^i}e^{-\frac{(n-2\gamma)L}{2}(1+\xi)}e^{-\sigma t^i_j}e^{-\sigma |t^i_j-t^i_{j'}|} \ \mbox{ if } l\geq 1,
\end{array}
\right. \mbox{ for }j\geq 1.
\end{equation*}
The proof is completed.
\end{proof}

\subsection{Derivatives of the numbers $\beta_{j,l}^i$ for the general parameters $a_j^i,r_j^i$}

In this subsection, we consider the derivatives of the $\beta_{j,l}^i$'s for  the general parameters $a_j^i, r_j^i$ satisfying \eqref{parar}-\eqref{para2}. We write the counterpart of Lemmas \ref{lemma503} and \ref{lemma504}, but we do not prove them since the methods are quite similar. In the following, we assume that $\tau<\sigma$.
\begin{lemma}\label{lemma505}
Suppose $a_j^i, r_j^i$ satisfy \eqref{parar}-\eqref{para2}. For $L$ large, let $\bar{u}$ and $\phi$ be as in Proposition \ref{proposition301}, then we have the following estimates for $\frac{\partial \phi}{\partial \xi_{j,l}^i}-\frac{\partial \bar{\phi}_i}{\partial \xi_{j,l}}$:
\begin{equation*}
\Big\|\frac{\partial \phi}{\partial \xi_{j,l}^i}-\frac{\partial \bar{\phi}_i}{\partial \xi_{j,l}}\Big\|_{*_\sigma}\leq C\left\{\begin{array}{l}
Ce^{-\frac{(n-2\gamma)L}{4}(1+\xi)}e^{-\tau t_j^i} \ \mbox{ if }l=0,\\
\frac{C}{\lambda_j^i}e^{-\frac{(n-2\gamma)L}{4}(1+\xi)}e^{-\tau t_j^i} \ \mbox{ if }l\geq 1,
\end{array}
\right.
\end{equation*}
for $j\geq 1$ if we choose  $\tau<\sigma$.

\end{lemma}

\begin{lemma}\label{lemma506}

Suppose $a_j^i, r_j^i$ satisfy \eqref{parar}  - \eqref{para2}. For $L$ large, let $\bar{u}$ and $\phi$ be as in Proposition \ref{proposition301}, then we have the following estimates
\begin{equation*}
\Big|\frac{\partial }{\partial \xi_{j,l}^i}(\beta_{j',l'}^i-\beta_{j',l'})\Big|
\leq \left\{\begin{array}{l}
Ce^{-\frac{(n-2\gamma)L}{2}(1+\xi)}e^{-\tau t^i_j}e^{-\sigma |t^i_j-t^i_{j'}|} \ \mbox{ if }l=0,\\
\frac{C}{\lambda_j^i}e^{-\frac{(n-2\gamma)L}{2}(1+\xi)}e^{-\tau t^i_j}e^{-\sigma |t^i_j-t^i_{j'}|} \ \mbox{ if } l\geq 1,
\end{array}
\right.
\end{equation*}
for $j\geq 1$ where $\xi>0$ is a positive constant independent of $L$ large.
\end{lemma}

\section{Proof of the main theorem}\label{sec6}

In this section we prove our main results. We keep the notation and assumptions in the previous sections. Before we start, we define some notation:
\begin{equation*}
\tilde{\bf a}^i=(\tilde{a}_0^i, \cdots, \tilde{a}_j^i,\cdots)^t, \quad {\bf r}^i=(r_0^i,r_1^i,\cdots,r_j^i,\cdots)^t,
\end{equation*}
\begin{equation*}
T_{\tilde{a}}^i(\tilde{\bf a}^i)=T_{\tilde{a}}^i \tilde{\bf a}^i, \quad T_r^i({\bf r}^i)=T_r^i {\bf r}^i,
\end{equation*}
where
\begin{equation*}
T_{\tilde{a}}^i=\left(\begin{array}{ccccccc}
-1&1+e^{-2L_i}&-e^{-2L_i}&0&\cdots&\cdots&\vdots\\
0&-1&1+e^{-2L_i}&-e^{-2L_i}&0&\cdots&\vdots\\
0&0&-1&1+e^{-2L_i}&-e^{-2L_i}&0&\vdots\\
\cdots&\cdots&\cdots&\cdots&\cdots&\cdots&\cdots\\
\end{array}
\right)
\end{equation*}
and
\begin{equation*}
T_r^i=\left(\begin{array}{ccccccc}
-1&2&-1&0&\cdots&\cdots&\vdots\\
0&-1&2&-1&0&\cdots&\vdots\\
0&0&-1&2&-1&0&\vdots\\
\cdots&\cdots&\cdots&\cdots&\cdots&\cdots&\cdots\\
\end{array}
\right).
\end{equation*}
For $\tau>0$, let us also introduce the weighted norm and space
\begin{equation*}
|(x_j)|_\tau=\sup_j e^{(2j+1)\tau}|x_j|
\end{equation*}
and
\begin{equation*}
\ell_\tau=\{(\tilde a_j^i, r_j^i): |(\tilde a_j^i)|_\tau+|(r_j^i)|_\tau< +\infty\}.
\end{equation*}

At first glance, these infinite dimensional matrices are not invertible, since they have the trivial kernel $(1,1,\cdots)^t$, but they are indeed invertible in some suitable weighted norm, which is given in the following:
\begin{lemma}\label{lemma601}
The operators $T_{\tilde{a}}^i$, $T_r^i$ have inverse, whose norm  can be bounded by $Ce^{-2\tau}$.

\end{lemma}
\begin{proof}
Given $(f_i)_{i\geq 1} $ with $|(f_j)_j|_\tau<\infty$, our goal is to solve $T_{\tilde a}^i({\bf \tilde a}^i)=(f_i)_i$. Defining
\begin{equation*}
\tilde a_j^i=\sum_{l=j+1}^\infty\Big(\sum_{s=0}^{l-j-1}e^{-2L_is}\Big)f_l,
\end{equation*}
one can easily check that the solution $\tilde a_j^i$ satisfies the required conditions and that the operator is an inverse of $T_{\tilde a}^i$ both from the left and from the right (here the index for $f $ starts from $1$, while the index for $a$ starts from $0$). Moreover, one has
\begin{equation*}
|\tilde a_j^i|\leq C|f_j|_\tau\sum_{l=j+1}^\infty\big(\sum_{s=0}^{l-j-1}e^{-2L_is}\big)e^{-(2l+1)\tau}\leq Ce^{-(2j+3)\tau}|f_j|_\tau,
\end{equation*}
which proves the result for $T_{\tilde a}^i$. The proof for the inverse for $T_r^i$ has been given in Lemma 7.3 of \cite{m}.

The lemma is proved.

\end{proof}

Recall that in Proposition \ref{proposition301} one has found a solution $u=\bar{u}+\phi$ for
\begin{equation*}
(-\Delta)^\gamma u-c_{n,\gamma}u^\beta=\sum_{i,j,l} c_{j,l}^i(w_j^i)^{\beta-1}Z_{j,l}^i.
\end{equation*}
The solvability of the original problem \eqref{eq101} is reduced to the following system of equations:
\begin{equation*}
\beta_{j,l}^i=\int_{\R^n}[(-\Delta)^\gamma u-c_{n,\gamma}u^\beta]{Z}_{j,l}^i\,dx=0,
\end{equation*}
for all $i=1,\cdots,k$, $j=0,\cdots,+\infty$, and $l=0,\cdots, n$.

Using the above lemma and a perturbation argument, we can prove the following result:

\begin{proposition}\label{prop:beta=0}
Given $\{R^i, \hat{a}_0^i, q_i\}$ satisfying \eqref{parar} and \eqref{paraa} with $L$ sufficiently large, if we choose $\tau<\min\{\xi, \sigma \}$, there exist $(\tilde{a}_j^i)_{i,j}$ and $(r_j^i)_{i,j}$ such that \eqref{para2} holds true with $\beta_{j,l}^i=0$ for $j\geq 1$ and all $l=0,\cdots,n$, $i=1,\cdots,k$.
\end{proposition}

\begin{proof}
For $l=0$, consider
\begin{equation*}
G_0^i=\left(\begin{array}{c}
\vdots\\
\frac{1}{F(L_i)}[\beta_{j,0}^i-\bar{\beta}_{j,0}^i]\\
\vdots
\end{array}
\right)-T_r^i({\bf r}^i)
\end{equation*}
and for $l=1,\cdots,n$,
\begin{equation*}
G_l^i=\left(\begin{array}{c}
\vdots\\
\frac{e^{\frac{n-2\gamma}{2}L_i}}{\lambda_j^i}[\beta_{j,l}^i-\bar{\beta}_{j,l}^i]\\
\vdots
\end{array}
\right)-T_{\tilde{a}}^i({\tilde{\bf a}}^i),
\end{equation*}
where $\bar{\beta}_{j,l}^i$ correspond to the numbers $\beta_{j,l}^i$ for the approximate solution $\bar{u}_i^0$, i.e., the solution when $\bar{a}_j^i$, $r_j^i$ are all zero.

One can easily see that $\beta_{j,l}^i=0$ for $j\geq 1$ if
\begin{equation}\label{eq612}
\tilde{\bf a}^i=-T_{\tilde{a}}^{-1}\Big\{\left(\begin{array}{c}
\vdots\\
\frac{e^{\frac{n-2\gamma}{2}L_i}}{\lambda_j^i}\bar{\beta}_{j,l}^i\\
\vdots
\end{array}
\right)+G_{l}^i\Big\}
\end{equation}
and
\begin{equation}\label{eq613}
{\bf r}^i=-T_r^{-1}\Big\{\left(\begin{array}{c}
\vdots\\
\frac{1}{F(L_i)}\bar{\beta}_{j,0}^i\\
\vdots
\end{array}
\right)+G_{0}^i\Big\}.
\end{equation}

Next we show that the terms on the right hand sides of (\ref{eq612})-(\ref{eq613}) are contractions in an appropriate sense. First, by Lemma \ref{lemma401}, one has
\begin{equation*}
|\bar{\beta}_{j,l}^i|\leq \left\{\begin{array}{l}
Ce^{-\frac{(n-2\gamma)L_i}{2}(1+\xi)}e^{-\sigma t_j^i} \ \mbox{ if } l=0,\\
C\lambda_j^ie^{-\frac{(n-2\gamma)L_i}{2}(1+\xi)}e^{-\sigma t_j^i}\  \mbox{ if } l\geq 1,
\end{array}
\right. \mbox{ for }j\geq 1.
\end{equation*}
We write the $j-$th component of $\big[\frac{e^{\frac{n-2\gamma}{2}L_i}}{\lambda_j^i}(\beta_{j,l}^i-\bar{\beta}_{j,l}^i)\big]
-T_{\tilde{a}}^i(\tilde{\bf a}^i)$ as $G_{1,l,j}+G_{2,l,j}$, where
\begin{equation*}
G_{1,l,j}=\int_0^1 \Big[\frac{e^{\frac{n-2\gamma}{2}L_i}}{\lambda_j^i}\frac{\partial \beta_{j,l}^i(t(\bar{\bf a}^i,{\bf r}^i))}{\partial t}-\bar{A}^i\Big](\bar{\bf a}^i,{\bf r}^i)\,dt
\end{equation*}
and
\begin{equation*}
G_{2,l,j}=\bar{A}^i-T_a^i(\tilde{\bf a}^i)
\end{equation*}
for
\begin{eqnarray*}
\bar{A}^i_j(\bar{\bf a}^i,{\bf r}^i)=\sum_{j'=0}^\infty \frac{e^{\frac{n-2\gamma}{2}L}}{\lambda_{j}^i}
{\frac{\partial \beta_{j,l}}{\partial \bar{a}_{j'}^i}\cdot[\bar{a}_{j'}^i]},
\end{eqnarray*}
where $\beta_{j,l}$ is given before Lemma \ref{lemma502} and $\bar{a}_j^i$ corresponds to the translation perturbation of the $j-$th bubble in the Delaunay solution, see Lemma \ref{lemma502}. Also observe that
\begin{equation*}
T_{\tilde{a}}(\bar{\bf a}^i)=T_{\tilde{a}}(\tilde{\bf a}^i).
\end{equation*}
Let us begin by estimating $G_{1,l,j}$: using Lemma \ref{lemma504} for $l=1,\ldots,n$, one finds
\begin{equation*}\begin{split}
|G_{1,l,j}|&\leq \sum_{j'=0}^\infty  \frac{e^{\frac{n-2\gamma}{2}L}}{\lambda_{j}^i}\Big|\frac{\partial \beta_{j,l}^i}{\partial \bar{a}_{j'}^i}-\frac{\partial \beta_{j,l}}{\partial \bar{a}_{j'}^i}\Big|\,|\bar{a}_{j'}^i|+O(e^{-\frac{(n-2\gamma)L_i}{2}\xi}e^{-\min\{\sigma, \tau\} t_j^i})\\
&\leq Ce^{-\frac{(n-2\gamma)L_i}{2}\xi}\sum_{j'}e^{-\sigma t_{j'}}e^{-\sigma |t_j-t_{j'}|}|\bar{a}_{j'}^i|
+O(e^{-\frac{(n-2\gamma)L_i}{2}\xi}
e^{-\min\{\sigma, \tau\} t_j^i})\\
&\leq C(e^{-\frac{(n-2\gamma)L}{2}\xi}
e^{-\min\{\sigma, \tau\} t_j^i}).
\end{split}\end{equation*}
To estimate $G_{2,l,j}$, we apply Lemma \ref{lemma502} which gives
\begin{eqnarray*}
|G_{2,l,j}|\leq Ce^{-\frac{(n-2\gamma)L_i}{2}\xi}\Big[e^{-\sigma L_i}\big(|\tilde{a}_{j-1}^i|+|\tilde{a}_{j+1}^i|\big)
+\sum_{j'\neq j\pm 1}e^{-\sigma|t_{j'}^i-t^i_j|}|\tilde{a}_{j'}^i|\Big].
\end{eqnarray*}
Combining the above two estimates, one has for $\tau<\sigma$,
\begin{equation*}
\|G_l^i\|_{\frac{\tau L_i}{2}}\leq Ce^{\tau L}e^{-\frac{(n-2\gamma)L_i}{2}\xi}\|\tilde{\bf a}^i\|_{\frac{\tau L_i}{2}}+O(e^{-\frac{(n-2\gamma)L_i}{2}\xi}), \quad \text{for } l=1,\ldots,n.
\end{equation*}
Similarly, for $l=0$, one can get that
\begin{equation*}
\|G_0^i\|_{\frac{\tau L_i}{2}}\leq Ce^{\tau L}e^{-\frac{(n-2\gamma)L_i}{2}\xi}\|{\bf r}^i\|_{\frac{\tau L_i}{2}}+O(e^{-\frac{(n-2\gamma)L_i}{2}\xi}).
\end{equation*}

Next, with some abuse of notation, equations (\ref{eq612}) and (\ref{eq613}) are equivalent to
\begin{equation*}
\tilde{\bf a}^i=T_a^{-1}[e^{\tau L}e^{-\frac{(n-2\gamma)L_i}{2}\xi}\|\tilde{\bf a}^i\|_{\tau L_i}
+O(e^{-\frac{(n-2\gamma)L_i}{2}\xi})]=:G_{\bf a}(\tilde{\bf a}^i)
\end{equation*}
and
\begin{equation*}
{\bf r}^i=T_r^{-1}[e^{\tau L}e^{-\frac{(n-2\gamma)L_i}{2}\xi}\|{\bf r}^i\|_{\tau L_i}
+O(e^{-\frac{(n-2\gamma)L_i}{2}\xi})]=:G_{\bf r}({\bf r}^i),
\end{equation*}
where the terms on the right hand sides of the above two equations are estimated in $\|\cdot\|_{\frac{\tau L_i}{2}}$ norm.
We now consider the set
\begin{equation}\label{eq614}
\mathcal{B}=\{(\tilde{{\bf a}}_j^i, {\bf r}_j^i) \,:\, \|\tilde{{\bf a}}^i\|_{\frac{\tau L_i}{2}}+\|{\bf r}^i\|_{\frac{\tau L_i}{2}}\leq Ce^{-\tau L}\}.
\end{equation}
For $\tau<\xi$ small enough, it follows that $(G_{\bf a}, G_{\bf r})$ maps $\mathcal{B}$ into itself for $L$ large. Furthermore, it is a contraction mapping.  So by fixed point theory, there exists a  fixed point in set $\mathcal{B}$. Thus we have found  $\tilde{a}_j^i, r_j^i$ such that $\beta_{j,l}^i=0$ for all $j\geq 1$, as desired.

\end{proof}

We are now in the position to prove our existence result.\\

\noindent{\bf Proof of Theorem \ref{main-theorem}.} By Proposition \ref{prop:beta=0}, we are reduced to find $R^i, \hat{a}_0^i$ and $q_i$ for which $\beta_{0,l}^i=0$.

For $j=0$, from Lemma \ref{lemma403}, one has that equation $\beta_{0,0}^i=0$ is reduced to
\begin{equation*}\begin{split}
\beta_{0,0}^i&=-{c_{n,\gamma}}q_i\Big[A_2\sum_{i'\neq i}|p_{i'}-p_i|^{-(n-2\gamma)}(R_0^iR_0^{i'})^{\frac{n-2\gamma}{2}}q_{i'}-\Big(\frac{R_1^i}{R_0^i}\Big)^{\frac{n-2\gamma}{2}}q_i\Big] e^{-\frac{(n-2\gamma)L}{2}}(1+o(1))\\
&+O(e^{-\frac{(n-2\gamma)L}{2}(1+\xi)})=0.
\end{split}\end{equation*}
Recall that by the definition of $R_j^i$, i,e., $R_0^i=R^i(1+r_0^i)$, and the estimate for $r_j^i$ (\ref{eq614}),  then the above equation can be rewritten as
\begin{equation}\label{eq605}
A_2\sum_{i'\neq i}|p_{i'}-p_i|^{-(n-2\gamma)}(R^iR^{i'})^{\frac{n-2\gamma}{2}}q_{i'}-q_i=o(1).
\end{equation}
On the other hand, the equations $\beta_{0,l}^i=0$ for $l=1,\cdots,n$ are reduced to
\begin{equation*}\begin{split}
\beta_{0,l}^i&=c_{n,\gamma}\Big[A_3\sum_{i'\neq i}\frac{(p_{i'}-p_i)_l}{|p_{i'}-p_i|^{n-2\gamma+2}}(R_0^iR_0^{i'})^{\frac{n-2\gamma}{2}}q_{i'}\\
&+A_0\Big(\frac{R_1^i}{R_0^i}\Big)^{\frac{n-2\gamma}{2}}\frac{a_0^i-a_1^i}{(\lambda_0^i)^2} q_i \Big]q_ie^{-\frac{n-2\gamma}{2}L}+{O(e^{-\frac{(n-2\gamma)L}{2}(1+\xi) })}=0.
\end{split}\end{equation*}
By the definition of $a_j^i$, i.e. $a_j^i=(\lambda_j^i)^2\bar{a}_j^i$ and $\bar{a}_j^i=\hat{a}_0^i+\tilde{a}_j^i$, and the estimates satisfied by $\tilde{a}_j^i$ \eqref{eq614}, the above equation can be rewritten as
\begin{eqnarray}\label{eq606}
A_3\sum_{i'\neq i}\frac{(p_{i'}-p_i)_l}{|p_{i'}-p_i|^{n-2\gamma+2}}(R^iR^{i'})^{\frac{n-2\gamma}{2}}q_{i'}
+A_0\hat{a}_0^iq_i=o(1).
\end{eqnarray}

Our last step is to choose suitable $\hat{a}_0^i, R^i, q_i $ such that  equations (\ref{eq605}) and (\ref{eq606}) are solvable. Recalling the balancing conditions \eqref{balance1}-\eqref{balance2} satisfied by $\hat{a}_0^{i,b}, R^{i,b}, q_i^b$, the solvability of \eqref{eq605}-\eqref{eq606} depends on the following invertibility property of the linearized operator of the above equations around $\hat{a}_0^{i,b}, R^{i,b}, q_i^b$:
\begin{lemma}\label{lemma603}
If we denote by
\begin{equation*}
\mathcal{F}(R^i, q_i)=A_2\sum_{i'\neq i}|p_{i'}-p_i|^{-(n-2\gamma)}(R^iR^{i'})^{\frac{n-2\gamma}{2}}q_{i'}-q_i,
\end{equation*}
then the linearized operator of $\mathcal{F}$ around $( R^{i,b}, q_i^b)$ is invertible.
\end{lemma}
\begin{proof}
From the definition of $\mathcal{F}$, one has the following expression for the linearized operator
\begin{equation*}
\mathcal{F}_{R^{i}, q_i}|_{(R^{i,b},q_i^b)}: \R^{2k} \to \R^k,
\end{equation*}
where
\begin{equation*}
\mathcal{F}_{q_i}=({\bf q}_{ij})
\end{equation*}
for
\begin{equation*}
{\bf q}_{ii}=-1, \quad {\bf q}_{ij}=A_2|p_i-p_j|^{-(n-2\gamma)}(R^{i,b}R^{j,b})^{\frac{n-2\gamma}{2}}, \, i\neq j,
\end{equation*}
and
\begin{equation*}
\mathcal{F}_{R^i}=({\bf R}_{ij})
\end{equation*}
for
\begin{equation*}\begin{split}
&{\bf R}_{ii}=\frac{n-2\gamma}{2}\frac{1}{R^{i,b}}\sum_{i'\neq i}A_2|p_{i'}-p_i|^{-(n-2\gamma)}(R^{i,b}R^{i',b})^{\frac{n-2\gamma}{2}}{q_i^b},\\
&{\bf R}_{ij}=\frac{n-2\gamma}{2}\frac{1}{R^{j,b}}A_2|p_{j}-p_i|^{-(n-2\gamma)}
(R^{i,b}R^{j,b})^{\frac{n-2\gamma}{2}}{q_j^b},\, i\neq j.
\end{split}\end{equation*}
From the balancing condition \eqref{balance1} we know that
\begin{equation*}
\mathcal{F}(R^{i,b},q_i^b)=0.
\end{equation*}
One can easily see that the matrix $\mathcal{F}_{q_i}$ is symmetric and has only one-dimensional kernel, which is given by
\begin{equation*}
\operatorname{Ker}(\mathcal{F}_{q_i})=Span\{(q^b_1,\cdots, q^b_k)\}.
\end{equation*}
The balancing condition \eqref{balance1} also implies that
\begin{equation*}
\mathcal{F}_{R^i}\left(\begin{array}{c}
R^{1,b}\\
\vdots\\
R^{k,b}
\end{array}
\right)
=\frac{(n-2\gamma)}{{2}}\left(\begin{array}{c}
q^b_1\\
\vdots\\
q^b_k
\end{array}
\right).
\end{equation*}
Thus we conclude that the operator $\mathcal{F}_{R^{i,b}, q^b_i}$ is surjective.

\end{proof}

From Lemma \ref{lemma603} and the balancing condition \eqref{balance1}, one can easily find $(R^i, q_i)$ which solves (\ref{eq605}) by perturbing near $(R^{i,b}, q_i^b)$. Looking at the second balancing condition \eqref{balance2}, once $(R^i, q_i)$ are known, one can find $\hat{a}_0^i$ around $\hat{a}_0^{i,b}$ which solves \eqref{eq606}.

In conclusion, we have chosen $R^i, q_i, \hat{a}_0^i$ such that (\ref{eq605})-(\ref{eq606}) are solved, i.e. $\beta_{0,l}^i=0$. The last step in our argument is to use the maximum principle in \cite{dpgw} to show that $u>0$. This concludes the proof of the main Theorem.

\qed

%\begin{remark}
%\textcolor{red}{From the definition of the balancing condition, one can see that given $ q_i^b$, we can find corresponding solutions $R^{i,b}, \hat{a}_0^{i,b}$. This $q_i^b$ forms a $k$- dimensional parameter set, which means that we have found a $k$-dimensional solution set for problem \eqref{eq101}.}
%\end{remark}

\section{Appendix}

In this appendix we will derive some useful integrals which are important in our proof. All of the following expressions may be found in A. Bahri's book \cite{b} for the special case $\gamma=1$. Below we derive the estimates for general $\gamma$.

We define
$$w_1=\Big(\frac{\lambda_1}{\lambda_1^2+|x|^2}\Big)^{\frac{n-2\gamma}{2}},\quad w_2=\Big(\frac{\lambda_2}{\lambda_2^2+|x|^2}\Big)^{\frac{n-2\gamma}{2}},\quad w_3=\Big(\frac{\lambda_3}{\lambda_3^2+|x-p|^2}\Big)^{\frac{n-2\gamma}{2}}.$$

\begin{lemma} It holds
\begin{equation}\label{eqa01}
\beta \int_{\R^n}w_1^{\beta-1}w_2\frac{\partial w_1}{\partial \lambda_1}\,dx=\frac{1}{\lambda_1}F\big(\big|\log \tfrac{\lambda_2}{\lambda_1}\big|\big)\frac{\log \frac{\lambda_2}{\lambda_1}}{|\log\frac{\lambda_2}{\lambda_1}|},
\end{equation}
where
\begin{equation*}
F(\ell):=
\beta\int_{\R}v(t)^{\beta-1}v(t+\ell)v'(t)\,dt=e^{-\frac{n-2\gamma}{2}\ell}(1+o(1)), \quad \ell\to\infty.
\end{equation*}
\end{lemma}

\begin{proof}
By the relation between $w$ and $v$, one has
\begin{equation*}
w_1=|x|^{-\frac{n-2\gamma}{2}}v(-\log|x|+\log\lambda_1), \ w_2=|x|^{-\frac{n-2\gamma}{2}}v(-\log|x|+\log\lambda_2).
\end{equation*}
Thus
\begin{equation*}\begin{split}
\beta\int_{\R^n}w_1^{\beta-1}w_2\frac{\partial w_1}{\partial \lambda_1}\,dx&=\beta\int_{\R^n}|x|^{-2\gamma}v_1^{\beta-1}|x|^{-\frac{n-2\gamma}{2}}v_1'\frac{1}{\lambda_1}|x|^{-\frac{n-2\gamma}{2}}v_2\,dx\\
&=\beta\frac{1}{\lambda_1}\int_{\R}v^{\beta-1}(t+\log\lambda_1)v'(t+\log\lambda_1)v(t+\log\lambda_2)\,dt\\
&=\beta\frac{1}{\lambda_1}\int_{\R}v^{\beta-1}(t)v'(t)v(t+\log\tfrac{\lambda_2}{\lambda_1})\,dt\\
&=\frac{1}{\lambda_1}F(|\log\frac{\lambda_2}{\lambda_1}|)\frac{\log\frac{\lambda_2}{\lambda_1}}{|\log\frac{\lambda_2}{\lambda_1}|}.
\end{split}\end{equation*}
\end{proof}

\begin{lemma}
If  $\lambda_3=O(\lambda_1)$ then the following estimates hold:
\begin{eqnarray}
\label{eqa02}
&&\beta\int_{\R^n}w_1^{\beta-1}w_3\frac{\partial w_1}{\partial \lambda_1}\,dx=A_2\frac{|p|^{-(n-2\gamma)}}{\lambda_1}(\lambda_1\lambda_3)^{\frac{n-2\gamma}{2}}[1+O(\lambda_1)^2],\\
\label{eqa03}
&&\beta \int_{\R^n}w_1^{\beta-1}w_3\frac{\partial w_1}{\partial x_l}\,dx=A_3\frac{p_l}{|p|^{n-2\gamma+2}}(\lambda_1\lambda_3)^{\frac{n-2\gamma}{2}}(1+O(\lambda_1^2)),
\quad l=1,\ldots,n,
\end{eqnarray}
and the constants are given by
\begin{equation*}\begin{split}
&A_2=\frac{n+2\gamma}{2}\int_{\R^n}\frac{|x|^2-1}{(1+|x|^2)^{\frac{n+2\gamma+2}{2}}}\,dx>0,\\
&A_3=-\frac{(n-2\gamma)^2}{n}\int_{\R^n}\frac{|x|^2}{(1+|x|^2)^{\frac{n+2\gamma+2}{2}}}\,dx<0.
\end{split}\end{equation*}
\end{lemma}

\begin{proof}
We calculate
\begin{equation*}\begin{split}
\beta\int_{\R^n}w^{\beta-1}_1w_3\frac{\partial w_1}{\partial \lambda_1}\,dx&=\frac{n+2\gamma}{2}
\int_{\R^n}\frac{\lambda_1^{\frac{n+2\gamma}{2}-1}(|x|^2-\lambda_1^2)}
{(|x|^2+\lambda_1^2)^{\frac{n+2\gamma}{2}+1}}\Big(\frac{\lambda_3}{|x-p|^2+\lambda_3^2}\Big)
^{\frac{n-2\gamma}{2}}\,dx\\
&=\frac{n+2\gamma}{2}\lambda_1^{\frac{n-2\gamma}{2}-1}\lambda_3^{\frac{n-2\gamma}{2}}
\int_{\R^n}\frac{|x|^2-1}{(1+|x|^2)^{\frac{n+2\gamma}{2}+1}}\frac{1}{(|\lambda_1x-p|^2
+\lambda_3^2)^{\frac{n-2\gamma}{2}}}\,dx\\
&=\frac{n+2\gamma}{2}\lambda_1^{\frac{n-2\gamma}{2}-1}\lambda_3^{\frac{n-2\gamma}{2}}
|p|^{-(n-2\gamma)}\int_{\R^n}\frac{|x|^2-1}{(1+|x|^2)^{\frac{n+2\gamma}{2}+1}}\,dx\,
(1+O(\lambda_1^2)),
\end{split}\end{equation*}
where we have used the expansion
\begin{equation}\label{Taylor}
(\lambda_3^2+|\lambda_1x-p|^2)^{-\frac{n-2\gamma}{2}}=|p|^{-(n-2\gamma)}+(n-2\gamma)\frac{\lambda_1 p\cdot x}{|p|^{n-2\gamma+2}}+O(\lambda_1^2).
\end{equation}
Moreover, rescaling  $\lambda_1$ in the second step,
\begin{equation*}\begin{split}
\frac{n+2\gamma}{2}\int_{\R^n}\frac{|x|^2-1}{(1+|x|^2)^{\frac{n+2\gamma}{2}+1}}\,dx&
=\frac{\partial }{\partial \lambda_1}\Big|_{\lambda_1=1}\int_{\R^n }w_1^\beta \,dx\\
&=\frac{\partial }{\partial \lambda_1}\Big|_{\lambda_1=1}\int_{\R^n}\frac{\lambda_1^{\frac{n-2\gamma}{2}}}{(1+|x|^2)^{\frac{n+2\gamma}{2}}}\,dx\\
&=\frac{n-2\gamma}{2}\int_{\R^n}\frac{1}{(1+|x|^2)^{\frac{n+2\gamma}{2}}}\,dx>0.
\end{split}\end{equation*}
Next, by \eqref{Taylor} again,

\begin{equation*}\begin{split}
\int_{\R^n}\beta w_1^{\beta-1}w_3\frac{\partial w_1}{\partial x_l}\,dx&
=-(n-2\gamma)\int_{\R^n}\frac{\lambda_1^{\frac{n+2\gamma}{2}}x}{(\lambda_1^2+|x|^2)
^{\frac{n+2\gamma}{2}+1}}\Big(\frac{\lambda_3}{\lambda_3^2+|x-p|^2}\Big)^{\frac{n-2\gamma}{2}}\,dx\\
&=-\frac{(n-2\gamma)^2}{n}(\lambda_1\lambda_3)^{\frac{n-2\gamma}{2}}\frac{p_l}{|p|^{n-2\gamma+2}}
\int_{\R^n}\frac{|x|^2}{(1+|x|^2)^{\frac{n+2\gamma}{2}+1}}\,dx\,(1+O(\lambda_1^2)).
\end{split}\end{equation*}
\end{proof}

\begin{lemma}
For $|a|\leq \max\{\lambda_1^2,\lambda_2^2\}<<1$ and $\min\{\frac{\lambda_1}{\lambda_2}, \frac{\lambda_2}{\lambda_1}\}<<1$, the following estimates hold:
\begin{equation}\label{eqa04}\begin{split}
\int_{\R^n}&\frac{\partial }{\partial a}\Big(\frac{\lambda_1}{\lambda_1^2+|x-a|^2}\Big)^{\frac{n+2\gamma}{2}}
\Big(\frac{\lambda_2}{\lambda_2^2+|x|^2}\Big)^{\frac{n-2\gamma}{2}}dx\\
&=-A_0\min\Big\{\Big(\frac{\lambda_1}{\lambda_2}\Big)^{\frac{n-2\gamma}{2}}, \Big(\frac{\lambda_2}{\lambda_1}\Big)^{\frac{n-2\gamma}{2}}\Big\}\frac{a}{\max\{\lambda_1^2,\lambda_2^2\}}\\
&+O\Big(\Big(\frac{a}{\max\{\lambda_1,\lambda_2\}}\Big)^2
+\min\Big\{\Big(\frac{\lambda_1}{\lambda_2}\Big)^{{\frac{n-2\gamma}{2}}}, \Big(\frac{\lambda_2}{\lambda_1}\Big)^{\frac{n-2\gamma}{2}}\Big\}
\frac{a}{\max\{\lambda_1,\lambda_2\}}\Big)\\
&\cdot\min\Big\{\Big(\frac{\lambda_1}{\lambda_2}\Big)^{\frac{n-2\gamma}{2}}, \Big(\frac{\lambda_2}{\lambda_1}\Big)^{\frac{n-2\gamma}{2}}\Big\},
\end{split}\end{equation}
where
\begin{equation*}
A_0=\frac{(n+2\gamma)(n-2\gamma)}{n}\int_{\R^n}\frac{1}{|x|^{n-2\gamma}(1+|x|^2)^{\frac{n+2\gamma}{2}+1}}\,dx>0.
\end{equation*}
\end{lemma}
\begin{proof}
We consider the case $\lambda_2<<\lambda_1$.
\begin{equation*}\begin{split}
&\frac{1}{n+2\gamma}\int_{\R^n}\frac{\partial }{\partial a}\Big(\frac{\lambda_1}{\lambda_1^2+|x-a|^2}\Big)^{\frac{n+2\gamma}{2}}\Big(\frac{\lambda_2}{\lambda_2^2+|x|^2}\Big)^{\frac{n-2\gamma}{2}}\,dx\\
&\,=\int_{\R^n}\frac{\lambda_1^{\frac{n+2\gamma}{2}}(x-a)}{(\lambda_1^2+|x-a|^2)^{\frac{n+2\gamma}{2}+1}}
\Big(\frac{\lambda_2}{\lambda_2^2+|x|^2}\Big)^{\frac{n-2\gamma}{2}}dx\\
&\,=\lambda_1^{\frac{n+2\gamma}{2}}\lambda_2^{\frac{n-2\gamma}{2}}\int_{\R^n}\frac{\lambda_1 x}{\lambda_1^{n+2\gamma+2}(1+|x|^2)^{\frac{n+2\gamma}{2}+1}}\frac{1}{(\lambda_2^2+|\lambda_1 x+a|^2)^{\frac{n-2\gamma}{2}}}\lambda_1^n\,dx\\
&\,=\lambda_1^{-\frac{n-2\gamma}{2}-1}\lambda_2^{\frac{n-2\gamma}{2}}\int_{\R^n}\frac{x}{(1+|x|^2)^{\frac{n+2\gamma}{2}+1}}\frac{1}{((\frac{\lambda_2}{\lambda_1})^2+|x|^2+\frac{2a\cdot x}{\lambda_1}+|\frac{a}{\lambda_1}|^2)^{\frac{n-2\gamma}{2}}}\,dx.
\end{split}\end{equation*}
Using the assumption that $|a|\leq C\lambda_1^2$ and $\lambda_2<<\lambda_1$, by Taylor's expansion for the second term in the integral, the above integral is
\begin{equation*}\begin{split}
&-(n-2\gamma)\Big(\frac{\lambda_2}{\lambda_1}\Big)^{\frac{n-2\gamma}{2}}
\lambda_1^{-1}\int_{\R^n}\frac{\lambda_1^{-1}(a\cdot x) x|x|^{-(n-2\gamma)-2}}{(1+|x|^2)^{\frac{n+2\gamma}{2}+1}}
\Big[1+O\Big(\frac{a}{\lambda_1}\Big)^2+O\Big(\frac{\lambda_2}{\lambda_1}\Big)^2\frac{a}{\lambda_1}\Big]\,dx\\
&=-\frac{n-2\gamma}{n}\Big(\frac{\lambda_2}{\lambda_1}\Big)^{\frac{n-2\gamma}{2}}
\frac{a}{\lambda_1^2}\int_{\R^n}\frac{1}{|x|^{n-2\gamma}(1+|x|^2)^{\frac{n+2\gamma}{2}+1}}\,dx
\Big[1+O\Big(\frac{a}{\lambda_1}\Big)^2+O\Big(\frac{\lambda_2}{\lambda_1}\Big)^2\frac{a}{\lambda_1}\Big]\\
&=-\frac{A_0}{n+2\gamma}\Big(\frac{\lambda_2}{\lambda_1}\Big)^{\frac{n-2\gamma}{2}}
\frac{a}{\lambda_1^2}+O\Big(\frac{\lambda_2}{\lambda_1}\Big)^{\frac{n-2\gamma}{2}}
\Big[\Big(\frac{a}{\lambda_1}\Big)^2+\Big(\frac{\lambda_2}{\lambda_1}\Big)^2\frac{a}{\lambda_1}\Big].
\end{split}\end{equation*}
One can deal similarly with the case  $\lambda_1<<\lambda_2$; we leave this proof to the reader.
\end{proof}

\end{document}